\documentclass[11pt]{article}
\usepackage{amsmath,amssymb,amsthm,amsfonts,authblk,graphicx,enumerate,verbatim,appendix,subcaption,overpic}
\usepackage[labelfont=bf]{caption}
\usepackage[margin=1in]{geometry}
\usepackage{xcolor}
\usepackage{hyperref}
\usepackage[maxbibnames=99]{biblatex}
\renewbibmacro{in:}{}
\DeclareFieldFormat[article]{pages}{#1}

\def\R{\mathbb{R}}

\def\eps{\varepsilon}

\numberwithin{equation}{section}

\usepackage{multicol}
\usepackage{fancyhdr}

\usepackage{comment}

\newtheorem{theorem}{Theorem}[section]
\newtheorem{lemma}{Lemma}[section]
\newtheorem{corollary}{Corollary}[section] 
\newtheorem{hypothesis}{Hypothesis}[section]
\newtheorem{remark}{Remark}[section]
\newtheorem{proposition}{Proposition}[section]
\newtheorem{definition}{Definition}[section]

\title{Singular Turing bifurcations and spatial canard solutions
\\ in nonlinear reaction-diffusion systems
}
\author[1]{Robert Jencks}
\author[1]{Tasso J. Kaper}
\author[2]{Theodore Vo}

\affil[1]{\footnotesize Department of Mathematics and Statistics, Boston University, Boston, MA 02215, USA}
\affil[2]{\footnotesize School of Mathematics, Monash University, Clayton, Victoria 3800, Australia}
\date{September 7, 2026}

\begin{document}
\maketitle

\begin{abstract}
We study a general class of nonlinear reaction-diffusion equations that model pattern-forming systems. 
The class includes the Gierer-Meinhardt PDE, Brusselator model, and van der Pol PDE, as well as the Gray-Scott, Klausmeier, Lengyel-Epstein, Schnakenberg PDEs and others of activator-inhibitor type. 
In the limit in which the activator diffusivity is much smaller than that of the inhibitor, these PDEs exhibit singular Turing bifurcations, which have only recently begun to receive attention.
We analytically establish that the spatially-periodic solutions that emerge in both the sub-critical and super-critical cases of singular Turing bifurcations are spatially-periodic canard solutions.
These canard patterns are new types of spatially-periodic solutions that --just beyond the Turing point-- have fast-slow structure in space, rather than the classical sinusoidal profile.
In addition, their amplitude grows more rapidly than the classical square root growth for Turing patterns.
Indeed, even for parameter values that differ by one part in a hundred from the Turing point, they can have $\mathcal{O}(1)$ amplitude.
We also establish the existence of general spatially-dependent canard solutions with fast-slow structure.
Our analysis focuses on the spatial ODEs that govern the time-independent solutions. 
We show that these ODEs have a reversible folded saddle-node singularity of type II asymptotically close to the singular Turing point, that reversible folded saddles occur for parameters away from it, and that the true and faux canards of these folded singularities are the mechanisms responsible for creating the spatial canard solutions in the general class of PDEs.
\end{abstract}

\noindent
{\bf Dedication:} The authors dedicate this article to Peter Szmolyan for his 65$^{\rm th}$ birthday, in honour of his contributions to the theory and applications of folded singularities and multi-scale systems. 

\medskip 


\medskip 


\section{Introduction}

In classical Turing bifurcations \cite{T1952}, time-independent, spatially homogeneous solutions of systems of nonlinear reaction-diffusion equations (RDE) destabilize, and steady spatially-periodic solutions are created.
They are sinusoidal solutions with small amplitude, centered about the homogeneous state, and of the form given by $e^{ik_T x}$ to leading order, where $k_T$ is the critical Turing wavenumber.
As the value of the bifurcation parameter is taken further from the Turing point, the amplitude of these periodic solutions grows proportional to the square root of the difference between the value and the Turing point. 
Furthermore, the Turing bifurcations can either be sub-critical or super-critical, in which case the newly created periodic solutions are linearly unstable or stable, respectively.
Overall, Turing bifurcations lie at the heart of pattern formation
theory in mathematics, physics, fluid dynamics, chemistry, and biology.
A small sampling of the vast scientific literature includes
\cite{AK2002,B1978,CH1993,DES1971,EK2005,EP1998,GS1997,H2006,IMD1989,K1999,K2015,K1984,LE1992,LE1991,M1989,Mu1982,NW1969,PL1968,S2003,SS1971,W1997}.
We add that Turing bifurcations can also give
rise to hexagonal patterns, spots, stripes, and other patterns, see for example \cite{AK2002,CH1993,EP1998,K2015,LE1992,LE1991,W1997}.

Rigorous mathematical theory has been developed for the existence and stability of the spatially-periodic solutions created by Turing bifurcations.
In the system of spatial ordinary differential equations that govern time-independent solutions, Turing bifurcations correspond to 1:1 resonant Hopf bifurcations, in which the equilibrium that corresponds to the homogeneous state of the RDE system possesses two coincident pairs of pure imaginary eigenvalues.
Existence of families of spatially-periodic solutions follows --in both the sub-critical and super-critical cases-- from the normal form and infinite-dimensional center manifold theory for systems with these bifurcation points, as established in \cite{HI2011,IP1993}.

Then, as to the stability of these solutions, a large set of the spatially-periodic solutions created in the super-critical case are known to be weakly nonlinearly stable \cite{CE1990,CH1993,DES1971,E1965,E1993,JZ2010,NW1969,S1983,SS1971}.
By contrast, those created 
in sub-critical Turing bifurcations are unstable.
The stability has been demonstrated using the Ginzburg-Landau equation 
$A_\tau = D A_{\chi \chi} + M A  - \mathcal{L} \vert A \vert^2 A,$
where $A(\chi,\tau): \mathbb{R} \times \mathbb{R}^+ \to \mathbb{C}, D \in \mathbb{C}, M>0, $ and
$\mathcal{L} \in \mathbb{R}$.
It is an amplitude, or modulation, equation for the (scaled) amplitude $A$ of the solutions, subject to long wavelength instabilities of the homogeneous state (here $\chi = \varepsilon x$, $\tau=\varepsilon^2 t$, with $x,t$ being the space and time variables of the 
original reaction-diffusion system and $\varepsilon$ being linearly proportional to the difference between the parameter value and the Turing value of that bifurcation parameter).
One can see that if $\mathcal{L}>0$ then the Ginzburg-Landau PDE has constant solutions with amplitude $\vert A \vert^2 = \frac{M}{\mathcal{L}}$, which represent the time-independent, spatially-periodic solutions in the original RDE system.
Moreover, in the super-critical case, they are locally attracting due to the spectral properties of the PDE linearized about them.

The stability analysis of spatially-periodic solutions culminated in the works \cite{CE1990,JZ2010,S1996,S1998}, where the full, nonlinear stability (also known as diffusive stability) was demonstrated for large classes of 
nonlinear PDEs.
In the parameter-wavenumber space, the region in which these solutions are stable is known as the Busse balloon \cite{AACS2025,B1978,D2019,K1984,N2002,SDHR2013}.
Complete treatments of the stability analysis are given in \cite{M2002} and \cite{SU2017}, and we also refer to \cite{SZJV2018} for a complete treatment of nonlinear stability specific to the Brusselator PDE.
Further general results about the validity of the Ginzburg-Landau equation
as an amplitude or modulation equation that governs
the dynamics and stability of the spatially-periodic solutions
are given in
\cite{D2019,E1993,vH1991,MS1995,SU2017}, among other references.

Recently in \cite{JDKV2026} and \cite{VDK2025}, new types of spatially-periodic solutions known as spatially-periodic canards were discovered using numerical simulations and asymptotic analysis in the Brusselator model and van der Pol PDE, respectively.
These canard solutions emerge from singular Turing bifurcations
which occur in the limit of small activator diffusivity.
The emergent periodic solutions differ structurally from the small-amplitude, spatially-periodic solutions that emerge from classical Turing bifurcations in RDE systems.
Instead of having the sinusoidal form $e^{ik x}$ to leading order with $k$ near the critical Turing wavenumber, these canard solutions are small-amplitude, spatially-periodic solutions that --just beyond the Turing point-- have distinct fast-slow structure.
They have intervals on which they vary gradually, interspersed with short spatial intervals on which they exhibit spikes and steep gradients.
Then, for parameters further from the Turing bifurcation, they exhibit long intervals on which they vary slowly interspersed with short intervals with large-amplitude jumps/layers.
Furthermore, the amplitude of these spatially-periodic canard solutions grows much faster than the square root of the distance in parameter from the Turing point.
Indeed, numerical simulations show that they can have $\mathcal{O}(1)$ amplitude already for parameters only one part in a hundred away from the Turing point. 
Most interestingly, in the super-critical case, these new spatially-periodic canard patterns have been observed in numerical simulations to be attractors of the PDEs, while near the Turing bifurcation in the sub-critical case they have been observed to be unstable states that guide the transient evolution of patterns to canard attractors, see \cite{JDKV2026,VDK2025}.

In this article, we present rigorous analysis establishing the existence of spatially-periodic canard solutions that emerge from singular Turing bifurcations, as well as the existence of more general spatial canard solutions.
We consider the following general system of RDEs,
\begin{equation}  \label{eq:genAI-UV}
\begin{split}
    \partial_{t}U &= \delta^2 \partial^2_x U + F(U,V), \\
    \partial_{t}V &= \partial^2_x V + G(U,V).
\end{split}
\end{equation} 
Here, $t>0$, $\partial_t$ denotes the partial derivative on $t$, $x \in \mathbb{R}$, $\partial^2_x$ denotes the Laplacian in $x$ that models classical diffusion, and $0<\delta \ll 1$ is a small parameter.
The dependent variables $U=U(x,t)$ and $V=V(x,t)$ denote the concentrations of two chemical species: an activator and an inhibitor, respectively.
These have non-negative values.
With $0<\delta \ll 1$, the activator diffusivity is significantly less than that of the inhibitor, which has been set to one without loss of generality.

Motivated by the numerical and formal asymptotic results for the Brusselator and van der Pol PDEs \cite{JDKV2026,VDK2025}, we perform a rigorous analysis of spatial canards and spatially-periodic canards for general systems \eqref{eq:genAI-UV} in which the null set $\{ F(U,V) = 0 \}$ is given by the graph of a function $V=V_0(U)$ and there is a non-degenerate fold point on this graph.
Hence, we recall:
\begin{definition}
Let $U_f\ge 0$ and $V_f\ge 0$ be fixed such that the point $(U_f,V_f)$ lies on the null set of $F$, {\it i.e.}, such that $V_f=V_0(U_f)$. 
We say $(U_f,V_f)$ is a non-degenerate fold point of $F$ in \eqref{eq:genAI-UV} if 
(i) $\frac{\partial F}{\partial U}(U_f,V_f)=0$,
(ii) $\frac{\partial F}{\partial V}(U_f,V_f) \ne 0$, and 
(iii) $\frac{\partial^2 F}{\partial U^2}(U_f,V_f) \ne 0$.
\end{definition}
\noindent 
Effectively, the null set of $F$ is locally quadratic in $U$ for all $V$ in a neighborhood of $V_f$.
The following proposition (which is proven in Appendix~\ref{app:prop1}) establishes the structure of the vector field in systems \eqref{eq:genAI-UV} with a non-degenerate fold point:
\begin{proposition} \label{prop1}
   Let the null set $\{ F=0 \}$ be given by the graph of $V=V_0(U)$.
   Let $(U_f,V_f)$ be a non-degenerate fold point of \eqref{eq:genAI-UV},
   with $V_f=V_0(U_f)$, and the functions $F$ and $G$ be at least $C^4$ in a neighborhood of $(U_f,V_f)$.   
   There exists a $\delta_0>0$ sufficiently small such that, for all $0< \delta < \delta_0$, there is a smooth coordinate change $\psi: (U,V) \to (u,v)$ that, in a neighborhood of $(U_f,V_f)$,
   transforms $(U,V)=(U_f,V_f)$ to $(u,v)=(0,0)$, the system \eqref{eq:genAI-UV} to the system
\begin{equation}  \label{eq:genAI}
\begin{split}
    \partial_{t}u &= \delta^2 \partial^2_x u + f(u,v), \\
    \partial_{t}v &= \partial^2_x v + g(u,v),
\end{split}
\end{equation}
with
\begin{equation}\label{eq:fg} 
\begin{split} 
f(u,v) &= - v - u^2 - \alpha_1 u v - \alpha_2 u^3 +\mathcal{O}(v^2,u^2v,u^4),\\
    g(u,v) &= \frac{1}{2}\lambda - \alpha_3 u - \alpha_4 v - \alpha_5 u^2 + \mathcal{O}(uv,v^2,u^3),
    \end{split} 
\end{equation}
and the curve $V_0(U)$ into $v_0(u)$, where
\begin{equation}\label{eq:v0}
v_0(u) = -u^2 + (\alpha_1-\alpha_2) u^3 + \mathcal{O}(u^4), \, \, {\rm as}  \, \, u \to 0.
\end{equation}
Moreover, $\lambda\ne 0$ if $G(U_f,V_f) \ne 0$, whereas $\lambda=0$ if $G(U_f,V_f)=0$.
\end{proposition}
\noindent 
Here, $\lambda, \alpha_i \in \mathbb{R}$ for $i=1,\ldots,5$,
and the formulas for $\alpha_i$ are given in \eqref{eq:AppA-alpha_i}.
Also, $\mathcal{O}(uv)$ denotes functions that vanish at least as fast as $uv$ does as $(u,v)\to (0,0)$, for example.
In a neighborhood of the locally quadratic function $v_0(u)$, the higher order terms in $f$ are effectively quartic in $u$, and those in $g$ are at least cubic in $u$.

The key characteristics of general systems \eqref{eq:genAI} are that the null set $\{ f(u,v)=0 \}$ is locally quadratic, as given by \eqref{eq:v0}, and that the null set has a non-degenerate fold at $(u_f,v_f)=(0,0)$.
These features are common to both the Brusselator model and van der Pol PDE studied in \cite{JDKV2026} and \cite{VDK2025}, respectively.
They are also shared by the Gierer-Meinhardt (both in its classical and generalized forms) \cite{EK2005,GM-orig,Me1982,M1989,Mu1982,W1997}, Gray-Scott \cite{DGK2002,GS1997,Pearson1993,SU2017}, Hodgkin-Huxley \cite{HH1952,M1989}, 
Klausmeier \cite{K1999,SDHR2013}, 
Lengyel-Epstein \cite{LE1992,LE1991,SU2017,W1997}, 
and Schnakenberg PDEs \cite{EK2005,M1989,S1979,SU2017}.
These PDEs all have non-degenerate fold points in the reaction terms for the activator variables.
The motivation for studying systems \eqref{eq:genAI-UV}, and hence \eqref{eq:genAI}, is that it represents a general class of RDEs that includes these and other reaction-diffusion models with non-degenerate fold points. 

The first main result of this article is the existence of time-independent, spatial canard solutions in the PDEs \eqref{eq:genAI}, equivalently \eqref{eq:genAI-UV} with non-degenerate fold points.
These solutions exhibit a ``fast-slow" (or singularly-perturbed) spatial structure.
They have intervals where they vary gradually in space, interspersed with short intervals of rapid spatial variation where they exhibit spikes and steep gradients.

The second main result is the existence of spatially-periodic canard solutions of \eqref{eq:genAI}. 
We show that linearly unstable canards are created in sub-critical Turing bifurcations and linearly stable canards are created in super-critical Turing bifurcations. 
In both cases, these canards are purely sinusoidal ($e^{ik_Tx}$) to leading order at the Turing bifurcations.
However, in general, these spatially-periodic solutions exhibit ``fast-slow" spatial structure, even as soon as one takes a value of the parameter that is slightly away from the Turing bifurcation, and their amplitude grows more rapidly than the classical square root growth as the parameter is taken further from the Turing point.
These results will demonstrate rigorously that spatially-periodic canard solutions are created not only in subcritical Turing bifurcations, as observed numerically for the Brusselator and van der Pol PDEs \cite{JDKV2026,VDK2025}, but also in the supercritical case in this general class \eqref{eq:genAI-UV}.
In the supercritical case, the analysis reveals that these time-independent attractors vary gradually over most of the period, and they
exhibit spikes in the activator component on short, complementary intervals.
We illustrate the supercritical Turing bifurcations to stable spatially-periodic canards in the classical Gierer-Meinhardt PDE \cite{EK2005,GM-orig,M1989,W1997}.

We have labeled the general spatial canard solutions and the spatially-periodic canard solutions of the PDEs \eqref{eq:genAI-UV} as {\it spatial canards}, because they are created by the canards of folded singularities in the spatial dynamics problem and because the spikes and layers in these solutions are nucleated by canards.
Here, we recall that folded singularities are points that lie on fold curves in the phase spaces of singularly-perturbed ODEs.
First studied and classified by Takens \cite{T1976}, they are points through which solutions can pass from one slow invariant manifold to another.
Folded singularities, see for example \cite{B1983,BKW2006,DR1996,KS2001,KW2010,MW2017,RRW2015,RKW2008,SW2001,VW2015}, have been studied primarily in the context of temporal fast-slow ODEs,
where the temporal canard solutions \cite{BCDD1981,D1984,DR1996,E1983,Moehlis,RKZE2003} 
pass from an attracting slow manifold to an unstable one, or vice versa.
They are classified as folded nodes, folded saddles, folded saddle-nodes, folded centers, among others, based on the eigenvalues of the linearization of the associated desingularized system.
Many of these have canard solutions that pass through them.
For instance, folded saddles have true and faux canards
that pass through them from one branch of the critical manifold to the other.
Also, folded nodes have primary and secondary strong canards, as well as weak canards.
Here, in the spatial ODE problems associated to the PDEs \eqref{eq:genAI-UV}, we find folded saddles (FS) and folded saddle-nodes of type II (FSN-II) \cite{KW2010,MW2017,SW2001}.
We will see that solutions pass through these points from saddle slow invariant manifolds to elliptic slow invariant manifolds, and vice versa, which is also a novel feature of the singularly-perturbed spatial dynamics problem here.

The proofs consist primarily of analyzing the ODEs that govern the time-independent solutions of \eqref{eq:genAI}.
These spatial ODEs are obtained by setting $\partial_t u=0$ and $\partial_t v=0$ in \eqref{eq:genAI},
\begin{equation}  \label{eq:x-ODE}
\begin{split}
    \delta u_x &= p \\
    \delta p_x &= v + u^2 + \alpha_1 uv + \alpha_2 u^3 + \mathcal{O}(v^2, u^2v, u^4),  \\
    v_x &= q \\
    q_x &= -\frac{\lambda}{2} + \alpha_3 u + \alpha_4 v + \alpha_5 u^2 + \mathcal{O}(uv,v^2,u^3).
\end{split}
\end{equation}
Here, $x$ is the independent variable, and the subscript on the dependent variables denotes the total derivative with respect to $x$.
This formulation will be useful for the analysis of the ``slow" spatial dynamics, 
{\it i.e.,} of the long intervals on which solutions vary gradually in space.

Let $y=x/\delta$.
For all $\delta>0$, an equivalent formulation of this ODE system is
\begin{equation}  \label{eq:y-ODE}
\begin{split}
    u_y &= p \\
    p_y &= v + u^2 + \alpha_1 uv + \alpha_2 u^3 + \mathcal{O}(v^2, u^2v, u^4),  \\
    v_y &= \delta q \\
    q_y &= \delta \left( -\frac{\lambda}{2} + \alpha_3 u + \alpha_4 v + \alpha_5 u^2 + \mathcal{O}(uv,v^2,u^3)\right).
\end{split}
\end{equation}
This latter formulation will be useful for the analysis of the ``fast" spatial dynamics, {\it i.e.,} of the short intervals on which solutions have large spatial gradients, and hence vary rapidly in space.
Note that the systems \eqref{eq:x-ODE} and \eqref{eq:y-ODE} have a reversibility symmetry 
\begin{equation}\label{eq:reversible}
{\mathcal R}: \, (x,u,p,v,q) \to (-x,u,-p,v,-q),
\end{equation}
which they inherit from the fact that the PDEs \eqref{eq:genAI-UV}
are invariant under the transformation $x \to -x$.

We make the following definitions: 

\begin{definition}
Let $\delta>0$ be sufficiently small.
A time-independent solution $(U(x),V(x))$ of \eqref{eq:genAI-UV} is called
a {\rm spatial canard solution} if $\left(U, \frac{dU}{dx}, V, \frac{dV}{dx} \right)$ is a canard solution
of a folded singularity on the fold curve in the spatial ODE system \eqref{eq:y-ODE}.
\end{definition}
\begin{definition}
A spatial canard solution is called a maximal spatial canard solution if it is a maximal canard solution 
of a folded singularity on the fold curve
in the spatial ODE system \eqref{eq:y-ODE}.
\end{definition}
The first main result is 
\begin{theorem}\label{thm:1}
    Let the PDE \eqref{eq:genAI-UV} have a non-degenerate fold point.
    There exists a $\delta_0>0$ small such that, for $0<\delta<\delta_0$ and for each $\lambda \ge 0$, the PDE \eqref{eq:genAI-UV} has spatial canard solutions and maximal spatial canards.
\end{theorem}
We will show that the spatial canards and the maximal spatial canards of the PDE \eqref{eq:genAI-UV} are exactly the canards and maximal canards of the FS and FSN-II points in the spatial ODEs \eqref{eq:y-ODE}.
This will follow from the fact that the latter are time-independent.
Hence, the existence result will follow from the analysis of the folded singularities and their canards in the phase space of \eqref{eq:y-ODE}.

\medskip

Next, we present results for the main types of spatially-periodic canards that arise in pattern formation through singular Turing bifurcations in \eqref{eq:genAI-UV}, equivalently \eqref{eq:genAI}.
We show that the Turing bifurcation points in \eqref{eq:genAI} occur for parameter values that are asymptotically close to the parameter value $\lambda=0$ at which the spatial ODE system \eqref{eq:y-ODE} has a reversible FSN-II point (RFSN-II).
This was one of the central observations made about the Brusselator and van der Pol PDEs, recall \cite{JDKV2026,VDK2025}.
Here, we find that this property holds quite generally.
We make the following definition:
\begin{definition}
A spatial canard solution $(U(x),V(x))$ of \eqref{eq:genAI-UV} is called
a {\rm spatially-periodic canard solution} if there exists an $X>0$ 
such that $(U(x+X),V(x+X))=(U(x),V(x))$ for all $x$.
\end{definition}

The second main result, 
which we prove by constructing spatially-periodic canards in \eqref{eq:y-ODE}, is
\begin{theorem}\label{thm:2}
Let the PDE \eqref{eq:genAI-UV} have a non-degenerate fold point $(U_f,V_f)$ 
and a parametrized family of equilibria $(U_e,V_e)$.
Let the parameter-dependent family of equilibria undergo either a subcritical or a supercritical Turing bifurcation as the equilibrium point passes through the fold point with non-zero speed as the parameter is varied. 
There exists a $\delta_0>0$ sufficiently small such that, for $0<\delta<\delta_0$ and for each $\lambda\ge 0$, the PDE \eqref{eq:genAI-UV} has spatially-periodic canard solutions.
Furthermore, these are members of the families of periodic solutions created in the Turing bifurcation.     
\end{theorem}

Geometrically, the spatial canards and spatially-periodic canards exist in the PDEs \eqref{eq:genAI-UV} because the null set $\{ F=0 \}$ of the activator is a non-degenerate quadratic in the neighborhood of the non-degenerate fold point $(U_f,V_f)$.
Hence, the associated system of spatial ODEs \eqref{eq:y-ODE} has a critical manifold with saddle and center sheets separated by a fold curve, and we will see that this geometric structure naturally gives rise to the folded singularities and their attendant canards in the spatial dynamics problem \eqref{eq:y-ODE}.
Moreover, this type of null set is ubiquitous for models of chemical oscillators, where the temporal kinetics ODEs naturally have stable (attracting) and unstable (repelling) branches that meet at fold points.
Therefore, the spatial canards whose existence we establish here are ubiquitous in the PDE models that govern these types of chemical oscillators when the species also diffuse.

We also observe that the spatially-periodic canards in PDEs \eqref{eq:genAI-UV} are the analogs in spatial dynamics of the temporally periodic limit cycle canards first discovered in the fast-slow (or ``relaxation" limit of the) van der Pol ODE \cite{BCDD1981,D1984,DR1996,E1983}, as well as of the canards of folded saddles and folded saddle-nodes in other fast-slow (temporal) systems, see {\it e.g.} \cite{KS2001,KW2010,MKKR1984,MW2017,Moehlis,RKZE2003}. 
First, the singular Turing bifurcation points in the PDEs \eqref{eq:genAI} correspond to 1:1 resonant Hopf bifurcations in the spatial ODE system \eqref{eq:y-ODE} (see Section~\ref{sec:Turing+NF}).
These are the spatial analogs of the singular Hopf bifurcations in temporal fast-slow ODEs.
Second, the families of small-amplitude, spatially-periodic solutions created in the singular Turing bifurcations here are the analogs in spatial dynamics of the small-ampltiude, temporally oscillating solutions in fast-slow ODEs that exist just beyond singular Hopf bifurcations. 
Third, the critical value at which the maximal spatial canards exist is the analog in spatial dynamics of the critical canard value in fast-slow temporal ODEs. 
Fourth, the spatial canards generated by reversible folded saddles for each $\lambda>0$ and by the reversible folded saddle-nodes of type II at $\lambda=0$ are the analogs in spatial dynamics of the canards created by folded singularities in the temporal fast-slow ODEs.
Fifth, the segments of slow drift (in $x$) of the spatial canards near the true and faux canards $\Gamma_t^\delta$ and $\Gamma_f^\delta$ of the folded saddle and folded saddle-node of type II points are the analogs of the segments of slow drift (in time) along the attracting and repelling branches of the one-dimensional critical manifolds for systems with temporal canards.

There are important structural differences between the spatial canards here and temporal canards in fast-slow ODEs, due to the geometry of the four-dimensional phase space of the spatial ODEs \eqref{eq:y-ODE}.
The equilibria of the fast system can only be saddles, centers, and saddle-nodes, not attractors or repellors, since the fast subsystem of \eqref{eq:y-ODE} is Hamiltonian.
Hence, the spatial canards of \eqref{eq:y-ODE} consist of segments of slow (spatial) drift near two-dimensional saddle slow manifolds in alternation with segments of (spatial) drift near two-dimensional elliptic manifolds during which the solutions oscillate about center points of the fast system.
Another difference is that, for spatial canards, there is a unique parameter value at which the single-branched stable and unstable manifolds of the cusp point, which are the true and faux canards of the RFSN-II point, continue into each other to all orders in the small parameter $\delta$.
By contrast, for temporal canards, the critical value corresponds to the unique value at which the branches of one-dimensional attracting and repelling slow manifolds coincide to all orders.

This article is organized as follows.
In Section~\ref{sec:Turing+NF}, we apply the classical Turing bifurcation analysis
and the normal form theory for 1:1 resonant Hopf bifurcations in a straightforward manner to determine the conditions for which the Turing bifurcations are sub- and super-critical.
Next, in Section~\ref{sec:layerproblem+criticalmanifold+desingularizedsystem}, we analyze the key structures of the spatial ODE system \eqref{eq:y-ODE}, including the layer problem, the critical and slow manifolds, and the desingularized reduced system. 
Also, we identify the two key folded singularities: the reversible folded saddle-node of type II that exists for $\lambda=0$ and the reversible folded saddle point that exists for each $\lambda>0$.
The geometric desingularization analysis of these reversible folded singularities is presented in Sections~\ref{sec:geodesing} and \ref{sec:RFSNII-desing}.
Then, Sections~\ref{sec:proof-theorem1} and \ref{sec:proof-theorem2} contain the proofs of Theorems~\ref{thm:1} and \ref{thm:2}.
The classical Gierer-Meinhardt PDE is analyzed in Section~\ref{sec:GM-analysis}, as a prototype of the general class of PDEs \eqref{eq:genAI-UV}, and results obtained from numerical simulations of the classical Gierer-Meinhardt PDE are included to illustrate
the stable spatially-periodic canards whose existence is established by Theorem~\ref{thm:2}.
Appendices~\ref{app:prop1}, \ref{app:Turing}, \ref{app:proptwist}, and \ref{app:secondgeodesing} present the proofs of Propositions~\ref{prop1}, \ref{prop:haragusiooss}, \ref{prop:twist}, and \ref{prop:roleofGamma20}, respectively.
Appendix~\ref{app:GM-Turing} contains a brief derivation of the formula for the Turing bifurcation in the Gierer-Meinhardt PDE.
Appendix~\ref{app:examples} briefly presents two additional prototypical PDEs of the form \eqref{eq:genAI-UV} with non-degenerate fold points, which further illustrate the results of Theorem~\ref{thm:2}.

\section{The Turing Bifurcation and the Normal Form of the Spatial ODE for Reversible 1:1 Resonant Hopf Bifurcations} \label{sec:Turing+NF}

In this section, we embed the general system \eqref{eq:genAI} into the following one-parameter family of PDEs: 
\begin{equation}  \label{eq:genAI-mu}
\begin{split}
    \partial_t u &= \delta^2 \partial_x^2 u + f(u,v;\mu), \\
    \partial_t v &= \partial_x^2 v + g(u,v;\mu),
    \end{split}
\end{equation} 
where $\mu$ is a real-valued parameter.
We study parameter values at which a spatially homogeneous state undergoes a Turing bifurcation.
Our approach will rely on the fact that Turing points correspond to 1:1 resonant Hopf bifurcation points in the associated system of spatial ODEs that governs the time-independent solutions, see for example \cite{E1965,HI2011,IMD1989,IP1993}.

In particular, for \eqref{eq:genAI-mu}, the spatial ODE system is
\begin{equation}  \label{eq:y-ODE-mu}
\begin{split}
    u_y &= p \\
    p_y &= - f(u,v;\mu) \\
    v_y &= \delta q \\
    q_y &= - \delta g(u,v;\mu),
\end{split}
\end{equation}
where again $y=x/\delta$.
Following \cite{HI2011} in a straightforward manner, we will apply the normal form bifurcation analysis of 1:1 resonant Hopf points of \eqref{eq:y-ODE-mu} to identify the Turing points and to determine how their criticality depends on parameter values.

\subsection{Conditions for the Turing bifurcation in \texorpdfstring{\eqref{eq:genAI-mu}}{Lg}}
\label{sec:Turing}

We make the following assumption about the general system \eqref{eq:genAI-mu}:
\begin{hypothesis}\label{familyofequilib}
There exists an interval $I$ such that for each $\mu\in I$ the spatial  
ODE system \eqref{eq:y-ODE-mu} of the PDE \eqref{eq:genAI} has an equilibrium at 
\begin{equation}\label{eq:equilib}
(u,p,v,q)=(u_e(\mu),0,v_e(\mu),0).
\end{equation}
\end{hypothesis}
At an equilibrium, the Jacobian matrix of \eqref{eq:y-ODE-mu} is
\begin{equation*}
    J = \left[
     \begin{array}{cccc}
     0 & 1 & 0 & 0 \\
     -f_u & 0 & -f_v & 0 \\
     0 & 0 & 0 & \delta \\
     -\delta g_u & 0 & -\delta g_v & 0 
     \end{array}
    \right].
\end{equation*}
The characteristic equation is the following quadratic function of $\sigma^2$:
\begin{equation}\label{eq:detJ-sigmaI}
    {\rm det} ( J - \sigma \mathbb{I})
    = \sigma^4 + (f_u + \delta^2 g_v) \sigma^2 + \delta^2 ( f_u g_v - f_v g_u) = 0.
\end{equation}
Hence, the eigenvalues of $J$ are
\begin{equation}\label{eq:sigma}
    \sigma= \pm \frac{1}{\sqrt{2}} 
    \sqrt{ - (f_u + \delta^2 g_v) \pm \sqrt{ (f_u+\delta^2 g_v)^2 - 4 \delta^2(f_u g_v - f_v g_u)} },
\end{equation}
where the derivatives are evaluated at the equilibria.

The equilibrium is a 1:1 resonant Hopf bifurcation point at the parameter value where the eigenvalues form two coincident pure imaginary pairs.
To find these for \eqref{eq:y-ODE-mu}, we set 
$\det (J-\sigma \mathbb{I}) = 0$ and $\frac{d (\det (J-\sigma \mathbb{I}))}{d\sigma^2} = 0$,
\begin{equation}\label{eq:discriminant} 
    \begin{split}
        &(f_u+\delta^2 g_v)^2 - 4 \delta^2(f_u g_v - f_v g_u) = 0, \\
        &\sigma_T^2 = - \frac{1}{2}(f_u + \delta^2 g_v), \, \, \, {\rm with} \, \, \, f_u + \delta^2 g_v > 0.
    \end{split}
\end{equation}
Solving the first condition, we find
\begin{equation*}
    f_u = \pm 2 \delta \sqrt{-f_v g_u} + \delta^2 g_v.
\end{equation*}
We focus on the first case (positive root); results for the second case (negative root) follow in a similar manner.
Hence, we make the following additional assumption,
\begin{hypothesis}\label{hypo2}
    For all $\mu\in I$, let $f_v(u_e(\mu),0,v_e(\mu),0) g_u(u_e(\mu),0,v_e(\mu),0) < 0$.
    There exists a locally unique value of $\mu$, which we label $\mu_T$, in $I$ such that
\begin{equation}\label{eq:Turing-condition}
    f_u = 2 \delta \sqrt{-f_v g_u} + \delta^2 g_v \qquad {\rm and} \qquad 
    \sqrt{-f_v g_u} + \delta g_v > 0,
\end{equation}
where the derivatives are evaluated at the equilibria.
\end{hypothesis}

\begin{figure}[ht!]
  \centering
  \includegraphics[width=5in]{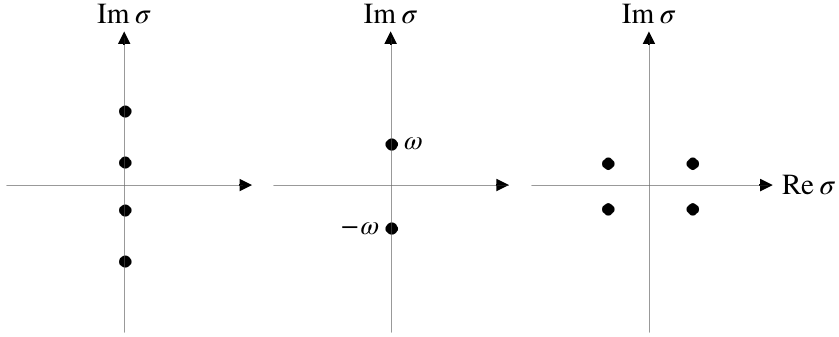}
  \put(-360,135){(a)}
  \put(-244,135){(b)}
  \put(-128,135){(c)}
  \caption{The quartet of eigenvalues of the spatial ODEs \eqref{eq:y-ODE-mu} as a parameter sweeps the system through the 1:1 resonant Hopf bifurcation at $\mu_T$. 
  (a) On one side of bifurcation, there are two distinct pairs of purely imaginary eigenvalues $(\pm i\omega, \pm i \omega)$. 
  (b) At the bifurcation, there is a pure imaginary pair of eigenvalues with algebraic multiplicity two. 
  (c) On the other side, the eigenvalues consist of a complex pair with negative real part and a complex pair with positive real part.}
  \label{fig:eigenvalues}
\end{figure}

With the above hypotheses, the system \eqref{eq:y-ODE-mu} of spatial ODEs has a 1:1 resonant Hopf bifurcation at $\mu = \mu_T$, with the Jacobian having eigenvalues $(\pm i \omega, \pm i \omega)$, where 
\begin{equation}\label{eq:omega-NF}
    \pm i \omega = \pm i \sqrt{\delta} \omega_0 \quad {\rm and} \quad \omega_0 = \sqrt{\sqrt{-f_vg_u} + \delta g_v}.
\end{equation}
Hence, the homogeneous state of the system of PDEs \eqref{eq:genAI-mu} undergoes a Turing bifurcation there.
A sketch of the eigenvalues $\sigma$ under variation of the parameter $\mu$ is shown in Figure~\ref{fig:eigenvalues}.
On one side of $\mu_T$, the Jacobian has two pairs of pure imaginary eigenvalues so that the fixed point is a center-center point (Figure~\ref{fig:eigenvalues}(a)), and on the other it has a quartet of complex-valued eigenvalues with non-zero real parts (Figure~\ref{fig:eigenvalues}(c)) so that the state is normally hyperbolic.
The Turing bifurcation at $\mu_T$ occurs at the edge, where normal hyperbolicity of the state is lost.

Applying the theory to PDEs \eqref{eq:genAI} with $\alpha_3 < 0$, we find that the Turing bifurcation occurs at 
\begin{equation}\label{eq:lambdaT}
\lambda_T = 2 (-\alpha_3)^{3/2} \delta
\end{equation}
to leading order in $\delta$.
This follows from the first condition in \eqref{eq:Turing-condition} in the limit that $u_e$ and $\lambda$ are small, where $u_e=\frac{\lambda}{2\alpha_3}$, to leading order, where the derivatives in the definition 
\eqref{eq:AppA-alpha_i} of $\alpha_3$ are evaluated at the Turing point.
(Note: for systems in which $\alpha_3=0$, $\lambda_T$ is of higher order.)

\subsection{Normal form theory for 1:1 resonant Hopf bifurcations}\label{sec:oneoneNF}

System \eqref{eq:y-ODE-mu} has a reversibility symmetry that is induced by the $x \to -x$ symmetry of \eqref{eq:genAI}. 
Let
\begin{align}
    \mathcal{R} = 
    \begin{bmatrix}
        1 & 0 & 0 & 0 \\
        0 & -1 & 0 & 0 \\ 
        0 & 0 & 1 & 0 \\
        0 & 0 & 0 & -1
    \end{bmatrix}, \ \ \ \
    {\bf u} =
    \begin{bmatrix}
        u \\ p \\ v \\ q
    \end{bmatrix}, \ \ \ 
    {\bf F} = {\bf F} (u,p,v,q;\mu) = 
    \begin{bmatrix}
          p \\
          -f(u,v;\mu)\\
          \delta q\\
          - \delta g(u,v;\mu) 
    \end{bmatrix}.
\end{align}
Here, $\mathcal{R}$ is a reversibility symmetry because it anti-commutes with the vector field,
\begin{equation}\label{reversibility}
    \mathcal{R} {\bf F} ({\bf u}) = - {\bf F}( \mathcal{R}{\bf u}).
\end{equation}
Also, let $\mathcal{L}$ denote the operator obtained by linearizing the vector field ${\bf F}$ about the equilibria \eqref{eq:equilib}.
As a consequence of the reversibility symmetry, the spectrum of $\mathcal{L}$ is symmetric with respect to both the real-axis and the imaginary-axis in the $\sigma$ plane, since ${\bf u}$ is real-valued and since we have $\mathcal{R}\mathcal{L} = - \mathcal{L} \mathcal{R}$ and $(\sigma I + \mathcal{L} )^{-1} \mathcal{R} = \mathcal{R} (\sigma I - \mathcal{L} )^{-1}$.

The normal form theory for reversible, 1:1 resonant (spatial) Hopf bifurcations guarantees the existence of families of (spatially) periodic solutions, (spatially) homoclinic solutions, and other solutions of the spatial ODE system. 
Application of Theorem 3.21 in Chapter 4.3.3 of \cite{HI2011} yields

\medskip
\begin{proposition} \label{prop:haragusiooss}
For the spatial ODE system \eqref{eq:y-ODE-mu} with $\delta>0$ and $\tfrac{\mu-{\mu}_T}{\mu_T}$ sufficiently small, the following statements hold:
\begin{enumerate}[(i)]
\item For all $\frac{\mu-\mu_T}{\mu_T}\ne 0$ and small and for all $b<0$, there is a symmetric equilibrium, a one-parameter family of periodic orbits, and a two-parameter family of quasi-periodic orbits located on KAM tori.
\item For $\frac{\mu-\mu_T}{\mu_T}>0$ small and $b>0$, there is a symmetric equilibrium, a one-parameter family of periodic orbits, and a two-parameter family of quasi-periodic orbits located on KAM tori. 
There is also a one-parameter family of reversible homoclinic orbits to the periodic orbits.
\item For $\frac{\mu-\mu_T}{\mu_T}<0$ small and $b<0$, 
there is a pair of reversible homoclinic orbits to the symmetric equilibrium.
\item For $\frac{\mu-\mu_T}{\mu_T}<0$ small and $b>0$, a symmetric equilibrium exists, but no other bounded solutions.
\end{enumerate}
\end{proposition}

\noindent
The proof of this proposition is presented in Appendix~\ref{app:Turing}.

\medskip

The key coefficient (on the cubic term) in the normal form is given asymptotically as
\begin{equation}\label{eq:b-exp}
b = b_0 + \delta b_1 + \mathcal{O}(\delta^2),
\end{equation}
\begin{equation}
\begin{split} 
b_0 &= \frac{2 f_v^2 f_{uu}^2 \omega_0^4}{f_v g_u - 8 \omega_0^4}, \hskip0.3truein 
b_1 = \frac{1}{4 g_u (f_vg_u - 8 \omega_0^4)^2} 
\left( B_2 \omega_0^2 + B_6 \omega_0^6 + B_{10} \omega_0^{10} \right), \\
B_2 &= 3 f_v^3 g_u^2 ( f_{uv}f_{uu}g_u - 2 f_{uu}^2 g_v + 2 f_v f_{uu} g_{uu} - f_v f_{uuu} g_u )\\
B_6 &=  4 f_v^2 g_u ( 26 f_{uu}^2 g_v - 13 f_{uv}f_{uu} g_u + 12 f_v f_{uuu} g_u -22 f_v f_{uu} g_{uu}) \\
B_{10} &= 32 f_v ( 10 f_v f_{uu} g_{uu} -10 f_{uu}^2 g_v + 7 f_{uv}f_{uu}g_u - 6 f_v f_{uuu} g_u).
\end{split}
\end{equation}
See \eqref{eq:def-b} in Appendix~\ref{app:Turing} for the exact formula for $b$.
The Turing bifurcation is subcritical for $b<0$ and supercritical for $b>0$.

\begin{corollary}\label{eq:corllary2.1}
The conclusions of Proposition~\ref{prop:haragusiooss} hold for the PDEs \eqref{eq:genAI}, with the parameter $\lambda$
and the value $\lambda_T$ \eqref{eq:lambdaT} playing 
the roles of $\mu$ and $\mu_T$.
\end{corollary}

\section{The Layer Problem, Critical and Slow Manifolds, Desingularized Reduced System, and Folded Singularities of \texorpdfstring{\eqref{eq:y-ODE}}{Lg}}
\label{sec:layerproblem+criticalmanifold+desingularizedsystem}

In this section, we study the layer problem of \eqref{eq:y-ODE}, which governs the fast spatial dynamics.
Also, we identify the critical manifold, derive the desingularized reduced slow flow on it, and identify the key folded singularities responsible for generating the canard solutions of \eqref{eq:y-ODE}.

\subsection{Layer problem}\label{sec:layerprob}

We begin by setting $\delta=0$ in \eqref{eq:y-ODE}.
This yields the layer system (or reduced fast system),
\begin{equation} \label{layer}
\begin{split}
    u_y &= p \\
    p_y &= v + u^2 + \alpha_1 uv + \alpha_2 u^3 + \mathcal{O}(v^2,u^2v,u^4),
\end{split}
\end{equation}
in which $v$ and $q$ are constants.
It is a planar Hamiltonian system, with
\begin{equation}\label{H_f}
H_f (u,p;v) = \frac{1}{2} p^2 - uv - \frac{1}{3} u^3 
-\frac{\alpha_1}{2} u^2v - \frac{\alpha_2}{4} u^4  
+ \mathcal{O}(uv^2, u^3v, u^5).
\end{equation}

The layer problem has a curve of equilibria given by
\begin{equation}
   \mathcal{C} = \{ (u,p;v): u\in \mathbb R, \ 
   p=0, \  v = v_0(u) = - u^2 + (\alpha_1 - \alpha_2) u^3 + \mathcal{O}(u^4) \},
\end{equation}
see also \eqref{eq:v0}.
Linearizing the layer problem at points on $\mathcal{C}$, we find the  Jacobian 
\begin{equation*}
    J \vert_{\mathcal{C}} = 
    \left[ 
    \begin{array}{cc}
      0 & 1 \\
      2u + (3\alpha_2 - \alpha_1) u^2 +\mathcal{O}(u^3)& 0
    \end{array} 
    \right].
\end{equation*}
The trace of $J\vert_{\mathcal{C}}$ is zero, and its determinant is 
\begin{equation}\label{eq:detJ}
{\rm det} J\vert_{\mathcal{C}} = -2u - (3\alpha_2 - \alpha_1) u^2 + \mathcal{O}(u^3).
\end{equation}
Hence, in a neighborhood of the origin, the equilibria on $\mathcal{C}$ are saddle fixed points of \eqref{layer} for $u>0$, 
and they are centers for $u<0$.

\subsection{The critical manifold and its saddle and center sheets}

In the full $(u,p,v,q)$ space of \eqref{eq:y-ODE}, the equilibria on $\mathcal{C}$ form the two-dimensional critical manifold 
\begin{equation} \label{criticalmanifold}
    S^0 = \left\{ (u,p,v,q): 
        u \in \mathbb{R}, \, \, 
          p = 0, \, \, v = v_0(u) = -u^2 +(\alpha_1-\alpha_2) u^3 + \mathcal{O}(u^4), \, \, q \in \mathbb{R} 
        \right\} \, .
\end{equation}
It consists of two sheets that are separated by a fold line,  
\begin{equation}\label{eq:union}
S^0  = S^0_s \cup  L \cup S^0_c,
\end{equation}
\begin{equation*} 
\begin{split}
    S^0_s &= \left\{
          u >0, \, p = 0, \, v=v_0(u), \, q \in \mathbb{R}
        \right\}, \\
    L &= \{ u = 0, \, p=0, \, v=0, \, u=0, \, q \in \mathbb{R} \}, \\ 
    S^0_c &= \left\{
          u < 0, \, p = 0, \, v=v_0(u), \, q \in \mathbb{R}
        \right\}.
\end{split}
\end{equation*}
Here, $S_s^0$ is the saddle sheet and $S_c^0$ the center sheet, due to the stability types of the equilibria on $\mathcal{C}$, with the subscripts $s$ and $c$ denoting saddle and center. 
See Figure~\ref{fig:criticalmanifold} for a sketch of $S^0$.

\begin{figure}[ht!]
   \centering
   \includegraphics[width=5in]{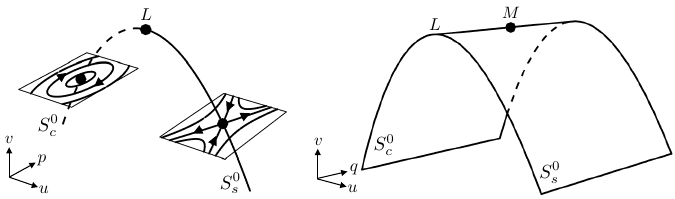}
   \put(-360,94){(a)}
   \put(-188,94){(b)}
   \caption{Projection of the critical manifold $S^0$ into the (a) $(u,p,v)$ space and (b) $(u,q,v)$ space. The hyperbolic saddle subset $S_s^0$ is separated from the elliptic center subset $S_c^0$ by the fold set $L$. 
   The analysis in Section~\ref{sec:MRFS+MRFSNII}
   shows that there is a distinguished point, $M \in L$, where the projection of the reduced flow onto the slow $(v,q)$ plane is tangent to $L$ at $M$.}
   \label{fig:criticalmanifold}
\end{figure}

We are particularly interested in the saddle sheet $S^0_s$, because away from the fold line $L$ it is a normally hyperbolic invariant manifold for $\delta=0$.
We examine its geometry in the four-dimensional phase space of \eqref{eq:y-ODE} with $\delta=0$.
Let $u_s(v_0)$ denote the solution of $v_0=-u^2 + (\alpha_1 - \alpha_2) u^3 +\mathcal{O}(u^4)$ that represents $S^0_s$.
Over each point $(u_s(v_0),0,v_0,q_0) \in S^0_s$, there exist one-dimensional fast stable and unstable fibers. 
These are given by the local stable and unstable manifolds $W^{s,u}_{\rm loc}(u_s(v_0),0)$ of the saddle equilibrium of \eqref{layer}.
The unions of these one-dimensional fibers over all points on $S^0_s$ form the three-dimensional local stable and unstable manifolds of $S^0_s$:
\begin{eqnarray*}
    W^s(S^0_s) &=& \bigcup_{(u_s(v_0),0,v_0,q_0) \in S^0_s} W^s_{\rm loc}(u_s(v_0),0) \\ 
     W^u(S^0_s) &=& \bigcup_{(u_s(v_0),0,v_0,q_0) \in S^0_s} W^u_{\rm loc}(u_s(v_0),0).
\end{eqnarray*} 

Now, fix a small $\Delta>0$.
Since $S^0_s$ is normally hyperbolic for $u< - \Delta$, it persists in $\{ u > 0 \}$ for $0<\delta \ll 1$.
Indeed, by the Fenichel theory for the persistence of normally hyperbolic critical manifolds (see \cite{F1979,HPS1977,J1995}) there is a family of saddle, slow, invariant manifolds $S^\delta_s$,
\begin{equation}\label{eq:S-delta}
    S^\delta_s 
    = \left\{ 
u = u_s(v) + \mathcal{O}(\delta^2), 
p = -\delta \left( \tfrac{ q }{2 u_s(v)} + \tfrac{3q}{4}(\alpha_1 - \alpha_2) \right) +\mathcal{O}(\delta u_s(v), \delta^2),  
v \ge 0 , 
q \in \mathbb{R}
    \right\}.
\end{equation}
These are $C^r$ $\mathcal{O}(\delta)$ close to $S^0_s$ for any $r>0$.

For $0<\delta \ll 1$, $W^s_{\rm loc}(S^0_s)$ and $W^u_{\rm loc}(S^0_s)$  persist as invariant stable and unstable manifolds $W^s_{\rm loc}(S^\delta_s)$ and $W^u_{\rm loc}(S^\delta_s)$ in the phase space of \eqref{eq:y-ODE}.
This also follows from the Fenichel theory for the persistence of normally hyperbolic invariant manifolds \cite{F1979,HPS1977,J1995}.
Moreover, in a neighborhood of $S_s^\delta$, these perturbed manifolds are $C^r$ $\mathcal{O}(\delta)$ close to their unperturbed counterparts.
We will show that $W^u(S^\delta_s)$ and $W^s(S^\delta_s)$ intersect transversely (see Section~\ref{sec:proof-theorem2}).

\begin{remark}
  In the general situation that we study in which the null set $\{ F=0 \}$ in the PDE \eqref{eq:genAI-UV} is given by the graph of a function $V=V_0(U)$ and $(U_f,V_f)$ is a non-degenerate fold point of \eqref{eq:genAI-UV} with $V_f=V_0(U_f)$,
the sheet $S_s^0$ corresponds to the following component of the null set:
\begin{equation}\label{eq:calU}
    \mathcal{U} = \left\{ (U,V) :   V=V_0(U), \, \, \,  
 \left. \frac{\partial F}{\partial U} \right\vert_{V_0(U)} < 0 \right\},
\end{equation}
and the boundary of $\cal{U}$ corresponds to the fold line $L$, where $\left. \frac{\partial F}{\partial U} \right\vert_{V_0(U)} = 0$.  
\end{remark}

\subsection{Desingularized reduced vector field on the critical manifold \texorpdfstring{$S^0$}{Lg}}

On the critical manifold $S^0$, the equations are given by \eqref{eq:x-ODE} with $\delta=0$,
\begin{equation}\label{eq:onS0}
\begin{split}
    0 &= p \\
    0 &= v + u^2 + \alpha_1 uv + \alpha_2 u^3 + \mathcal{O}(v^2,u^2v,u^4) \\
    v_x &= q \\
    q_x &= -\frac{\lambda}{2} + \alpha_3 u + \alpha_4 v + \alpha_5 u^2 + \mathcal{O}(uv,v^2,u^3).
\end{split}
\end{equation}
This is an algebraic-differential system in which the two algebraic conditions define $S^0$ and the two differential equations govern the dynamics of solutions on $S^0$.

All four variables $(u, p, v, q)$ arise naturally in the algebraic conditions, as well as in the differential equations on $S^0$. 
However, because $S^0$ is a two-dimensional manifold, one may obtain a
self-contained vector field involving only two variables that governs the dynamics on it.
Moreover, due to the parabolic cylinder shape of $S^0$, it is convenient to use the $(u,q)$ variables.
In particular, we differentiate the second condition with respect to $x$ and evaluate the terms on $S^0$ to obtain
\begin{equation}\label{eq:reduced}
\begin{split}
    u(2 + (3\alpha_2-\alpha_1)u + \mathcal{O}(u^2)) u_x &= -(1+\alpha_1 u +\mathcal{O}(u^2))  q \\
    q_x &= -\frac{\lambda}{2} + \alpha_3 u + (\alpha_5-\alpha_4) u^2 + \mathcal{O}(u^3).
\end{split}
\end{equation}
This system is referred to as the reduced (slow) system.
It is singular on the fold line $L$, {\it i.e.}, at $u=0$.
Solutions approach $L$ in finite time, forward and backward.
See Figure~\ref{fig:reducedflow}.

\begin{figure}[ht!]
   \centering
   \includegraphics[width=5in]{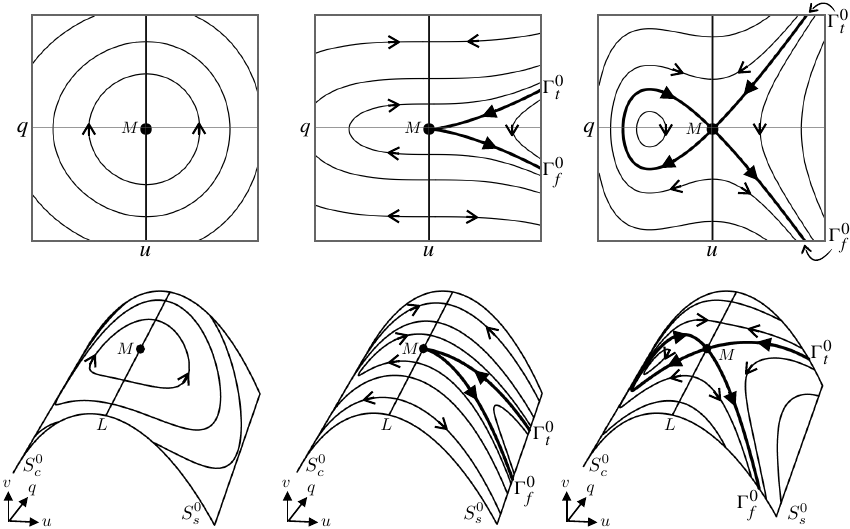}
   \caption{Projection of the reduced flow \eqref{eq:reduced} with $\alpha_3 < 0$ into the $(u,q)$ plane (top row) and into the $(u,q,v)$ space (bottom row), for three different values of $\lambda$.
   The critical manifold has a hyperbolic saddle sheet, $S_s^0$, for $u>0$ and an elliptic center sheet, $S_c^0$, for $u<0$. 
   As shown in Section~\ref{sec:MRFS+MRFSNII}, there is a reversible folded singularity, $M$, along the fold curve, $L$.
   {\bf Left column:} with $\lambda<0$, $M$ is a RFC. 
   {\bf Middle column:} with $\lambda=0$, $M$ is a RFSN II with true and faux canards, $\Gamma_t^0$ and $\Gamma_f^0$, that meet at $M$ in a cusp. 
   {\bf Right column:} with $\lambda>0$, $M$ is a RFS with true canard, $\Gamma_t^0$, that passes from $S_s^0$ to $S_c^0$, and faux canard, $\Gamma_f^0$, that passes from $S_c^0$ to $S_s^0$.
   These orbits form a homoclinic loop on $S_c^0$.}
   \label{fig:reducedflow}
\end{figure}

Next, to study the dynamics of the reduced system, we desingularize it by rescaling the independent variable from $x$ to $x_d$, where
$\frac{d}{dx_d} = u(2 + (3\alpha_2-\alpha_1)u + \mathcal{O}(u^2))\frac{d}{dx}$. 
The system becomes 
\begin{equation}\label{eq:desingularized-reduced}
\begin{split}
    u_{x_d} &= -(1+\alpha_1 u +\mathcal{O}(u^2))  q \\
    q_{x_d} &= -\lambda u + \left( 2\alpha_3 + \frac{\lambda}{2} ( \alpha_1 - 3\alpha_2) \right) u^2 + \mathcal{O}(u^3).
\end{split}
\end{equation}
This system is often referred to as the desingularized reduced system,
or simply the desingularized system.
The rescaling preserves the direction of the flow on $S^0_s$ since
$u>0$ there; {\it i.e.,} the flow of \eqref{eq:desingularized-reduced} is topologically equivalent to that of \eqref{eq:reduced} on $S^0_s$.
In contrast, the direction of the flow on $S^0_c$, which is described by the same desingularized system, is reversed, since $u<0$ on $S^0_c$.

\begin{remark}\label{rem:Hamiltonian-desingularizedsystem}
After another rescaling of the independent variable, by $(1+\alpha_1 u + \mathcal{O}(u^2))^{-1}$,
we see that \eqref{eq:desingularized-reduced} is a one degree of freedom Hamiltonian system, with $H_d(u,q;\lambda)=\tfrac{1}{2}q^2 + \int^u n(\bar{u}) d\bar{u},$
where $n(u)= \left( -\lambda u + (2\alpha_3 + \tfrac{\lambda}{2}(\alpha_1 - 3\alpha_2))u^2 + \mathcal{O}(u^3)\right) \left(1 + \alpha_1 u + \mathcal{O}(u^2)\right)^{-1}.$
Hence, the folded singularities must be either folded saddles, folded saddle-nodes, or folded centers.
\end{remark}

\subsection{The folded singularities \texorpdfstring{$M_{\rm RFS}$ and $M_{\rm RFSN-II}$}{Lg}}
\label{sec:MRFS+MRFSNII}

For each $\lambda\ne 0$, the desingularized reduced vector field \eqref{eq:desingularized-reduced} has two singularities: a folded singularity and an ordinary singularity.
The former is the fold point
\begin{equation*}
M = \{ u=0, \ q=0 \},
\end{equation*}
where $f$ and $\frac{\partial f}{\partial u}=0$.
It lies on the fold line $L$.
Moreover, $M$ is a {\it reversible} folded singularity
due to the reversibility symmetry \eqref{eq:reversible}.
We add that, for $\lambda\ne 0$, $M$ is not an equilibrium of the full slow flow \eqref{eq:onS0}, since $g = \frac{1}{2} \lambda$ there.
Then, the latter singularity is the equilibrium,
\begin{equation}
  E = \left\{ u = \frac{2\lambda}{4\alpha_3 + \lambda (\alpha_1 - 3 \alpha_2)}, \ \ q = 0 \right\}.
\end{equation}
It is also an equilibrium of \eqref{eq:onS0}, where $g$ vanishes.
The formula is for the leading order.

The trace of the Jacobian matrix of \eqref{eq:desingularized-reduced} is $-(\alpha_1 + \mathcal{O}(u))q$. 
Hence, the trace is zero at both $M$ and $E$.
The determinant of the Jacobian is $-\lambda + (4\alpha_3 -3\alpha_2\lambda)u + \mathcal{O}(u^2)$.
In particular, at $M$, the determinant is exactly $-\lambda$, whereas at $E$ it is $\lambda + \mathcal{O}(\lambda^2)$ for small $\lambda>0$.
Hence, from the general classification of folded singularities \cite{B1983,KW2010,SW2001}, we have the following linear stability results.

For $\lambda<0$, $M$ is a reversible folded center (RFC), and E is a saddle fixed point.
Thus, there are no canards, since FCs do not have canards \cite{SW2001}.
See Figure~\ref{fig:reducedflow}(a).

Then, at $\lambda=0$, $E$ merges with $M$ at $(0,0)$ in a reversible folded saddle-node of type II,
\begin{equation}
\label{eq:M-RFSNII}
M_{\rm RFSN-II} = \{ u=0, \, q=0, \, \lambda=0 \}.
\end{equation}
The stable and unstable manifolds of $M_{\rm RFSN-II}$ each have one branch.
They are cusp shaped near the origin.
They form the true and faux canards, $\Gamma_{t}^0$ and $\Gamma_{f}^0$, of $M_{\rm  RFSN-II}$. 
See \cite{KW2010} for the analysis of general FSN-II points, and see Figure~\ref{fig:reducedflow}(b) for an illustration of the local dynamics here.

Most interestingly, for each $\lambda>0$, $M$ is a reversible folded saddle (RFS),
\begin{equation}\label{eq:M-RFS} 
M_{\rm RFS}=\{ u=0, \, q=0, \, \lambda>0 \},
\end{equation}
and $E$ is a center.
Moreover, the true canard $\Gamma_t^0$ of $M_{\rm RFS}$ is given by the two-branched stable manifold of $M_{\rm RFS}$, and the faux canard $\Gamma_f^0$ by the two-branched unstable manifold.
See \cite{B1983,SW2001,MW2017} for the analysis of general FS points, and 
Figure~\ref{fig:reducedflow}(c) for an illustration of the local dynamics here.

\begin{proposition}\label{prop:singularcanardsconnect}
In the $(u,q)$ phase plane of the reduced system \eqref{eq:reduced}, the branches of the singular true and faux canards $\Gamma_t^0$ and $\Gamma_f^0$ in $\{ u\le 0\}$ coincide in an orbit homoclinic to the saddle $(0,0)$.
\end{proposition}

\begin{proof}
From Remark~\ref{rem:Hamiltonian-desingularizedsystem},
we recall the Hamiltonian $H_d(u,q;\lambda)$ for the rescaled version of the  desingularized reduced system \eqref{eq:desingularized-reduced}. 
The folded singularity $M$ lies on the level set $H_d=0$.
In $\{ u \leq 0 \}$, this level set is a simple closed loop that connects $M_{\rm RFS}$ to itself for each $\lambda>0$. 
\end{proof}

A central observation about the folded singularity $M$ is that, for $\lambda\ne 0$, it is not a fixed point of the reduced system \eqref{eq:reduced}.
The fact that the coefficient on $u_x$ vanishes at $u=0$ implies that solutions approach $M$ in finite time, both in forward and backward $x$,
and that multiple solutions may pass through this point.
There is a l'Hopital type cancellation in the first component of the vector field at $M$, so that $u_x \ne 0$ there and solutions of \eqref{eq:reduced} that lie along the singular canards of $M$ pass through it with finite non-zero speed.
This is the mechanism by which solutions of the reduced system can pass across $L$ from $S_s^0$ to $S_c^0$ and vice versa.
This situation contrasts with that in the desingularized system \eqref{eq:desingularized-reduced}, where $M$ is a hyperbolic fixed point, and the singular canards approach $M$ in infinite time along its stable and unstable manifolds.
(We refer the reader to \cite{T1976} for general theory about these solutions in differential-algebraic systems obtained in the limit of fast-slow ODEs.
We also refer to Section 2 of \cite{KS2001}.)

\begin{remark}
    $M_{\rm RFSN-II}$ exists in the singular limit ($\delta \to 0$) of the subcritical Turing bifurcation.
    This relation between the Turing point and the RFSN-II point was a central discovery in \cite{VDK2025}, for the van der Pol PDE. 
    The same phenomenon was observed in the Brusselator PDE in \cite{JDKV2026}.
\end{remark}

\section{Geometric desingularization of \texorpdfstring{$M_{\rm RFS}$}{Lg}}\label{sec:geodesing}

In this section, we present the geometric desingularization (a.k.a. blow-up) analysis of the RFS point, $M_{\rm RFS}$ \eqref{eq:M-RFS}.
We consider $\mathcal{O}(1)$ values of $\lambda>0$ in \eqref{eq:y-ODE}. 
The blow-up transformation is
\begin{equation}\label{eq:blowup}
u = r^2 \bar{u}, \, \, \,
p = r^3 \bar{p}, \, \, \,
v = r^4 \bar{v}, \, \, \,
q= r^2 \bar{q}, \, \, \,
\delta = r^3 \bar{\delta}.
\end{equation}
This transformation is a map
$\Phi: \mathbb{S}^4 \times [-r_0,r_0]  \to \mathbb{R}^5,$
where $\bar{u}^2 + \bar{p}^2 + \bar{v}^2 + \bar{q}^2 + \bar{\delta}^2 = 1$ and $r_0>0$ is sufficiently small, independent of $\delta$.
The variables $u$ and $q$ scale with the same powers of $r$ as $r \to 0$, because $M_{\rm RFS}$ is a folded saddle whose stable and unstable manifolds intersect transversely in the $(u,q)$ plane, with $\mathcal{O}(1)$ angle independent of $\delta$ and $r$.
Moreover, $v$ scales as the square of $u$ to balance the two leading order terms in the second component of \eqref{eq:y-ODE}.

We examine the dynamics induced by \eqref{eq:y-ODE} in two coordinate charts:
the central (or rescaling) chart $K_2$ defined by $\{\bar{\delta}=1\}$ and 
the entry/exit chart $K_1$ defined by $\{ \bar{u}=1 \}$.
See Sections~\ref{sec:K2} and \ref{sec:K1}.
In $K_2$, we find the algebraic solutions that form the singular true and faux canards of $M_{\rm RFS}$.
Then, we track the singular canards into chart $K_1$ (see Section~\ref{sec:transitionmap-kappa21}), and identify the key center manifolds, showing that these intersect and that the singular canards persist as maximal canards for $0<\delta \ll 1$, which lie in these intersections, see Section~\ref{sec:persistence-RFS}.
We use $\square_i$ to denote variable $\bar\square$ in chart $K_i$, $i=1,2$.

\subsection{The central chart \texorpdfstring{$K_2$ for $M_{\rm RFS}$}{Lg} } \label{sec:K2}
In the central chart $K_2$ where $\delta_2\equiv 1$, the blow-up transformation is
\begin{equation} \label{eq:K2-coords}
K_2: \, \, 
  u = r_2^2 u_2, \, \,  \,
  p = r_2^3 p_2, \, \,  \,
  v = r_2^4 v_2, \, \,  \, 
  q = r_2^2 q_2, \, \,  \,
  \delta=r_2^3.    
\end{equation}
Substituting this coordinate change into system \eqref{eq:y-ODE}
with the equation $\delta_y=0$ appended,
and introducing the independent variable $y_2$ defined 
by $\frac{d}{dy_2} = \frac{1}{r_2} \frac{d}{dy}$ (which desingularizes the system),
we find that the governing equations in $K_2$ are
\begin{equation} \label{eq:K2}
\begin{split}
    \frac{du_2}{dy_2} &= p_2, \\ 
    \frac{dp_2}{dy_2} &= v_2 + u_2^2 + r_2^2 h_2(u_2,v_2,r_2), \\
    \frac{dv_2}{dy_2} &= q_2, \\
    \frac{dq_2}{dy_2} &= -\frac{1}{2}\lambda + r_2^2 k_2(u_2,v_2,r_2),\\
    \frac{dr_2}{dy_2} &= 0,
\end{split}
\end{equation}
where $h_2= \alpha_1 u_2 v_2 + \alpha_2 u_2^3 + \mathcal{O}(r_2^2)$
and $k_2= \alpha_3 u_2 + r_2^2(\alpha_4 v_2 + \alpha_5 u_2^2) + \mathcal{O}(r_2^4)$.

For each $r_2 \ge 0$, the hyperplane $\{ r_2 = {\rm constant} \}$ is invariant.
On $\{ r_2 = 0 \}$, the vector field is 
\begin{equation} \label{eq:K2-on-r2=0}
\begin{split}
    \frac{du_2}{dy_2} &= p_2, \\ 
    \frac{dp_2}{dy_2} &= v_2 + u_2^2, \\
    \frac{dv_2}{dy_2} &= q_2, \\
    \frac{dq_2}{dy_2} &= -\frac{1}{2}\lambda.
\end{split}
\end{equation}
The equations for $v_2$ and $q_2$ decouple.
This system possesses the following key algebraic solutions:
\begin{equation} \label{eq:gamma_TF}
 \begin{split}
    \Gamma_t^0 (y_2) &= \left( -\frac{1}{2}\sqrt{\lambda} y_2, \, -\frac{1}{2} \sqrt{\lambda},  \, -\frac{1}{4} \lambda y_2^2, \, -\frac{1}{2}\lambda y_2 \right) \qquad y_2 \in \mathbb{R},\\
    \Gamma_f^0 (y_2) &= \left( \frac{1}{2}\sqrt{\lambda} y_2, \,\, \frac{1}{2} \sqrt{\lambda}, \, -\frac{1}{4} \lambda y_2^2, \, -\frac{1}{2}\lambda y_2 \right) \qquad y_2 \in \mathbb{R}.
\end{split}
\end{equation}
These represent the singular true and faux canards, respectively, of $M_{\rm RFS}$ on $\{ r_2=0 \}$.

The local dynamics near $\Gamma^0_t$ and $\Gamma_f^0$ are established in the following proposition: 
\begin{proposition}\label{prop:twist}
In the limit $r_2 \to 0$, solutions of \eqref{eq:K2} near $\Gamma_t^0$ can exhibit infinitely many twists on $S_c^0$, but none on $S_s^0$.
In contrast, solutions of \eqref{eq:K2} near $\Gamma_f^0$ can exhibit infinitely many twists on $S_s^0$, but none on $S_c^0$.
\end{proposition}

\begin{proof}
See Appendix~\ref{app:proptwist}.
\end{proof}

\subsection{The entry/exit chart \texorpdfstring{$K_1$ for $M_{\rm RFS}$}{Lg}} \label{sec:K1}

In the entry/exit chart $K_1$ where $u_1\equiv 1$, the coordinate transformation is
\begin{equation} \label{eq:K1-coords}
K_1: \, \, 
  u = r_1^2, \, \,  \,
  p = r_1^3 p_1, \, \,  \,
  v = r_1^4 v_1, \, \,  \, 
  q = r_1^2 q_1, \, \,  \,
  \delta=r_1^3 \delta_1.    
\end{equation}
We substitute this coordinate change into \eqref{eq:y-ODE} with $\delta_y=0$ appended, and we introduce $y_1$ defined by $\frac{d}{dy_1} = \frac{1}{r_1} \frac{d}{dy}$, to desingularize the system.
The governing equations in $K_1$ are
\begin{equation} \label{eq:K1}
\begin{split} 
    \frac{dr_1}{dy_1} &= \frac{1}{2} r_1 p_1, \\
    \frac{dp_1}{dy_1} &= v_1 + 1 - \frac{3}{2} p_1^2 + r_1^2 h_1(p_1,v_1,q_1,r_1), \\
    \frac{dv_1}{dy_1} &= -2p_1v_1 + \delta_1 q_1, \\
    \frac{dq_1}{dy_1} &= -p_1q_1 -\frac{1}{2}\lambda \delta_1 + r_1^2 \delta_1 k_1(p_1,v_1,q_1,r_1), \\
    \frac{d\delta_1}{dy_1} &= - \frac{3}{2} p_1 \delta_1,
\end{split}
\end{equation}
where $h_1= \alpha_1 v_1 + \alpha_2 + \mathcal{O}(r_1^2)$
and $k_1= \alpha_3 + r_1^2 (\alpha_4 v_1 + \alpha_5) + \mathcal{O}(r_1^4)$.

The hyperplane $\{ r_1 = 0 \}$ is invariant.
On it, there are two isolated equilibria
\begin{equation}\label{eq:Epm}
 E_{\pm} = \left\{ r_1=0, \, p_1 = \pm\rho, \, v_1=0, \, q_1=0, \, \delta_1=0 \right\},
\end{equation}
where $\rho=\sqrt{\tfrac{2}{3}}$. 
See Figure~\ref{fig:pvq-RFS}. 
The stable and unstable spectra at $E_{+}$ are 
\begin{equation*}
\sigma_s(E_{+}) = \{  -3\rho, -2\rho, 
-\tfrac{3}{2}\rho,
-\rho \} \quad
{\rm and} \quad 
\sigma_u(E_{+}) = \{ \tfrac{1}{2}\rho \}.
\end{equation*}
The associated eigenspaces are 
\begin{equation*}
\mathbb{E}^s(E_+)
    = {\rm span} \, \left\{ 
    \left[
      \begin{array}{c}
      0 \\ 1 \\ 0 \\ 0 \\ 0
      \end{array}
      \right],
       \left[
      \begin{array}{c}
      0 \\ 1 \\ \rho \\ 0 \\ 0
      \end{array}
      \right],
      \left[
      \begin{array}{c}
      0 \\ 0 \\ 0 \\ \lambda \\ \rho
      \end{array}
      \right],     
      \left[
      \begin{array}{c}
      0 \\ 0 \\ 0 \\ 1 \\ 0
      \end{array}
      \right]
      \right\}
      \quad {\rm and} \quad 
       \mathbb{E}^u(E_{+}) 
    = {\rm span} \, \left[
      \begin{array}{c}
      1 \\ 0 \\ 0 \\ 0 \\ 0
      \end{array}
      \right].
\end{equation*}
Hence, $E_{+}$ has a four-dimensional stable manifold and a one-dimensional unstable manifold.

The stable and unstable spectra at $E_{-}$ are
\begin{equation*} 
\sigma_s(E_-) = \{ -\tfrac{1}{2}\rho \} 
\quad {\rm and} \quad 
\sigma_u(E_-) = \{ \rho,   
\tfrac{3}{2}\rho,
2 \rho,
3 \rho \}.
\end{equation*}
The eigenspaces are
\begin{equation*}
    \mathbb{E}^s(E_{-}) 
    = {\rm span} \, \left[
      \begin{array}{c}
      1 \\ 0 \\ 0 \\ 0 \\ 0
      \end{array}
      \right]
      \quad {\rm and} \quad 
      \mathbb{E}^u(E_{-}) 
    = {\rm span} \, \left\{ 
    \left[
      \begin{array}{c}
      0 \\ 0 \\ 0 \\ 1 \\ 0
      \end{array}
      \right],
      \left[
      \begin{array}{c}
      0 \\ 0 \\ 0 \\ -\lambda \\ \rho
      \end{array}
      \right],
      \left[
      \begin{array}{c}
      0 \\ -1 \\ \rho  \\ 0 \\ 0
      \end{array}
      \right],
      \left[
      \begin{array}{c}
      0 \\ 1 \\ 0 \\ 0 \\ 0
      \end{array}
      \right]
      \right\}. 
\end{equation*}
Hence, $E_{-}$ has a one-dimensional stable manifold and a four-dimensional unstable manifold.
Furthermore, there is a heteroclinic connection from $E_{-}$ to $E_{+}$ that lies on the $p_1$-axis, where the system reduces to $\frac{d p_1}{dy_1}=1- \frac{3}{2} p_1^2$.

\begin{figure}[ht!]
   \centering
   \includegraphics[width=5in]{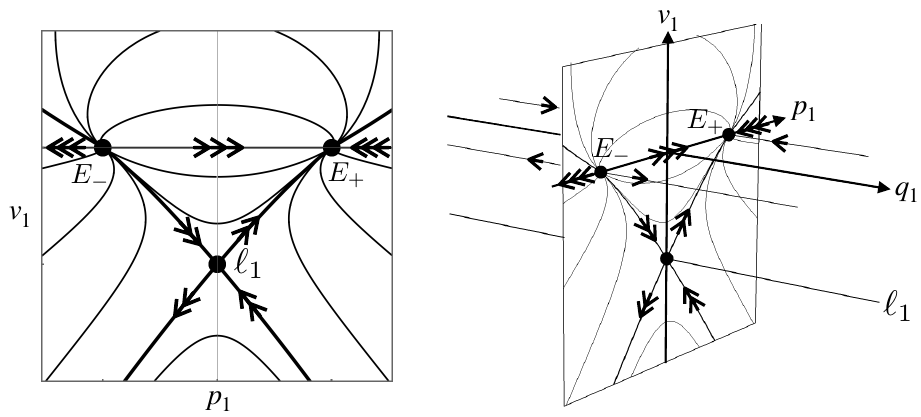}
   \caption{
   Dynamics of \eqref{eq:K1} in the invariant subspace $\{ r_1, \delta_1 = 0 \}$, which is coordinatized by $(p_1,v_1,q_1)$.
   The two hyperbolic equilibria $E_\pm$ \eqref{eq:Epm} and the line $\ell_1$  of fixed points \eqref{eq:ell1} emerge from the folded saddle $M_{\rm RFS}$ in the blowup transformation.
   }
   \label{fig:pvq-RFS}
\end{figure}

System \eqref{eq:K1} also possesses a line of fixed points, 
\begin{equation}\label{eq:ell1}
   \ell_1(q_1) = \{ r_1=0, \, p_1 = 0, \, v_1 = -1, \, q_1 \in \mathbb{R}, \, \delta_1=0\}.
\end{equation}
We sometimes refer to it as $\ell_1$.
For each $q_1 \in \mathbb{R}$, the linearization at $\ell(q_1)$ has point spectrum
\begin{equation*}
    \sigma_{\rm pt} (\ell_1(q_1)) = \{ \pm \sqrt{2}, 0, 0, 0 \}.
\end{equation*}
Hence, at each point on $\ell_1$, there are one-dimensional stable and unstable eigenspaces and a three-dimensional generalized center eigenspace,
\begin{equation*}
    \mathbb{E}^{s,u}(\ell_1(q_1)) = {\rm span} \, 
    \left[ \begin{array}{c} 
    0 \\ \pm \sqrt{2} \\ -2 \\ q_1 \\ 0 
    \end{array} \right], 
\quad 
    \mathbb{E}^{c} (\ell_1(q_1)) = {\rm span} \,
    \left\{ 
    \left[ \begin{array}{c} 
    1 \\ 0 \\ 0 \\ 0 \\ 0 
    \end{array} \right],
    \left[ \begin{array}{c} 
    0 \\ 0 \\ 0 \\ 1 \\ 0 
    \end{array} \right],  
    \left[ \begin{array}{c} 
    0 \\ -q_1 \\ 0 \\ 0 \\ 2 
    \end{array} \right]
    \right\}.
\end{equation*}
The third vector in $\mathbb{E}^c(\ell_1(q_1))$ is a generalized eigenvector. 
It becomes a regular eigenvector at two points, where $q_1=\pm\sqrt{\lambda}$, see also \eqref{eq:limitKappa21-Gammat} and \eqref{eq:limitKappa21-Gammaf}.
By center manifold theory \cite{C1981,GH1983,HPS1977}, each point $\ell_1(q_1)$ has a three-dimensional center manifold, $W^c(\ell_1)$.
Moreover, $W^c(\ell_1)$ contains the surface of fixed points
\begin{equation}\label{eq:tildeS}
    \tilde{\mathcal{S}} = \{ r_1\ge 0, \, \, p_1=0, \, \, v_1=-1+\mathcal{O}(r_1^2), \, \, q_1 \in \mathbb{R}, \, \, \delta_1=0 \}.
\end{equation}

\subsection{The transition map from \texorpdfstring{$K_2$ {\rm to} $K_1$}{Lg}} \label{sec:transitionmap-kappa21}

The transition map from $K_2$ to $K_1$ is given by
\begin{equation}\label{eq:kappa12}
\kappa_{21} (u_2,p_2,v_2,q_2,r_2) 
= (r_1,p_1,v_1,q_1,\delta_1)
= \left( 
r_2\sqrt{u_2}, \frac{p_2}{u_2^{3/2}}, \frac{v_2}{u_2^2}, \frac{q_2}{u_2}, \frac{1}{u_2^{3/2}}
\right) \, \, {\rm for} \, \, u_2>0.
\end{equation}
In this section, we use the map $\kappa_{21}$ to study the backward asymptotics of solutions on $\Gamma_t^0$ and the forward asymptotics of solutions on $\Gamma_f^0$ in $K_1$.  
We evaluate the following two limits
\begin{equation} \label{eq:limitKappa21-Gammat}
\begin{split} 
    {\bf p}_- &= \lim_{y_2 \to -\infty} \kappa_{21}(\Gamma_t^0) \\
    &= \lim_{y_2 \to -\infty} \left(
    0, \frac{-\frac{1}{2}\sqrt{\lambda}}{(-\frac{1}{2}\sqrt{\lambda} y_2)^{3/2}},
    \frac{-\frac{1}{4}\lambda y_2^2}{(-\frac{1}{2}\sqrt{\lambda} y_2)^2},
    \frac{-\frac{1}{2}\lambda y_2}{-\frac{1}{2}\sqrt{\lambda} y_2},
    \frac{1}{(-\frac{1}{2}\sqrt{\lambda} y_2)^{3/2}}
    \right) \\
    &=(0,0,-1,\sqrt{\lambda},0) \in \ell_1. 
\end{split}
\end{equation} 
and 
\begin{equation*}
    \lim_{y_2 \to -\infty} 
    \frac{ \frac{d}{dy_2} \kappa_{21}(\Gamma_t^0)} 
    { \Vert \frac{d}{dy_2} \kappa_{21}(\Gamma^0_t) \Vert}
       =\left(0,\frac{-\sqrt{\lambda}}{\sqrt{\lambda + 4}},0,0,\frac{2}{\sqrt{\lambda+4}}\right) \in \mathbb{E}^c(\ell_1). 
\end{equation*}
These calculations show that $\Gamma_t^0$ emanates from the point ${\bf p}_-$ on $\ell_1$ with $q_1=\sqrt{\lambda}$, {\it i.e.,} from the point $\ell_1(\sqrt{\lambda})$, and that the tangent vector lies in the subspace of $\mathbb{E}^c(\ell_1)$ spanned by the generalized eigenvector with $q_1=\sqrt{\lambda}$.
Therefore, $\Gamma^0_t$ lies on $W^{c}(\ell_1)$.

Similarly, $\Gamma_f^0$ terminates at the point ${\bf p}_+$ on $\ell_1$ with $q_1=-\sqrt{\lambda}$, and $\Gamma^0_f$ lies on $W^{c}(\ell_1)$.
This follows from observations about the limits of $\Gamma_f^0$ and the tangent vector.
In particular, we find
\begin{equation}\label{eq:limitKappa21-Gammaf}
\begin{split} 
    {\bf p}_+ &= \lim_{y_2 \to \infty} \kappa_{21}(\Gamma_f^0) \\
    &= \lim_{y_2 \to \infty} \left(
    0, \frac{\frac{1}{2}\sqrt{\lambda}}{(\frac{1}{2}\sqrt{\lambda} y_2)^{3/2}},
    \frac{-\frac{1}{4}\lambda y_2^2}{(\frac{1}{2}\sqrt{\lambda} y_2)^2},
    \frac{-\frac{1}{2}\lambda y_2}{\frac{1}{2}\sqrt{\lambda} y_2},
    \frac{1}{(\frac{1}{2}\sqrt{\lambda} y_2)^{3/2}}
    \right) \\
    &=(0,0,-1,-\sqrt{\lambda},0) \in \ell_1, 
    \end{split} 
\end{equation}
and 
\begin{equation*}
    \lim_{y_2 \to \infty} 
    \frac{ \frac{d}{dy_2} \kappa_{21}(\Gamma_f^0)}
    { \Vert \frac{d}{dy_2} \kappa_{21}(\Gamma_f^0) \Vert}
       =\left(0,\frac{-\sqrt{\lambda}}{\sqrt{\lambda + 4}},0,0,\frac{-2}{\sqrt{\lambda+4}}\right) \in \mathbb{E}^c(\ell_1). 
\end{equation*}

\subsection{The existence of maximal true and faux canards, 
\texorpdfstring{$\Gamma^\delta_t$ and $\Gamma^\delta_f$, for $0< \delta \ll 1$}{Lg}}
\label{sec:persistence-RFS}

In this section, we establish the persistence of the ($\delta=0$) singular canards as maximal canards for $0<\delta \ll 1$ and the persistence of the connection between the maximal true and faux canards.

\medskip

\begin{lemma} \label{lem-persistence-true+fauxcanards-RFS}
There exists a $\delta_0>0$ sufficiently small such that for each $0<\delta<\delta_0$ and any $\mathcal{O}(1)$ value of $\lambda>0$,
the singular true and faux canards $\Gamma_t^0$ and $\Gamma_f^0$ persist as maximal true and faux canards, labeled $\Gamma_t^\delta$ and $\Gamma_f^\delta$, of \eqref{eq:y-ODE}.
The maximal canards lie in the transverse intersection of slow invariant manifolds,
and the connection between $\Gamma_t^\delta$ and $\Gamma_f^\delta$ persists for $0 < \delta< \delta_0$.
\end{lemma}

\begin{proof}
We begin by recalling that, in the phase plane of the $(u,q)$ reduced system \eqref{eq:reduced} and hence also in the vector field \eqref{eq:onS0} induced on $S_0$, the singular true and faux canards $\Gamma_t^0$ and $\Gamma_f^0$ of $M_{\rm RFS}$ are the stable and unstable manifolds of the saddle point $(u,q)=(0,0)$.
Moreover, by Proposition~\ref{prop:singularcanardsconnect}, we know that the branches of these manifolds in the half-plane $\{ u<0\}$ coincide in an orbit homoclinic to the saddle point $M_{\rm RFS}$.

Now, in the four-dimensional phase space of the spatial ODE system \eqref{eq:y-ODE}, we define the curve 
\begin{equation}\label{eq:Gamma}
\Gamma^0 = \Gamma_t^0 \cup \Gamma_f^0.
\end{equation}
Solutions on this singular orbit $\Gamma^0$ pass through the folded saddle $M_{\rm RFS}$ in finite time due to the l'Hopital zero over zero cancellation in the first component of the reduced vector field \eqref{eq:reduced} (recall the analysis in Section~\ref{sec:MRFS+MRFSNII}).
They may be parametrized by a real scalar variable, which we label as $\zeta$.
The limits \eqref{eq:limitKappa21-Gammat} and \eqref{eq:limitKappa21-Gammaf} for the singular canards and their tangent vectors imply that $\Gamma^0 \to {\bf p}_-$ as $\zeta \to -\infty$, that $\Gamma^0 \to {\bf p}_+$ as $\zeta \to \infty$, and that these solutions lie in the center manifold $W^c(\ell_1)$. 
Moreover, the segment of $\Gamma^0$ that lies in the half-space $\{ u < 0 \}$ is given exactly by the curve representing the homoclinic loop,
and without loss of generality we may parametrize that segment of $\Gamma^0(\zeta)$ by $\zeta \in [-L,L]$ for some $\mathcal{O}(1)$ number $L>0$, since the system is autonomous.

Next, we consider the equation of variation about $\Gamma^0$ in the full four-dimensional phase space. 
By definition, $\tfrac{d\Gamma^0}{d \zeta}$ is a solution of the variational equation.
Also, we know that, at ${\bf p}_\pm$, $W^c(\ell_1)$ is normally 
repelling in the limits $\zeta \to \pm \infty$, respectively.
Hence, $\tfrac{d\Gamma^0}{d\zeta}$ is the only solution
of the variational equation that is bounded in both limits $\zeta \to \pm \infty$;
all other solutions diverge from ${\bf p}_\pm$ in forward or backward time, respectively.
In turn, this implies that $W^c({\bf p}_-) \cap W^c({\bf p}_+)$ on $\{ \bar r = 0 \}$ along the one-dimensional curve $\Gamma^0$, and this intersection is transverse.

Having established that $W^c({\bf p}_-) \cap W^c({\bf p}_+)$
on the invariant set $\{ \bar r = 0 \}$, we now turn to analyze the dynamics and geometry of the invariant manifolds for small $\bar r > 0$. 
Regular perturbation theory for invariant manifolds implies that, in chart $K_2$, the transverse intersection $W^c({\bf p}_-) \cap W^c({\bf p}_+)$ persists for sufficiently small values of $r_2>0$.
Then, by tracking the manifolds into the entry/exit chart $K_1$, we see that the intersection remains transverse.
Furthermore, in chart $K_2$, the manifold $W^c({\bf p}_-) \cap W^c({\bf p}_+)$ is foliated by curves of $\{ r_2= {\rm constant} \}$, since $\tfrac{dr_2}{dy_2}=0$.
In turn, this foliation induces a foliation of the transverse intersection by $r_1 \delta_1^{1/3}={\rm constant}$ in $K_1$, since $r_1^3 \delta_1 = r_2^3 \delta_2$ and $\delta_2=1$ in $K_2$.
The orbits that lie in the intersection are exactly the persistent maximal true and faux canards, and they exist for sufficiently small $r_2>0$, and hence for sufficiently small $\delta>0$.

Lastly, to show that the connection between the true and faux canards persists, we use the Hamiltonian $H_d$ of the desingularized reduced system \eqref{eq:desingularized-reduced} in the $(u,q)$ plane, recall Remark~\ref{rem:Hamiltonian-desingularizedsystem}.
Let $h$ be the value of $H_d$ at 
$M_{\rm RFS}$.
The equation $H_d(u,q;\lambda)=h$ is a quadratic equation in $q$. 
The two roots represent the singular true and faux canards of $M_{\rm RFS}$.
For each $\mathcal{O}(1)$ value of $\lambda > 0$, a direct calculation of the slopes of the tangent lines to these solutions
at $M_{\rm RFS}$, where $q=0$, shows that the singular canards intersect transversely there.
Then, for all $0< \delta \ll 1$ and any $\mathcal{O}(1)$ value of $\lambda>0$, the intersection will persist somewhere along $\{ q = 0 \}$.
Lastly, for all $\delta \ge 0$ and $\lambda>0$, the reversibility symmetry ${\mathcal R}$ \eqref{eq:reversible} enforces $\Gamma_t^\delta = {\mathcal R}(\Gamma_f^\delta)$.
Hence, $\Gamma^\delta_t$ and $\Gamma^\delta_f$ must intersect $\{ q=0 \}$ at the same point, so that the connection persists.
\end{proof} 

\medskip
\noindent 
The maximal canards are $\mathcal{O}(\delta)$ close to their unperturbed counterparts for all finite $y_2$.
In addition, each maximal canard lies at the heart of a family of perturbed canards.
The members of these families are exponentially close to the corresponding maximal canard, with each member being close along an interval of $x$ values that is centered about the folded singularity but with each member being of different lengths in $x$. 
(The length of 
the interval of closeness may be thought of as the parameter.)
All persistent canards lie in the transverse intersections of $S_s^\delta$  with itself.

\section{Geometric desingularization of \texorpdfstring{$M_{\rm RFSN-II}$ }{Lg}}\label{sec:RFSNII-desing}

In this section, we present the geometric desingularization analysis of the RFSN-II point, $M_{\rm RFSN-II}$ \eqref{eq:M-RFSNII}.
We consider $0 \le \lambda\ll 1$ in the ODE system \eqref{eq:y-ODE}.
The blow-up transformation is
\begin{equation}\label{eq:blowup-RFSNII}
u = r^2 \bar{u}, \, \, \,
p = r^3 \bar{p}, \, \, \,
v = r^4 \bar{v}, \, \, \,
q= r^3 \bar{q}, \, \, \,
\delta = r^2 \bar{\delta}, \, \, \, 
\lambda=r^3 \bar{\lambda}.
\end{equation}
This transformation is a map
$\Phi: \mathbb{S}^5 \times [-r_0,r_0]  \to \mathbb{R}^6,$
where $\bar{u}^2 + \bar{p}^2 + \bar{v}^2 + \bar{q}^2 + \bar{\delta}^2 + \bar{\lambda}^2 = 1$ and $r_0>0$ is sufficiently small, independent of $\delta$.
The variables $u$ and $q$ scale with $r^2$ and $r^3$, respectively, because $M_{\rm RFSN-II}$ is a folded saddle-node at which the stable and unstable manifolds of $(u,q)=(0,0)$ are single-branched and meet in a cusp point with $q \sim u^{3/2}$ in the $(u,q)$ plane, and hence $q \sim (r^2)^{3/2}$ as $r \to 0$.  

We examine the dynamics induced by \eqref{eq:y-ODE} in two coordinate charts:
the central (or rescaling) chart $\hat{K}_2$ defined by $\{\bar{\delta}=1\}$
and the entry/exit chart $\hat{K}_1$ defined by $\{ \bar{u}=1 \}$.
We use $\hat{\square}_i$ to denote variable $\bar{\square}$ in chart $\hat K_i$ for i = 1, 2, where the hat differentiates the variables and labels used here from those in the desingularization analysis of $M_{\rm RFS}$ in Section~\ref{sec:geodesing}.

\subsection{The central chart \texorpdfstring{$\hat{K}_2$ for $M_{\rm RFSN-II}$}{Lg}} \label{sec:K2-RFSNII}
In the central chart $\hat{K}_2$ where $\hat{\delta}_2\equiv 1$, the blow-up transformation is
\begin{equation} \label{eq:K2-coords-RFSNII}
{\hat K}_2: \, \, \, 
  u = \hat{r}_2^2 \hat{u}_2, \, \, \,  \, \,
  p = \hat{r}_2^3 \hat{p}_2, \, \, \,  \, \, 
  v = \hat{r}_2^4 \hat{v}_2, \, \, \,  \,  \,
  q = \hat{r}_2^3 \hat{q}_2, \, \, \,  \, \, 
  \delta=\hat{r}_2^2, \, \, \, \, \, 
  \lambda=\hat{r}_2^3 \hat{\lambda}_2.
\end{equation}
We substitute these coordinates into system \eqref{eq:y-ODE} with $\delta_y=0$ and $\lambda_y= 0$ appended, and we introduce the independent variable $\hat{y}_2$ defined 
by $\frac{d}{d\hat{y}_2} = \frac{1}{\hat{r}_2} \frac{d}{dy}$.
The governing equations in $\hat{K}_2$ are
\begin{equation} \label{eq:K2-RFSNII}
\begin{split}
    \frac{d\hat{u}_2}{d\hat{y}_2} &= \hat{p}_2, \\ 
    \frac{d\hat{p}_2}{d\hat{y}_2} &= \hat{v}_2 + \hat{u}_2^2 + \hat{r}_2^2 \hat{h}_2(\hat{u}_2,\hat{v}_2,\hat{r}_2), \\
    \frac{d\hat{v}_2}{d\hat{y}_2} &= \hat{q}_2, \\
    \frac{d\hat{q}_2}{d\hat{y}_2} &= \alpha_3 \hat{u}_2  -\frac{1}{2}\hat{r}_2 \hat{\lambda}_2 + \hat{r}_2^2 \hat{m}_2(\hat{u}_2,\hat{v}_2,\hat{r}_2), \\
    \frac{d\hat{\lambda}_2}{d\hat{y}_2} &= 0 , \\
    \frac{d\hat{r}_2}{d\hat{y}_2} &= 0,
\end{split}
\end{equation}
where $\hat{h}_2= \alpha_1 \hat{u}_2 \hat{v}_2 + \alpha_2 \hat{u}_2^3 + \mathcal{O}(\hat{r}_2^3)$
and $\hat{m}_2=  \alpha_4 \hat{v}_2 + \alpha_5 \hat{u}_2^2 
+ \mathcal{O}(\hat{r}_2^2)$.

For each $\hat{r}_2>0$, the hyperplane $\{ \hat{r}_2 = {\rm constant} \}$ is invariant.
On $\{ \hat{r}_2=0 \}$, the vector field is 
\begin{equation} \label{eq:K2-on-r2=0-RFSNII}
\begin{split}
    \frac{d\hat{u}_2}{d\hat{y}_2} &= \hat{p}_2, \\ 
    \frac{d\hat{p}_2}{d\hat{y}_2} &= \hat{v}_2 + \hat{u}_2^2, \\
    \frac{d\hat{v}_2}{d\hat{y}_2} &= \hat{q}_2, \\
    \frac{d\hat{q}_2}{d\hat{y}_2} &= \alpha_3 \hat{u}_2.
\end{split}
\end{equation}
This system is known as the unperturbed problem.
It is Hamiltonian.
Let $\hat{q}_2 = - \alpha_3 \tilde{q}_2$ and define
${\bf u}_2 = (\hat{u}_2,\tilde{q}_2,\hat{p}_2,\hat{v}_2)$.
The system \eqref{eq:K2-on-r2=0-RFSNII} becomes
\begin{equation} \label{eq:K2-H2}
\dot{\bf u}_2=J {\bf \nabla}_{(\hat{u}_2,\tilde{q}_2,\hat{p}_2,\hat{v}_2)} H_2, 
\end{equation}
where $H_2=\tfrac{1}{2}(\hat{p}_2^2 + \alpha_3 \tilde{q}_2^2) - \hat{u}_2 \hat{v}_2 - \tfrac{1}{3}\hat{u}_2^3$ and 
\[ J = \begin{bmatrix} \mathbb O & \mathbb I \\ -\mathbb I & \mathbb O \end{bmatrix}, \]
where $\mathbb O$ and $\mathbb I$ denote the $2 \times 2$ zero and identity matrices, respectively.
We note that $H_2 \vert_{\Gamma_{20}} = 0.$

\begin{remark}
The Hamiltonian for \eqref{eq:K2-H2} exhibits a scaling invariance,
\begin{equation}\label{eq:scalinginvarH2}
    H_2(\nu^2 \hat{u}_2,\nu^3 \hat{p}_2,\nu^4 \hat{v}_2, \nu^3 \tilde{q}_2) = \nu^6 H_2(\hat{u}_2,\hat{p}_2,\hat{v}_2,\tilde{q}) \, \, \, {\rm for} \,\,  \nu \in \mathbb{R}.
\end{equation}
As a result, the system exhibits infinitely self-similar dynamics on the zero level set of $H_2$.
This scaling invariance and the attendant infinitely self-similar dynamics were first observed in the central charts for the Brusselator and van der Pol PDEs in \cite{JDKV2026,VDK2025}.
They have important consequences for the observed dynamics of spatially-periodic canards in the class of PDEs \eqref{eq:genAI-UV}.
\end{remark}

\subsubsection{A key algebraic solution}\label{subsubsec:Gamma20}

System \eqref{eq:K2-H2} has the following key algebraic solution:
\begin{equation} \label{eq:Gamma_20}
    \Gamma_{20} (\hat{y}_2) 
    = \left( 
    -\frac{\alpha_3}{12} \hat{y}_2^2, \, \,  
    -\frac{\alpha_3}{6} \hat{y}_2, \, \,  
    -\frac{\alpha_3^2}{144}\hat{y}_2^4 - \frac{\alpha_3}{6}, \, \,  
    \frac{\alpha_3}{36}\hat{y}_2^3
    \right),
\end{equation}
where we have reverted to listing the coordinates in the standard order $(\hat u_2, \hat p_2, \hat v_2, \tilde q_2)$.
The special solution $\Gamma_{20}$ can be split into two halves,
\begin{equation}
        \Gamma_0^+ = \{ \Gamma_{20}(\hat{y}_2) : \hat{y}_2 \ge 0 \} 
        \quad {\rm and} \quad 
          \Gamma_0^- = \{ \Gamma_{20}(\hat{y}_2) :  \hat{y}_2 \le  0\},
\end{equation}
with $\Gamma_{0}^+$ representing the true canard ($\Gamma_0^t$) and $\Gamma_0^-$ representing the faux canard ($\Gamma_0^f$) of $M_{\rm RFSN-II}$.
Solutions on $\Gamma_0^+$ approach $M_{\rm RFSN-II}$ as $\hat{y}_2 \to 0^+$,
and those on $\Gamma_0^-$ as $\hat{y}_2 \to 0^-$.
Both are cusp-shaped curves in the $(\hat{u}_2,\hat{q}_2)$ plane.

\begin{remark}
The algebraic solution $\Gamma_{20}$ may also be obtained in the limit of the true and faux canards, $\Gamma_t^0$ and $\Gamma_f^0$ \eqref{eq:gamma_TF}, of $M_{\rm RFS}$ as $\lambda\to 0^+$. 
The higher order terms in $r_2$ should be included to carry out this limiting analysis.
\end{remark}

\begin{remark}
In \eqref{eq:K2-coords-RFSNII}, we have scaled the parameter $\lambda$ with three powers of $\hat{r}_2$ to focus on the cusp singularity at $\hat{\lambda}_2=0$.
One may instead scale it with $\hat{r}_2^2$ to also unfold the cusp locally in the parameter.
In that case, the unperturbed system, which then includes the term $-\tfrac{1}{2}\hat{\lambda}_2$ in the fourth component, also has a special algebraic solution for all $\mathcal{O}(1)$ values of $\hat{\lambda}_2$, with $\hat{u}_2=-\tfrac{\alpha_3}{12}\hat{y}_2^2 + \tfrac{3\hat{\lambda}_2}{4\alpha_3}$.
Also, the corresponding Hamiltonian has the same scaling invariance, now with $\nu^3 \hat{\lambda}_2$ included.
\end{remark}

\subsubsection{The geometric role of the algebraic solution \texorpdfstring{$\Gamma_{20}$}{Lg} and the different types of heteroclinics on the hemisphere
}
\label{subsubsec:geomrole}

In this section, we describe the geometric role of the algebraic solution $\Gamma_{20}$. 
On the level set $H_2=0$, the system \eqref{eq:K2-H2} may be reduced to a third-order system by setting $\hat{v}_2 = \tfrac{1}{2\hat{u}_2} (\hat{p}_2^2 +\alpha_3 \tilde{q}_2^2) - \tfrac13 \hat{u}_2^2$.
Hence, after transforming the independent variable by $\hat{u}_2 \tfrac{d}{d\hat{y}_2}= \tfrac{d}{d\hat{\xi}_2}$ to desingularize the vector field, we find that the unperturbed system is 
\begin{equation}\label{eq:systemonH2zero}
    \begin{split}
        \frac{d\hat{u}_2}{d\hat{\xi}_2} &= \hat{u}_2 \hat{p}_2 \\
        \frac{d\hat{p}_2}{d\hat{\xi}_2} &= \tfrac12 (\hat{p}_2^2 + \alpha_3 \tilde{q}_2^2) + \tfrac23 \hat{u}_2^3 \\
        \frac{d\tilde{q}_2}{d\hat{\xi}_2} &= - \hat{u}_2^2.
    \end{split}
\end{equation}
It has the reversibility symmetry $(\hat{u}_2,\hat{p}_2, \tilde{q}_2,\hat{\xi}_2) \to (\hat{u}_2, -\hat{p}_2, -\tilde{q}_2, -\hat{\xi}_2)$.

This system possesses lines of equilibria
\begin{equation}
\label{eq:curvesL}
    \mathcal{L}_\pm = \left\{  \hat{u}_2 = 0,
\, \, \, 
\hat{p}_2 = \pm \sqrt{- \alpha_3}\, \tilde{q}_2  \right \}.
\end{equation}
At each point on $\mathcal{L}_+$, 
the spectrum is $\sigma_{s/u} = \{ \sqrt{-\alpha_3}\, \tilde{q}_2, \sqrt{-\alpha_3}\,\tilde{q}_2 \}$ and $\sigma_c= \{ 0 \}$, with eigenspaces
\[ 
\mathbb{E}^{s/u}(\mathcal L_+) = {\rm span} \, \left\{ \begin{bmatrix} 1 \\ 0 \\ 0 \end{bmatrix}, \begin{bmatrix} 0 \\ 1 \\ 0 \end{bmatrix} \right\} \quad \text{ and } \quad
\mathbb{E}^c (\mathcal L_+) = {\rm span} \, \begin{bmatrix} 0 \\ \sqrt{-\alpha_3} \\ 1 \end{bmatrix}. \]
Similarly, at each point on $\mathcal{L}_-$,
the spectrum is $\sigma_{s/u} = \{ -\sqrt{-\alpha_3}\, \tilde{q}_2, -\sqrt{-\alpha_3} \, \tilde{q}_2 \}$ and $\sigma_c= \{ 0 \}$, with eigenspaces
\[ 
\mathbb{E}^{s/u}(\mathcal L_-) = {\rm span} \, \left\{ \begin{bmatrix} 1 \\ 0 \\ 0 \end{bmatrix}, \begin{bmatrix} 0 \\ 1 \\ 0 \end{bmatrix} \right\}
\quad \text{ and } \quad
\mathbb{E}^c(\mathcal L_-) = {\rm span} \, \begin{bmatrix} 0 \\ \sqrt{-\alpha_3} \\ - 1 \end{bmatrix}.
\]
Hence, in the $(\hat{u}_2,\hat{p}_2)$ plane, the lines of fixed points are improper nodes,
with unstable halves 
\[
\mathcal{L}_+^u = \{ \mathcal{L}_+ :  \, \, \tilde{q}_2 > 0 \} 
\quad \text{ and } \quad
\mathcal{L}_-^u = \{ \mathcal{L}_- :  \, \, \tilde{q}_2 < 0 \},
\]
and stable halves
\[ 
\mathcal{L}_+^s = \{ \mathcal{L}_+ :  \, \, \tilde{q}_2 < 0 \} \quad \text{ and } \quad
\mathcal{L}_-^s = \{ \mathcal{L}_- :  \, \, \tilde{q}_2 > 0 \}.
\]

\begin{proposition}\label{prop:roleofGamma20}
Restricted to $\{ \bar{u}_2  \ge 0 \}$, system \eqref{eq:systemonH2zero} possesses three classes of heteroclinics:
\begin{itemize}
\setlength{\itemsep}{0pt}
    \item Class 1: solutions emanate from a point on $\mathcal{L}_+^u$ and terminate at a point on $\mathcal{L}_+^s$;
     \item Class 2: solutions emanate from a point on $\mathcal{L}_+^u$ and terminate at a point on $\mathcal{L}_-^s$;
     \item Class 3: solutions emanate from a point on $\mathcal{L}_-^u$ and terminate at a point on $\mathcal{L}_+^s$.
     \end{itemize}
     \noindent 
In addition, the stable manifold, $W^s(\Gamma_{0}^+)$, of $\Gamma_{0}^+$ is the phase-space boundary that divides between class 1 and class 2 heteroclinics.
The unstable manifold, $W^u(\Gamma_{0}^-)$, of $\Gamma_{0}^-$ is the phase-space boundary that divides between class 1 and class 3 heteroclinics. 
\end{proposition}
\noindent 
This proposition is proven in Appendix~\ref{app:secondgeodesing}, and it is illustrated in Figure~\ref{fig:secondblowup}. 
The proof follows closely that of Proposition 6.2 in \cite{VDK2025}, as given in Appendix B of \cite{VDK2025}, where the existence of three similar classes of heteroclinics is established in the rescaling chart of the analysis of Turing canards in the van der Pol PDE, which is a special case of the general PDEs \eqref{eq:genAI-UV}.

\begin{figure}[ht!]
    \centering
    \includegraphics[width=3.5in]{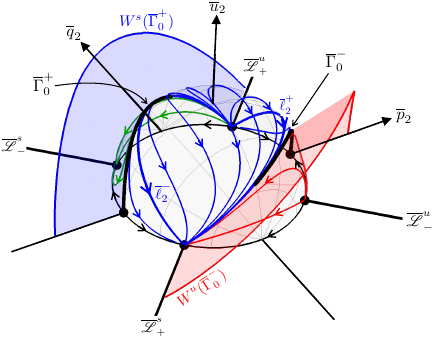}
    \caption{The three classes of heteroclinics of system \eqref{eq:systemonH2zero}. The nilpotent equilibrium at the intersection of $\mathcal L_{+}$ and $\mathcal L_-$ is blown-up to a hemisphere $\mathbb S_2 = \left\{ (\overline u_2,\overline p_2, \overline q_2): \overline u_2^2+\overline p_2^2+\overline q_2^2 = 1 \right\}$ by the mapping $(u_2,p_2,q_2) \mapsto (\rho_2 \overline u_2, \rho_2 \overline p_2, \rho_2 \overline q_2)$, where $\rho_2$ measures the radial distance from the surface of $\mathbb S_2$. The dynamics on the hemisphere are then carefully studied in the half-space $\left\{ (\overline u_2,\overline p_2, \overline q_2): \overline u_2 \geq 0 \right\}$ 
    (see Appendix~\eqref{eq:appDdesing}.)
    The lines of equilibria $\overline{\mathcal{L}}^u_{\pm}$ intersect the hemisphere on the equator $\{ (\overline u_2, \overline p_2, \overline q_2) \in \mathbb S_2 : \overline{u}_2 = 0\}$ at unstable nodes. 
    Similarly, the lines $\overline{\mathcal{L}}^s_{\pm}$ intersect the hemisphere on the equator at stable nodes. 
    The class 1, 2, and 3 heteroclinics are shown as blue, green, and red trajectories, respectively, on $\mathbb S_2$. 
    The pre-image $\overline{\Gamma}_0^+$ of the algebraic solution $\Gamma_0^t$ is the heteroclinic orbit that joins the fixed point $\overline{\mathcal{L}}_+^u \cap \mathbb S_2$ to the fixed point $(0,-1,0)$. 
    Similarly, the pre-image $\overline{\Gamma}_0^-$ of the algebraic solution $\Gamma_0^f$ is the heteroclinic orbit that joins the fixed point $(0,1,0)$ to the fixed point $\overline{\mathcal{L}}_+^s \cap \mathbb S_2$. 
    Moreover, the stable manifold $W^s(\overline{\Gamma}_0^+)$ (resp. unstable manifold $W^u(\overline{\Gamma}_0^+)$) of the fixed point $(0,-1,0)$ (resp. $(0,1,0)$) is the separatrix that divides the class 1 heteroclinics from the class 2 heteroclinics (resp. the class 1 from the class 3 heteroclinics).}
    \label{fig:secondblowup}
\end{figure}

\subsection{The entry/exit chart for \texorpdfstring{$M_{\rm RFSN-II}$}{Lg}}
\label{sec:K1-RFSNII}
In the entry/exit chart $\hat{K}_1$ where $\hat{u}_1\equiv 1$, the blow-up transformation is
\begin{equation} \label{eq:K1-coords-RFSNII}
\hat{K}_1: \, \, 
  u = \hat{r}_1^2, \, \, \, \, \, 
  p = \hat{r}_1^3 \hat{p}_1, \, \, \, \,  \,
  v = \hat{r}_1^4 \hat{v}_1, \, \, \, \,  \, 
  q = \hat{r}_1^3 \hat{q}_1, \, \,  \, \, \, 
  \delta=\hat{r}_1^2 \hat{\delta}_1, \, \, \, \, \,
  \lambda=\hat{r}_1^3 \hat{\lambda}_1.
\end{equation}
We substitute this coordinate change into \eqref{eq:y-ODE} with $\delta_y=0$ and $\lambda_y=0$ appended, and we 
introduce the variable $\hat{y}_1$ defined 
by $\frac{d}{d\hat{y}_1} = \frac{1}{\hat{r}_1} \frac{d}{dy}$.
The governing equations in $\hat{K}_1$ are
\begin{equation} \label{eq:K1-RFSNII}
\begin{split}
    \frac{d\hat{r}_1}{d\hat{y}_1} &= \tfrac{1}{2} \hat{r}_1 \hat{p}_1, \\ 
    \frac{d\hat{p}_1}{d\hat{y}_1} &= \hat{v}_1 + 1 - \tfrac{3}{2} \hat{p}_1^2 + \hat{r}_1^2 \hat{h}_1(\hat{v}_1,\hat{r}_1), \\
    \frac{d\hat{v}_1}{d\hat{y}_1} &= - 2 \hat{p}_1 \hat{v}_1 + \hat{\delta}_1 \hat{q}_1 , \\
    \frac{d\hat{q}_1}{d\hat{y}_1} &= -\tfrac{3}{2} \hat{p}_1 \hat{q}_1 + \hat{\delta}_1 \left( \alpha_3 -\tfrac{1}{2}\hat{r}_1\hat{\lambda}_1 + \hat{r}_1^2 \hat{m}_1(\hat{v}_1,\hat{r}_1)\right), \\
    \frac{d\hat{\delta}_1}{d\hat{y}_1} &= - \hat{\delta}_1 \hat{p}_1, \\
    \frac{d\hat{\lambda}_1}{d\hat{y}_1} &= -\tfrac{3}{2}\hat{\lambda}_1 \hat{p}_1,
\end{split}
\end{equation}
where $\hat{h}_1= \alpha_1 \hat{v}_1 + \alpha_2 + \mathcal{O}(\hat{r}_1^2)$
and $\hat{m}_1=  \alpha_4 \hat{v}_1 + \alpha_5 + \mathcal{O}(\hat{r}_1^2)$.
We establish

\begin{lemma}\label{lem-K1}
System \eqref{eq:K1-RFSNII} possesses
\begin{enumerate}
    \item A pair of isolated, symmetric equilibria, 
\begin{align}\label{eq:def-Epm-RFSNII}
    E_{\pm} = \left\{ \hat{r}_{1}=0,\, \, \, \hat{p}_{1}=\pm \rho, \, \, \, 
    \hat{v}_{1}=0,\, \, \hat{q}_{1}=0, \, \, \, \hat{\delta}_{1}=0, \, \, \,
    \hat{\lambda}_{1}=0 \right\},
\end{align}
where $\rho = \sqrt{\tfrac{2}{3}}$.
These equilibria lie on the invariant line  
\begin{align}\label{eq:def-I}
    I=&\left\{\hat{r}_{1}=0,\, \, \, \hat{p}_{1}\in\R, \, \, \, 
    \hat{v}_{1}=0, \, \, \, \hat{q}_{1}=0, \, \, \, \hat{\delta}_{1}=0,
    \, \, \, \hat{\lambda}_{1}=0\right\}.
\end{align}
The stable manifold of $E_{+}$ is five-dimensional, and its unstable manifold is one-dimensional. 
The stable manifold of $E_{-}$ is one-dimensional, and its unstable manifold is five-dimensional.
    \item A three-dimensional manifold of equilibria, 
    \begin{equation}\label{eq:def-M1}
        \mathcal{M}_{1}
        =\Big\{\hat{r}_{1}\ge 0, \, \, \, 
        \hat{p}_{1}=0, \, \, \, 
        \hat{v}_{1}=-1+(\alpha_1-\alpha_2) \hat r_1^2 + \mathcal O(\hat r_1^4), \, \, \, 
        \hat{q}_{1}\in\mathbb{R}, \, \, \, 
        \hat{\delta}_{1}=0, \, \, \, 
        \hat{\lambda}_{1}\in\mathbb{R}
        \Big\}.
    \end{equation}
    This manifold corresponds to $S_0^s$, the saddle sheet of the critical manifold on $u>0$, expressed in chart $\hat{K}_1$.
    Over each point in $\mathcal{M}_{1}$, there exist one-dimensional stable and unstable fibers.
    There is also a one-dimensional neutrally stable fiber containing sub-exponential dynamics.
    Also, $\mathcal{M}_1$ contains the line of saddle equilibria
\begin{align}\label{eq:def-ell1}
    \ell_{1}=&\left\{\hat{r}_{1}=0, \, \, \hat{p}_{1}=0, \, \,
    \hat{v}_{1}=-1, \, \, \hat{q}_{1}\in\R, \, \, 
    \hat{\delta}_{1}=0, \, \, \hat{\lambda}_{1}=0\right\}.
\end{align}
\end{enumerate}
\end{lemma} 

\noindent 
Lemma~\ref{lem-K1} shows that the blow-up transformation \eqref{eq:K1-coords-RFSNII} splits the RFSN-II point into a pair of hyperbolic fixed points, $E_{\pm}$, and a line $\ell_1$ of saddles.
The slow manifold is transformed to the three-dimensional manifold $\mathcal M_1$ of fixed points (with $\ell_1$ corresponding to the saddle sheet $S_s^0$).

\bigskip
\begin{proof} 
The equilibria of \eqref{eq:K1-RFSNII} are obtained by direct calculation. 
We find five distinct sets of equilibria: 
a pair of isolated points $E_{\pm}$ \eqref{eq:def-Epm-RFSNII}, a line $\ell_1$ \eqref{eq:def-ell1}, a two-dimensional surface $\mathcal S_1$ \eqref{eq:def-S1}, and a three-dimensional manifold $\mathcal M_1$ \eqref{eq:def-M1}. 
Below, we list their spectra and linear subspaces.

At the hyperbolic saddle $E_+$,
the stable and unstable spectra are 
\[ \sigma_s(E_+) = \left\{ -3\rho,-2\rho,-\tfrac{3}{2}\rho,-\tfrac{3}{2}\rho,-\rho \right\} \quad  \text{ and } \quad \sigma_u(E_+) = \left\{ \tfrac{1}{2}\rho \right\}. \] 
The associated stable and unstable subspaces are given by
\[ \mathbb E^s (E_+) = 
\operatorname{span} \left\{  
\begin{bmatrix} 0 \\ 1 \\ 0 \\ 0 \\ 0 \\ 0 \end{bmatrix}, \,\, 
\begin{bmatrix} 0 \\ 1 \\ \rho \\ 0 \\ 0 \\ 0 \end{bmatrix}, \,\,
\begin{bmatrix} 0 \\ 0 \\ 0 \\ 1 \\ 0 \\ 0 \end{bmatrix}, \,\,
\begin{bmatrix} 0 \\ 0 \\ 0 \\ 0 \\ 0 \\ 1 \end{bmatrix}, \,\, 
\begin{bmatrix} 0 \\ 0 \\ 0 \\ 2\alpha_3 \\ \rho \\ 0 \end{bmatrix}
\right\}
\quad \text{ and } \quad 
\mathbb E^u(E_+) = \operatorname{span} \begin{bmatrix} 1 \\ 0 \\ 0 \\ 0 \\ 0 \\ 0 \end{bmatrix}.
 \]
Similarly, at the hypberbolic equilibrium $E_-$, the stable and unstable spectra are 
\[ \sigma_s(E_-) = \left\{ -\tfrac{1}{2}\rho \right\} \quad  \text{ and } \quad \sigma_u(E_-) = \left\{ \rho,\tfrac{3}{2}\rho,\tfrac{3}{2}\rho,2\rho,3\rho \right\}. \] 
The corresponding linear subspaces are 
\[
\mathbb E^s (E_-) = \operatorname{span} \begin{bmatrix} 1 \\ 0 \\ 0 \\ 0 \\ 0 \\ 0 \end{bmatrix} 
\quad \text{ and } \quad 
\mathbb E^u(E_-) = 
\operatorname{span} \left\{  
\begin{bmatrix} 0 \\ 0 \\ 0 \\ -2 \alpha_3 \\ \rho \\ 0 \end{bmatrix}, \,\,
\begin{bmatrix} 0 \\ 0 \\ 0 \\ 0 \\ 0 \\ 1 \end{bmatrix}, \,\, 
\begin{bmatrix} 0 \\ 0 \\ 0 \\ 1 \\ 0 \\ 0 \end{bmatrix}, \,\, 
\begin{bmatrix} 0 \\ -1 \\ \rho \\ 0 \\ 0 \\ 0 \end{bmatrix}, \,\,
\begin{bmatrix} 0 \\ 1 \\ 0 \\ 0 \\ 0 \\ 0 \end{bmatrix}
\right\}.
\]
Since the eigenvalues of $E_{\pm}$ are real and bounded away from zero, the existence and dimensions of the stable and unstable manifolds of $E_{\pm}$ follow from standard invariant manifold theory \cite{HPS1977}. 
Also, $E_{\pm}$ are connected by a heteroclinic orbit along $I$ \eqref{eq:def-I}.
This completes the proof of Part 1.

At each point on the line $\ell_1$,
the stable, unstable, and center spectra are 
\[ \sigma_s(\ell_1) = \{ -\sqrt 2 \}, \quad \sigma_u(\ell_1) = \{ \sqrt 2 \}, \quad \text{ and } \quad \sigma_c(\ell_1) = \left\{ 0,0,0,0 \right\}. \]
The stable, unstable, and (generalized) center eigenspaces are 
\[
\mathbb E^s(\ell_1) = \operatorname{span} \begin{bmatrix} 0 \\ \tfrac{2\sqrt 2}{3} \\ -\tfrac{4}{3} \\ q_1 \\ 0 \\ 0 \end{bmatrix}, \,
\mathbb E^u(\ell_1) = \operatorname{span} \begin{bmatrix} 0 \\ \tfrac{2\sqrt 2}{3} \\ \tfrac{4}{3} \\ -q_1 \\ 0 \\ 0 \end{bmatrix}, \, \text{ and } \,
\mathbb E^c(\ell_1) = \operatorname{span} \left\{ 
\begin{bmatrix} 1 \\ 0 \\ 0 \\ 0 \\ 0 \\ 0 \end{bmatrix}, \,
\begin{bmatrix} 0 \\ 0 \\ 0 \\ 1 \\ 0 \\ 0 \end{bmatrix}, \, 
\begin{bmatrix} 0 \\ 0 \\ 0 \\ 0 \\ 0 \\ 1 \end{bmatrix}, \,
\begin{bmatrix} 0 \\ -q_1 \\ 0 \\ 0 \\ 2 \\ 0 \end{bmatrix}
\right\},
\]
where the last element of $\mathbb E^c(\ell_1)$ is a generalized eigenvector.

Lastly, at each point on $\mathcal{M}_1$,
the stable, unstable, and center spectra are
\[ \sigma_s(\mathcal M_1) = \left\{ -\sqrt 2 + \mathcal O(\hat r_1^2) \right\}, \, \, \, \sigma_u(\mathcal M_1) = \left\{ \sqrt 2 + \mathcal O(\hat r_1^2) \right\}, \, \, \, \text{ and } \, \, \,  \sigma_c(\mathcal M_1) = \{ 0,0,0,0 \}. \]
The corresponding stable, unstable, and (generalized) center subspaces are
\begin{equation*}
  \begin{split}
    \mathbb E^s(\mathcal M_1) &= \operatorname{span} \begin{bmatrix} -\tfrac{1}{2} \hat r_1 + \mathcal O(\hat r_1^4) \\ \sqrt 2 + \mathcal O(\hat r_1^2) \\ -2 + \mathcal O(\hat r_1^2) \\ \tfrac{3}{2} \hat q_1 \\ 0 \\ \tfrac{3}{2} \hat \lambda_1 \end{bmatrix}, \quad
    \mathbb E^u(\mathcal M_1) = \operatorname{span} \begin{bmatrix} -\tfrac{1}{2} \hat r_1 + \mathcal O(\hat r_1^4) \\ -\sqrt 2 + \mathcal O(\hat r_1^2) \\ -2 + \mathcal O(\hat r_1^2) \\ \tfrac{3}{2} \hat q_1 \\ 0 \\ \tfrac{3}{2} \hat \lambda_1 \end{bmatrix}, \\
    \mathbb E^c(\mathcal M_1) &= \operatorname{span} \left\{ 
    	\begin{bmatrix} 0 \\ 0 \\ 0 \\ 0 \\ 0 \\ 1 \end{bmatrix}, 
	\begin{bmatrix} 0 \\ 0 \\ 0 \\ 1 \\ 0 \\ 0 \end{bmatrix}, 
	\begin{bmatrix} \tfrac{1}{2(\alpha_1-\alpha_2)} + \mathcal O(\hat r_1^2) \\ 0 \\ \hat r_1 \\ 0 \\ 0 \\ 0 \end{bmatrix}, 
	\begin{bmatrix} 0 \\ -\hat q_1 + \mathcal O(\hat r_1^2) \\ 0 \\ 0 \\ 2 + \mathcal O(\hat r_1^2) \\ 0 \end{bmatrix}
    \right\}.
  \end{split}
\end{equation*}
The last element of $\mathbb E^c(\mathcal M_1)$ is a generalized eigenvector. 
It becomes a proper eigenvector at two points (see ${\bf p}_{\pm}$ \eqref{eq:p-pm-RFSNII} below). 
The statements about the stable, unstable, and center manifolds of $\mathcal M_1$ follow from standard invariant manifold theory \cite{C1981,HPS1977}. 
This completes the proof of Part 2.
\end{proof} 

\begin{remark}
System \eqref{eq:K1-RFSNII} also has a two-dimensional surface of saddle fixed points,
\begin{equation}\label{eq:def-S1}
\mathcal S_1 = \left\{ \hat r_1 \geq 0,  \hat p_1 = 0,  \hat v_1 = -1 + (\alpha_1-\alpha_2)\hat r_1^2 + \mathcal O(\hat r_1^4), \hat q_1 = 0,  \hat \delta_1 \in \mathbb R,  \hat \lambda_1 = \tfrac{2\alpha_3}{\hat{r}_1} + \mathcal O(\hat r_1) \right\}.
\end{equation} 
\end{remark}
\noindent 
It corresponds to the ordinary singularity in the blown-up locus.

\subsection{The coordinate change between charts and the limits of
\texorpdfstring{$\Gamma_{20}$}{Lg} in chart \texorpdfstring{$\hat{K}_1$}{Lg}}
\label{sec:limits-of-Gamma0-in-K1}

In this section, we state the coordinate change between charts $\hat{K}_2$ and $\hat{K}_1$, and we use it to show that the algebraic solution $\Gamma_{20}(\hat{y}_2)$ \eqref{eq:Gamma_20} limits on key equilibria on the equator of $\mathbb{S}^5$ in $\hat{K}_1$.
Also, we show that the tangent vector limits on key components of the center subspaces at the equilibria. 

The coordinate change from the central chart $\hat{K}_2$ to the entry/exit chart $\hat{K}_1$ is
\begin{equation} \label{eq:kappa21}
\hat{\kappa}_{21}: 
  (\hat{r}_1, \hat{p}_1, \hat{v}_1, \tilde{q}_1, \hat{\delta}_1, \hat{\lambda}_1)
  =\left(\hat{r}_2 \sqrt{\hat{u}_2}, \, \, 
    \frac{\hat{p}_2}{\hat{u}_2^{3/2}}, \, \, 
    \frac{\hat{v}_2}{\hat{u}_2^2},  \, \,
    \frac{\tilde{q}_2}{\hat{u}_2^{3/2}}, \, \,
    \frac{1}{\hat{u}_2}, \, \,
    \frac{\hat{\lambda}_2}{\hat{u}_2^{3/2}}\right) \, \, {\rm for} \, \, \, \hat{u}_2>0.
\end{equation}
Hence, on the invariant set $\{ \hat{r}_1=0 \}$ in $\hat{K}_1$, the algebraic solution $\Gamma_{20}(\hat{y}_2)$ is
\begin{equation}\label{eq:limitGamma20}
   \begin{split}
       \kappa_{21}(\Gamma_{20}(\hat{y}_2)) = \left\{ 0, 
       -\frac{4\sqrt{3}\alpha_3 \hat{y}_2}{ z_2^{3/2}}, 
       -1-\frac{\alpha_3}{24 z_2^2}, 
       \frac{-2 \alpha_3^2\hat{y}_2^3}{\sqrt{3}z_2^{3/2}}, 
       \frac{12}{z_2}, \frac{24 \sqrt{3} \hat \lambda_2}{z_2^{3/2}} \right\}, 
   \end{split} 
\end{equation}
where $z_2 = -\alpha_3 \hat{y}_2^2$
and we recall that we work with $\alpha_3 < 0$.
The limits of this solution are the following points on the equator of $\mathbb{S}^5$:
\begin{equation}\label{eq:p-pm-RFSNII}
   {\bf p}_\pm = 
    \lim_{\hat{y}_2 \to \pm \infty} 
    (\hat{r}_1,\hat{p}_1,\hat{v}_1,\hat{q}_1,\hat{\delta}_1,\hat{\lambda}_1)
    = \left(0,0,-1,\mp\frac{2\sqrt{-\alpha_3}}{\sqrt{3}},0,0\right) \in \ell_1.
\end{equation}
Also, the unit tangent vector has the following limits:
\begin{equation}\label{eq:limit-tanvectors}
\lim_{\hat{y}_2 \to \pm \infty} 
\frac{\frac{d}{d\hat{y}_2}\kappa_{21}(\Gamma_{20})}{\Vert \frac{d}{d\hat{y}_2}\kappa_{21}(\Gamma_{20})\Vert}
= \left(0, 
\frac{\alpha_3}{\sqrt{-\alpha_3(3-\alpha_3)}}, 
0, 0,
\mp \sqrt{\frac{3}{3 - \alpha_3}},  
0 \right) \in \mathbb E^c(\ell_1).
\end{equation}
It is a multiple of the generalized eigenvector.
Therefore, we have shown that the special solution limits on the points ${\bf p}_\pm \in \ell_1$ tangent to $W^{c}(\ell_1)$ at these points.

\subsection{Existence of the perturbed true and faux canards 
\texorpdfstring{$\Gamma^\delta_t$ and $\Gamma^\delta_f$}{Lg}}
\label{sec:persistence-RFSNII}

In this section, we prove

\begin{lemma}\label{lem-persistenttrue+fauxcanards-RFSNII}
There exists a $\delta_0>0$ sufficiently small such that for each $0< \delta<\delta_0$ and each $\lambda\ge 0$ sufficiently small,
the singular true canard $\Gamma_{t}^0$ and singular faux canard $\Gamma_{f}^0$ of $M_{\rm RFSN-II}$ perturb to maximal true and faux canard solutions,  which we label $\Gamma_{t}^{\delta}$ and $\Gamma_{f}^{\delta}$ respectively, of \eqref{eq:y-ODE}.
Also, for positive values of $\lambda > \mathcal{O}(\delta)$, the connection between $\Gamma_t^\delta$ and $\Gamma_f^\delta$ persists for $0<\delta<\delta_0$.
\end{lemma}

\begin{proof} 
In chart $\hat{K}_2$, the left-half ($-\infty < \hat{y}_2<0$) of the key algebraic solution $\Gamma_{20}(\hat{y}_2)$ corresponds to the singular true canard $\Gamma_{t}^0$ and the right-half ($0< \hat{y}_2 < \infty$) to the singular faux canard $\Gamma_{f}^0$.
By \eqref{eq:limitGamma20},
$\Gamma_{20}(\hat{y}_2)$ is backward asymptotic to the equilibrium ${\bf p}_{-} \in \ell_1 \subset \mathcal{M}_1$, and forward asymptotic to ${\bf p}_{+} \in \ell_1 \subset \mathcal{M}_1$.
Both points lie on the equator of the blown-up hemisphere $\mathbb{S}^5$.

The analysis of the asymptotics of the tangent vector in \eqref{eq:limit-tanvectors} shows that $\Gamma_{20}(\hat{y}_2)$ emanates from the four-dimensional normally attracting $W^{c}\left( \bf{p}_- \right)$ and terminates in the four-dimensional normally repelling $W^{c}\left( \bf{p}_+ \right)$ in the $\left\{ \hat{r}_1=0 \right\}$ hyperplane of chart $\hat{K}_1$.
The asymptotic behavior is decomposed into components tangent to 
$\mathcal{M}_{1}\rvert_{\hat{r}_{1}=0}$ and in the neutrally stable fiber over the limit points.
Hence, the only solution of the variational equation about $\Gamma_{20}$ that is bounded in both limits as $\hat{y}_2 \to \pm \infty$ is $\tfrac{d}{d\hat{y}_2}\kappa_{21}(\Gamma_{20})$.
All other solutions diverge from ${\bf p}_-$ in backward time and from ${\bf p}_+$ in forward time. 
Therefore, in the invariant set $\{ \hat r_2=0 \}$, $\Gamma_{20}(\hat{y}_2)$ lies in the transverse intersection of $W^{c}\left( \bf{p}_- \right)$ and $W^{c}\left( \bf{p}_+ \right)$, which is a two-dimensional intersection manifold.

Next, because this intersection is transverse on $\{ \hat r_2 = 0 \}$ for each fixed value of $\hat \lambda_2$, we know from regular perturbation theory that the intersection of $W^{c}({\bf p}_-)$ and $W^{c}({\bf p}_+)$ is also transverse in chart $\hat K_2$ for each $\hat r_2>0$ sufficiently small. 
By tracking this intersection into chart $\hat K_1$, we see that the intersection remains transverse.
Moreover, the intersection manifold $W^{c}\left( \bf{p}_- \right) \cap W^{c}\left( \bf{p}_+\right)$ is foliated by the parameter $\mathcal{\hat{\lambda}}_2$ in chart $\hat{K}_2$.
Hence, by tracking the solutions of this foliation into chart $\hat{K}_1$, we see that for each fixed $\hat{r}_2$ 
this foliation carries over into a nontrivial foliation by $\hat{r}_1^2 {\hat{\lambda}}_1 = {\rm constant}$ in $\hat{K}_1$
(where we recall that ${\hat{\lambda}}_1 \hat{r}_1^2 = {\hat{\lambda}}_2 \hat{r}_2^2$).
Therefore, for each fixed $\mathcal{\hat{\lambda}}_2$ (or $\hat{r}_1^2 \mathcal{\hat{\lambda}}_1$ with $\hat{r}_1$ sufficiently small), $\Gamma_{20}(\hat{y}_2)$ is a curve, and it corresponds to the maximal canard solution. 
This shows that the singular canards $\Gamma_{t}^0$ and $\Gamma_{f}^0$ perturb to maximal canard solutions $\Gamma_{t}^{\delta}$ and $\Gamma_{f}^{\delta}$ for all $\delta > 0$ sufficiently small, as stated in the lemma.

To establish the persistence of the connection, we recall that the last part of the proof of Lemma~\ref{lem-persistence-true+fauxcanards-RFS} establishes the persistence of the connection between the maximal true and faux canards in the regime of $\mathcal O(1)$ $\lambda>0$, where the folded singularity is an RFS.
The analysis performed there may be extended down to the regime in which $\lambda=\mathcal{O}(\delta)$ by using a $\delta$-dependent rescaling of $q$.
\end{proof}

\section{Proof of Theorem~\texorpdfstring{\ref{thm:1}}{Lg}}\label{sec:proof-theorem1}

In this brief section, we draw direct conclusions from the results of the analysis in Sections~\ref{sec:geodesing} and \ref{sec:RFSNII-desing} to prove Theorem~\ref{thm:1}. 

\medskip
\begin{proof} 
For all $\mathcal{O}(1)$ values of $\lambda> 0$, Lemma~\ref{lem-persistence-true+fauxcanards-RFS} establishes that maximal true and faux canards of the folded singularity $M_{\rm RFS}$ persist in the spatial ODE system \eqref{eq:y-ODE} for all $0 < \delta \ll 1$.
Similarly, for sufficiently small values $\lambda > 0$, Lemma~\ref{lem-persistenttrue+fauxcanards-RFSNII} establishes that maximal true and faux canards of the folded singularity $M_{\rm RFSN-II}$ persist in the spatial ODE system \eqref{eq:y-ODE} for all $0 < \delta \ll 1$.

All the canard solutions that exist in the spatial ODE system \eqref{eq:y-ODE} are time-independent, spatial canard solutions of the system of PDEs \eqref{eq:genAI}, equivalently (by Proposition~\ref{prop1}) of the activator-inhibitor reaction-diffusion systems \eqref{eq:genAI-UV} with non-degenerate fold points.
\end{proof}

\section{Proof of Theorem~\texorpdfstring{\ref{thm:2}}{Lg}}\label{sec:proof-theorem2}

We construct spatially-periodic canard solutions in the four-dimensional phase space of \eqref{eq:y-ODE} for all $\mathcal{O}(1)$ values of $\lambda>0$, with $0<\delta \ll 1$, thereby proving Theorem~\ref{thm:2}.
The spatially-periodic canards are time-independent solutions of the PDE \eqref{eq:genAI}.
The variables $(u,p)$ are fast spatial variables, {\it i.e.,} can vary rapidly on short intervals, while $(v,q)$ are slow, which vary gradually in $x$.
 
In step 1, we will show that for each $\lambda>0$ the invariant manifolds $W^u(S_s^\delta)$ and $ W^s(S_s^\delta)$ intersect transversely, where we recall that the manifolds are defined in Section~\ref{sec:layerproblem+criticalmanifold+desingularizedsystem}, see especially \eqref{eq:S-delta}. 
Then, in step 2, we will construct the singular limits ($\delta=0$) of the spatially-periodic canard solutions and show that they persist for $0<\delta \ll 1$. This  will establish the theorem.

We begin by scaling the variables to identify the distinguished limit in which the rate of slow variation of $(v,q)$ is the same as the amplitude of the dominant terms of the perturbation in the fast $(u,p)$ components of the vector field.
Namely, we set
\begin{equation}\label{eq:distinguishedlimit}
u = \delta^{2/5} {\bar U}, \quad
p = \delta^{3/5} {\bar P}, \quad 
v = \delta^{4/5} {\bar V}, \quad
q = \delta^{2/5} {\bar Q}, \quad 
y = \delta^{-1/5} \xi.
\end{equation}
System~\eqref{eq:y-ODE} becomes
\begin{equation}\label{eq:72}
\begin{split}
    {\bar U}_\xi &= {\bar P}, \\
    {\bar P}_\xi &= {\bar V} + {\bar U}^2 
                    + \delta^{2/5} \left( \alpha_1 {\bar U}{\bar V} + \alpha_2 {\bar U}^3 \right)
                    + \delta^{4/5}\left( \beta_1 {\bar V}^2 + \beta_2 {\bar U}^2 {\bar V} + \beta_3 {\bar U}^4 \right) 
                    + \mathcal{O}(\delta^{6/5}), \\
    {\bar V}_\xi &= \delta^{2/5} {\bar Q}, \\      
    {\bar Q}_\xi &= \delta^{2/5} \left(-\tfrac{1}{2} \lambda + \delta^{2/5}\alpha_3 {\bar U} + \mathcal{O}(\delta^{4/5}) \right).
\end{split}
\end{equation}
Similar to \eqref{eq:union}, this system has a critical manifold
\begin{equation*}
\begin{split} 
    S^0 &= S_s^0 \cup L \cup S_c^0, \quad {\rm where} \\
    S_s^0 &= \{ \bar{U}=\sqrt{-\bar{V} \, }, \, \, \bar{P}=0,  \, \, \bar{V}<0, \, \, \bar{Q} \in \mathbb{R} \}, \\ 
    L &= \{\bar{U}=0, \, \, \bar{P}=0, \, \,\bar{V}=0, \, \, \bar{Q} \in \mathbb{R} \}, \\
    S_c^0 &= \{\bar{U} = - \sqrt{-\bar{V} \, }, \, \,  \bar{P}=0,\,\, \bar{V}<0, \, \, \bar{Q} \in \mathbb{R} \}.
    \end{split} 
\end{equation*}

It is useful to work with the new dependent variable $W= - \bar{V}$ and to rectify the saddle sheet $S_s^0$ to the $(W,\bar{Q})$ plane.
Hence, we define the new variables
\begin{equation}
    U = \bar{U} - \bar{U}_s(W), \quad
    P = \bar{P}, \quad
    W = - \bar{V}, \quad 
    Q = \bar{Q},
\end{equation}
where $\bar{U}_s(W)=\sqrt{W}$.
System \eqref{eq:72} transforms to 
\begin{equation}\label{eq:UPWQ} 
\begin{split}
    U_\xi &= P + \delta^{2/5} \tfrac{1}{2\sqrt{W}} Q, \\
    P_\xi &= 2 \sqrt{W} U + U^2  
                    + \delta^{2/5} \left( -\alpha_1 W (U + \sqrt{W}) + \alpha_2 (U + \sqrt{W})^3\right) \\
               &{\hskip1.15truein} + \delta^{4/5}\left( \beta_1 W^2 - \beta_2 W (U + \sqrt{W})^2 + \beta_3 ({\bar U} + \sqrt{W})^4 \right) 
                    + \mathcal{O}(\delta^{6/5}), \\
    W_\xi &= - \delta^{2/5} Q, \\      
    Q_\xi &= \delta^{2/5} \left(-\tfrac{1}{2} \lambda + \delta^{2/5}\alpha_3 (U + \sqrt{W})  + \mathcal O(\delta^{4/5}) \right).
\end{split}
\end{equation}
The sheets of the critical manifold are now given by
\begin{equation}\label{eq:S0-sec7}
    S_s^0 = \{ U = 0, \, \, P=0 \} 
    \hskip0.2truein {\rm and} \hskip0.2truein
    S_c^0 = \{ U=-2\sqrt{W}, \, \, P=0 \}. 
\end{equation}

The layer problem of \eqref{eq:UPWQ} is obtained by setting $\delta=0$,
\begin{equation*}\label{eq:HF} 
    \begin{split}
        U_\xi &= P, \\
        P_\xi &= 2 \sqrt{W} U + U^2,
    \end{split}
\end{equation*}
in which $W$ and $Q$ are constants.
The layer problem is Hamiltonian with 
$H_F = \tfrac{1}{2}P^2 - \sqrt{W} U^2 - \tfrac{1}{3} U^3$,
and the saddle fixed point $(0,0)$ has a homoclinic orbit.
Therefore, in the full four-dimensional phase space of 
\eqref{eq:UPWQ} with $\delta=0$, every point $(0,0,W,Q) \in S_s^0$ is connected to itself by a homoclinic orbit 
\begin{equation}\label{eq:UH}
  \begin{split}
U_H(\xi;W) &= - 3 \sqrt{W} {\rm sech}^2\left( \frac{W^{1/4}}{\sqrt{2}} \xi \right), \\
P_H(\xi;W) &= \frac{d}{d\xi} U_H(\xi) =  -3 \sqrt{2} W^{3/4} {\rm sech}^2\left( \frac{W^{1/4}}{\sqrt{2}}\xi \right) 
{\rm tanh}\left( \frac{W^{1/4}}{\sqrt{2}} \xi  \right).
  \end{split}
\end{equation}
Geometrically, the three-dimensional manifolds $W^u(S_s^0)$ and $W^s(S_s^0)$ coincide in a two-parameter family of these homoclinics.

For $0<\delta \ll 1$, the perturbed manifolds $W^u_{\rm loc}(S_s^\delta)$ and $W^s_{\rm loc}(S_s^\delta)$ no longer coincide.
Let $\Gamma^{u,s}_\delta(\xi)$ denote solutions on
$W^{u,s}_{\rm loc}(S_s^\delta)$ for $\xi \in (-\infty,0)$ and $\xi\in (0,\infty)$, respectively.
Measured along a normal vector to the union of $W^u(S_s^0)$ and $W^s(S_s^0)$, the splitting distance between the local manifolds is
\begin{equation}\label{eq:def-DeltaHF}
    \Delta H_F = 
    \int_{-\infty}^0 \left. \frac{dH_F}{d\xi} \right\vert_{\Gamma^u_\delta} d\xi 
    + \int_0^{\infty} \left. \frac{dH_F}{d\xi} \right\vert_{\Gamma^s_\delta} d\xi.
\end{equation}
Now, differentiating $H_F$ along solutions of \eqref{eq:UPWQ}, we find
\begin{equation}
    \begin{split}
\frac{dH_F}{d\xi} &= \delta^{2/5} \left( - UQ - \alpha_1 P (U + \sqrt{W} ) W + \alpha_2 P (U + \sqrt{W})^3 \right)  \\
    &\qquad + \delta^{4/5} P \left( \beta_1 W^2 - \beta_2 W (U + \sqrt{W})^2 + \beta_3 (U + \sqrt{W})^4 \right) 
    + \mathcal{O}(\delta^{6/5}).
    \end{split}
\end{equation}
Also, we use the following asymptotic expansion of the solutions 
$\Gamma_\delta^{u,s}$: 
\begin{equation*}
    \Gamma_\delta^{u,s} (\xi)  
    = \Gamma_0 (\xi;W) + \delta^{2/5} \Gamma_1^{s,u}(\xi) + \delta^{4/5} \Gamma_2^{u,s}(\xi) + \mathcal{O}(\delta^{6/5}).
\end{equation*}
Then, order by order, we derive the solutions.

At $\mathcal{O}(1)$, we find
\begin{equation*}
    \Gamma_0(\xi;W_0) = 
    \begin{bmatrix}
     U_H (\xi; W_0) \\
     \frac{dU_H}{d\xi} (\xi; W_0) \\
     W_0 \\
     Q_0 
    \end{bmatrix}.
\end{equation*}
Next, at $\mathcal{O}(\delta^{2/5})$, the system is
\begin{equation}\label{eq:delta25sys}
    \begin{split}
        U_{1_\xi} &=  P_1 + \frac{1}{2\sqrt{W
    _0}} Q_0, \\
        P_{1_\xi} &= 2 (U_0 + \sqrt{W_0}) U_1 + \frac{U_0}{\sqrt{W_0}}W_1 - \alpha_1 W_0 (U_0 + \sqrt{W_0}) + \alpha_2 (U_0 + \sqrt{W_0})^3, \\
        W_{1_\xi} &= - Q_0, \\
        Q_{1_\xi} &= - \tfrac{1}{2} \lambda.
    \end{split}
\end{equation}
The equations for $W_1$ and $Q_1$ decouple, so that $W_1(\xi) =  - Q_0 \xi + W_1(0)$ and 
$Q_1(\xi) = -\tfrac{1}{2} \lambda \xi + Q_1(0)$.
Substituting these into the first two components of the vector field, we see that $(U_1,P_1)$ satisfies an inhomogeneous linear system, with solutions
\begin{equation}
    \begin{bmatrix}
        U_1 \\
        P_1
    \end{bmatrix}
    = \Phi_1(\xi) \Phi_1(0)^{-1} 
    \begin{bmatrix}
        U_1(0) \\
        P_1(0) 
    \end{bmatrix}
    + \int_0^\xi \Phi_1 (\xi) \Phi_1(s)^{-1} F_1 (s) ds,  
\end{equation}
where $\Phi_1(\xi)$ denotes a fundamental matrix of the homogeneous problem and 
\begin{equation*}
    F_1 (\xi) = 
    \begin{bmatrix}
         \frac{1}{2\sqrt{W_0}} Q_0 \\
         \frac{U_0}{\sqrt{W_0}} W_1 - \alpha_1 W_0 (U_0 + \sqrt{W}_0) + \alpha_2 (U_0+\sqrt{W}_0)^3
    \end{bmatrix}.
\end{equation*}
Let $\chi_H(\xi;W_0)=P_H(\xi;W_0) \int_0^\xi \frac{1}{P_H^2(s;W_0)} ds$ be the second, linearly independent solution (obtained via reduction of order). 
Then,
\begin{equation}
  \Phi_1(\xi) = 
  \begin{bmatrix}
      P_H(\xi;W_0)  &  \chi_H (\xi; W_0) \\
      \tfrac{dP_H}{d\xi} (\xi; W_0) &
      \tfrac{d\chi_H}{d\xi} (\xi; W_0)
  \end{bmatrix}.
\end{equation}
Moreover, one may proceed to higher order if desired.

Calculating the integrals in \eqref{eq:def-DeltaHF}, we find the splitting distance of $W^u_{\rm loc}(S_s^\delta)$ and $W^s_{\rm loc}(S_s^\delta)$ is 
\begin{equation}\label{eq:Delta_HF}
\Delta H_F = \delta^{2/5} D_1 + \delta^{4/5} D_2 + \mathcal{O}(\delta^{6/5}).
\end{equation}
Here,
\begin{equation}
D_1 = \int_{-\infty}^\infty \left. \left( -U_0 Q_0 -\alpha_1 P_0 W_0 (U_0 + \sqrt{W_0}) +\alpha_2 P_0 (U_0+\sqrt{W_0})^3 \right) \right\vert_{\Gamma_0^u \cup \Gamma_0^s} d\xi= 3\sqrt{2} \pi W_0^{1/4} Q_0,
\end{equation}
where the contributions from the second and third terms in the integrand vanish since they are odd about $\xi=0$ (recall $P_H=\tfrac{dU_H}{d\xi}$).
Hence, $D_1$ has a simple zero at $Q_0=0$.

Also,
\begin{equation}
\begin{split} 
D_2 = \int_{-\infty}^\infty 
&\left[ -U_0Q_1 - U_1 Q_0 -\alpha_1 W_0 P_1 (U_0 + \sqrt{W_0})
+\alpha_2 P_1 (U_0 + \sqrt{W_0})^3 \right. \\
&+P_0( -\alpha_1 U_1 W_0 + 3 \alpha_2 U_1 (U_0 +\sqrt{W_0})^2) \\
&+\frac{1}{2\sqrt{ W_0}} W_1( -2\alpha_1 U_0 \sqrt{W_0} -3\alpha_1 W_0 + 3\alpha_2 (U_0+\sqrt{W_0})^2) \\
&+\left. \left. \beta_1 W_0^2 -\beta_2 W_0 (U_0+\sqrt{W_0}^2 +\beta_3(U_0 + \sqrt{W_0})^4\right] \right\vert_{\Gamma_0^u \cup \Gamma_0^s} d\xi.
\end{split}
\end{equation}
$D_2$ can be evaluated symbolically. 
It has a zero when $U_1(0)=0,$
$P_1(0)=0,$ and
$Q_1(0)=0$.
Moreover, this zero is simple with respect to $Q_1(0)$.
Therefore, $W^u(S_s^\delta)$ and $W^s(S_s^\delta)$ intersect transversely in the hyperplane $\left\{ P=0, \, \, Q=0 \right\}.$
The intersection manifold is a 2-D surface.

Next, we calculate the takeoff and touchdown curves.
These are the curves on $S_s^\delta$ that contain the base points of the one-dimensional fast unstable and stable fibers, respectively, that foliate the two-dimensional surface in which $W^u(S_s^\delta)$ and $W^s(S_s^\delta)$ intersect.

Let $I_F= \left\{ \xi \in \mathbb{R}: -L < \xi < L \right\}, $ where $L$ is chosen so that $U_{\rm H}$ and $P_{\rm H}$ are $\mathcal{O}({\rm exp}(-\frac{1}{\delta^{2/5}}))$ at $\xi= \pm L.$
That is, we set 
\begin{equation}\label{eq:L}
L= \frac{\sqrt{2}}{W_0^{1/4}} \left(    \frac{1}{2\delta^{2/5}} + \log 2 \right).
\end{equation}
Along solutions that lie in the two-dimensional intersection surface,
the change (or `fast jump') in the $Q$ coordinate from $\xi=-L$ to $\xi=L$ is given by
\begin{equation*} 
\begin{split}
\Delta Q (W_0,\delta) &= 
\int_{I_F} \left. \frac{dQ}{d\xi}\right\vert_{\Gamma_t^0 \cup \Gamma_f^0}  d\xi
= \int_{I_F} \left[ \frac{dQ_0}{d\xi} + \delta^{2/5} \frac{dQ_1}{d\xi} + \delta^{4/5} \frac{dQ_2}{d\xi} + \mathcal{O}(\delta^{6/5})\right] d\xi \\
&= \int_{-L}^L \left[ -\frac{1}{2} \delta^{2/5}\lambda 
+\delta^{4/5} \alpha_3 (U_0 +\sqrt{W_0})  +\mathcal{O}(\delta^{6/5}) \right] d\xi\\
&= -
\delta^{2/5} \lambda L  + 2\delta^{4/5}\alpha_3 \left( \sqrt{W_0} L - 3\sqrt{2} W_0^{1/4} \tanh\left( \frac{W_0^{1/4}}{\sqrt{2}}L \right)\right)  +\mathcal{O}(\delta^{6/5}).
\end{split}
\end{equation*}
We note that the factor of $\tanh$ is $\mathcal{O}(1)$.

\begin{figure}[!ht]
   \centering
   \includegraphics[width=3in]{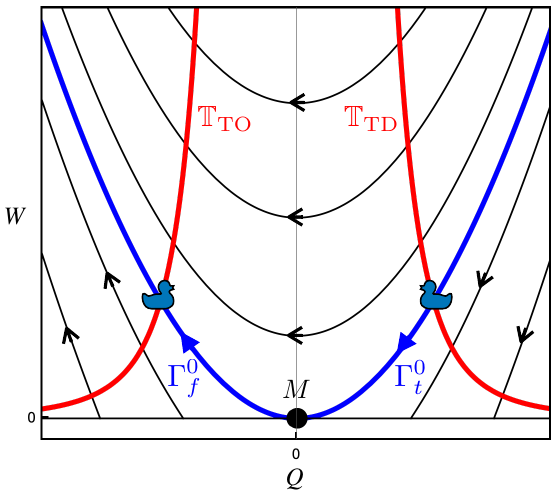}
   \caption{The (red) takeoff and touchdown curves, $\mathbb T_{\rm TO}$ and $\mathbb T_{\rm TD}$ \eqref{eq:TOTD}, on the saddle sheet $S_s^0$ \eqref{eq:S0-sec7} in the slow $(W,Q)$ plane. 
   $\mathbb T_{\rm TO}$ is the union of the base points of the fast unstable fibers on $W^u(S_s^\delta)$ in the intersection $W^u(S_s^\delta) \cap W^s(S_s^\delta)$. 
   $\mathbb T_{\rm TD}$ is the union of the base points of the fast stable fibers on $W^s(S_s^\delta)$ in that intersection. 
   The two-dimensional manifold $W^s(S_s^{\delta}) \cap W^u(S_s^{\delta})$ is foliated by homoclinic orbits that connect the base points on these curves. 
   The distance between the two duckies is $\Delta Q(W_{\rm TOTD},\delta)$, see \eqref{eq:intersectionequation-4-W} and \eqref{eq:WTOTD}.
   The reduced flow (black curves) on $S_s^0$ is also shown, as are 
   the true canard (blue, $\Gamma_t^0$) and faux canard (blue, $\Gamma_f^0$) of $M_{\rm RFS}$. 
   The canard solutions intersect $\mathbb T_{\rm TO}$ and $\mathbb T_{\rm TD}$ transversely at the takeoff and touchdown points (ducky markers).}
   \label{fig:takeofftouchdown}
\end{figure}

Now, using
\eqref{eq:L}, we find that 
to leading order 
\begin{equation}\label{eq:DeltaQ=2Psi}
\Delta Q (W_0,\delta)= 2 \Psi (W_0,\delta),
\end{equation}
where 
\begin{equation*}
\Psi(W_0,\delta) = -\frac{\lambda}{2\sqrt{2} W_0^{1/4} } 
+ \delta^{2/5} \left( \frac{1}{2} \alpha_3 \sqrt{2} W_0^{1/4} - \frac{\lambda \log 2}{\sqrt{2} W_0^{1/4} }\right) +\mathcal{O}(\delta^{4/5}).
\end{equation*}
Therefore, the takeoff and touchdown curves (red curves in Figures~\ref{fig:takeofftouchdown} and \ref{fig:pulseconstruction}) are
\begin{equation}\label{eq:TOTD}
\mathbb{T}_{\rm TO} = \left\{ (U,P,W,Q) \in S_s^\delta : Q=\Psi (W,\delta)\right\} \, \,  {\rm and} \, \, 
\mathbb{T}_{\rm TD} = \left\{  (U,P,W,Q) \in S_s^\delta : Q = - \Psi (W,\delta) \right\}.
\end{equation}

In the $(U,P,W,Q)$ coordinates, the slow flow on $S_s^0$ is given by
\begin{equation}
\begin{split}
\frac{dW}{d\zeta} &= - Q\\
\frac{dQ}{d\zeta} &= -\frac{1}{2} \lambda,
\end{split}
\end{equation}
and the singular true and faux canards are given by
\begin{equation}\label{eq:Gamma0-UPWQ}
\begin{split}
\Gamma_t^0 &= \left\{ (U,P,W,Q) =\left(0,0,\tfrac{1}{4}\lambda \zeta^2, -\tfrac{1}{2}\lambda \zeta\right), \, \, \zeta<0\right\} 
= \left\{ (U,P,W,Q) \in S_s^0 : Q = \sqrt{\lambda W} \right\}, 
\\
\Gamma_f^0 &= \left\{ (U,P,W,Q) =\left(0,0,\tfrac{1}{4}\lambda \zeta^2, -\tfrac{1}{2}\lambda \zeta\right), \, \, \zeta > 0\right\} 
= \left\{ (U,P,W,Q) \in S_s^0 : Q = - \sqrt{\lambda W} \right\}.
\end{split}
\end{equation}

Now, we show that the singular canards intersect the takeoff and touchdown curves transversely on $S_s^0$, as follows.
For each set of the system parameters $\lambda>0,$ $\alpha_1, \ldots, \alpha_5 \in \mathbb{R},$ and $0< \delta \ll 1$,
the $W$ coordinate, which will be the same for both intersection points by symmetry, is given by the unique positive root of the following equation:
\begin{equation}\label{eq:intersectionequation-4-W}
    \Psi(W,\delta) = -\sqrt{\lambda W}.
\end{equation}
The root is given by
\begin{equation}\label{eq:WTOTD}
W_{\rm TOTD} (\delta) = \tfrac{1}{4} \lambda^{2/3} + \tfrac{1}{3} ( 2 \lambda^{2/3} \log 2 - \alpha_3) \delta^{2/5} + \mathcal{O}(\delta^{4/5}).
\end{equation}
Moreover, the intersections are transverse, as may be verified by a direct calculation of the slopes of the curves at the intersection points.

\begin{figure}[!ht]
   \centering
   \includegraphics[width=3.75in]{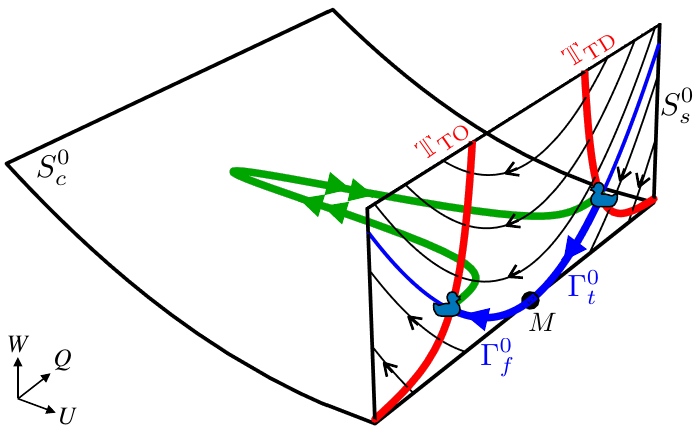}
   \caption{Construction of a singular spatially-periodic canard solution. 
   The singular orbit consists of three distinct segments.
   $\gamma_f^{\rm slow}$: the slow segment of the faux canard $\Gamma_f^0$ (blue curve) from the RFS (labelled $M$ in the figure) to the takeoff point $\Gamma_f^0 \cap \mathbb T_{\rm TO}$ (left ducky), which is at the height of $W_{\rm TOTD}$.
   $\gamma^{\rm fast}$: the fast segment (green curve) that corresponds to the singular limit in the plane $W=W_{\rm TOTD}$ of the orbit in $W^s(S_s^{\delta}) \cap W^u(S_s^{\delta})$ that connects the takeoff point $\Gamma_f^0 \cap \mathbb T_{\rm TO}$ (left ducky) to the touchdown point $\Gamma_t^0 \cap \mathbb T_{\rm TD}$ (right ducky).
   $\gamma_t^{\rm slow}$: the slow segment of the true canard $\Gamma_t^0$ (blue curve) from the touchdown point $\Gamma_t^0 \cap \mathbb T_{\rm TD}$ (right ducky) to the RFS.}
   \label{fig:pulseconstruction}
\end{figure}

With the above components, we now construct the singular spatially-periodic orbits.
We make the following definition,
\begin{definition}\label{def:singpo}
Fix a set of system parameters $\lambda>0,$ $\alpha_1, \ldots, \alpha_5 \in \mathbb{R},$ and $0< \delta \ll 1$.
Let $W_{\rm TOTD}$ \eqref{eq:WTOTD} denote the unique solution of \eqref{eq:intersectionequation-4-W} corresponding to the transverse intersections $\Gamma_f^0 \cap \mathbb{T}_{\rm TO}$ and $\Gamma_t^0 \cap \mathbb{T}_{\rm TD}$.
Let $\gamma^{\rm slow}_{\rm max}$ be the union of the segment of $\Gamma_t^0$ on $S_s^0$ from the touchdown point $\Gamma_t^0 \cap \mathbb{T}_{\rm TD}$ with $W=W_{\rm TOTD}$ to $M_{\rm RFS}$ with the segment of $\Gamma^0_f$ on $S_s^0$ from $M_{\rm RFS}$ to the takeoff point at $\Gamma_f^0 \cap \mathbb{T}_{\rm TO}$ with $W=W_{\rm TOTD}$.
In the plane $W=W_{\rm TOTD}$, let $\gamma^{\rm fast}$ be the singular limit of the heteroclinic orbit that connects the takeoff point $\Gamma_f^0 \cap \mathbb{T}_{\rm TO}$ to the touchdown point $\Gamma_t^0 \cap \mathbb{T}_{\rm TD}$.
Then, we say that the concatenated curve 
\begin{equation} 
\gamma^0_{\rm max}=\gamma^{\rm slow}_{\rm max} \cup \gamma^{\rm fast}
\end{equation}
is a {\it singular maximal spatially-periodic canard solution} of \eqref{eq:y-ODE}.
\end{definition}

Using the results of \cite{dRDR2016,ST2001}, we now establish the following lemma:

\begin{lemma} For each singular maximal spatially-periodic canard solution $\gamma^0_{\rm max}$, 
there exists a $\delta_0>0$ sufficiently small such that it persists as a maximal spatially-periodic canard solution $\gamma^\delta_{\rm max}$ of \eqref{eq:y-ODE} for all $0 < \delta < \delta_0$.
The period of a persistent maximal spatially-periodic solution is given to leading order by the time of flow (in $x$) along $\gamma^{\rm slow}_{\rm max}$ between the takeoff point $\Gamma_f^0 \cap \mathbb{T}_{\rm TO}$ and the touchdown point $\Gamma_t^0 \cap \mathbb{T}_{\rm TD}$.
\end{lemma}

\medskip
\begin{proof} 
The singular maximal spatially-periodic canard solutions constructed in the limit of $\delta = 0$ consist in alternation of a segment $\gamma^{\rm fast}$ along the singular limit of a homoclinic orbit in the transverse intersection of $W^u(S_s^\delta)$ and $W^s(S_s^\delta)$, and a segment $\gamma^{\rm slow}_{\rm max}$ along the distinguished orbits of the slow spatial system on $S_s^0$. 
In addition, the takeoff curve, which consists of the base points of the fast unstable fibers in the intersection of the invariant manifolds, transversely intersects the singular faux canard  on $S_s^0$ at the takeoff point $\mathbb{T}_{\rm TO} \cap \Gamma_f^0$.
Similarly, the touchdown curve, which consists of the base points of the fast stable fibers in the intersection, transversely intersects the singular true canard $\Gamma_t^0$ on $S_s^0$ at the touchdown point $\mathbb{T}_{\rm TD} \cap \Gamma_t^0$.
Furthermore, in neighborhoods of the takeoff and touchdown points, the flow on $S_s^0$ is transverse to $\mathbb{T}_{\rm TO}$ and $\mathbb{T}_{\rm TD}$, respectively.
Therefore, with one small modification, the hypotheses of Theorem 2.11 in \cite{dRDR2016} and Theorem 2 (with $p=2$) in \cite{ST2001} are satisfied, and each theorem directly gives the desired persistence result.

The need for the small modification arises from the fact that in the analyses of \cite{dRDR2016,ST2001} the slow segments of the singular periodic orbits  lie entirely on normally hyperbolic critical manifolds,
whereas here the slow segment $\gamma^{\rm slow}_{\rm max}$ contains one point --the folded singularity-- on $L$ at the boundary of the saddle sheet $S_s^0$ of the critical manifold, where normal hyperbolicity is lost.
However, this point lies outside of neighborhoods of the takeoff and touchdown points, and hence none of the key ingredients in the proofs are affected. 
In a small neighborhood of the folded singularity, one may use the method of geometric desingularization to blow up the vector field and extend the analysis used to track the invariant manifolds along the slow segments in Section 2.3 of \cite{dRDR2016} and Sections 3-5 of \cite{ST2001}.
We omit the details.

The statement about the periods of the persistent maximal canards follows from the fact that the length of the interval of $x$ values corresponding to $\gamma^{\rm slow}_{\rm max}$ is $\mathcal{O}(1)$ whereas the length of the interval corresponding to $\gamma^{\rm fast}$ is $o(1)$.
Hence, to leading order, the spatial period is exactly the length of the interval corresponding to $\gamma^{\rm slow}_{\rm max}$, and the length of the spatial interval over which the persistent maximal canard lies outside a neighborhood of $S_s^\delta$ is a higher order correction.
\end{proof} 

\medskip 
\begin{lemma}  
For each $0<\delta \ll 1$ and each persistent maximal spatially-periodic canard $\gamma^\delta_{\rm max}$, there is a one-parameter family of spatially-periodic canards with shorter periods that lie in the neighborhood of the maximal canard and that have spatial segments of $\mathcal{O}(1)$ in length near $S_s^\delta$ in phase space but not along $\Gamma_t^\delta$ and $\Gamma_f^\delta$.
\end{lemma}

\medskip
\begin{proof} 
For each $\mathcal{O}(1)$ value of $W_1 > W_{\rm TOTD}$,
one can construct a singular (non-maximal) spatially-periodic orbit $\gamma^0_{W_1}$, as follows.
Take the points on $\mathbb{T}_{\rm TO}$ and $\mathbb{T}_{\rm TD}$ with $W=W_1$.
Let $\gamma^{\rm fast}$ denote the singular limit in the plane $W=W_1$ of the homoclinic orbit in the intersection of $W^u(S_s^\delta)$ and $W^s(S_s^\delta)$ that connects these points on $\mathbb{T}_{\rm TO}$ and $\mathbb{T}_{\rm TD}$.  
Let $\gamma^{\rm slow}_{W_1}$ denote the segment of the slow flow on $S_s^0$ between the two points.
(This is a hyperbolic orbit segment of the slow flow that lies above $\Gamma_t^0$ and $\Gamma^0_f$.)
Then, $\gamma^0_{W_1} = \gamma^{\rm slow}_{W_1} \cup\gamma^{\rm fast}$ is a singular (non-maximal) spatially-periodic orbit.
The persistence of these singular (non-maximal) spatially-periodic canards as periodic solutions of \eqref{eq:y-ODE} follows directly from Theorem 2.11 in \cite{dRDR2016} and Theorem 2 in \cite{ST2001}, without any modification, because the slow segments $\gamma^{\rm slow}_{\rm max}$ lie entirely on $S_s^0$.
\end{proof} 

\medskip
\noindent
This completes the proof of Theorem~\ref{thm:2}.
$\Box$.

\medskip

\begin{remark}
    The fast segment $\gamma^{\rm fast}$ that lies in the singular limit of $W^u(S_s^\delta) \cap W^s(S_s^\delta)$ and that connects $\mathbb{T}_{\rm TO}$ and $\mathbb{T}_{\rm TD}$ is parametrized as follows,
    \begin{equation}\label{eq:gammafast}
    \tilde{\gamma}(\xi;\delta)
    = \begin{bmatrix}
          u_{\rm H} (\xi; \tfrac{1}{4} \lambda^{2/3}) + \delta^{2/5} \tilde{u}_1 (\xi)  + \mathcal{O}(\delta^{4/5}) \\
          p_{\rm H} (\xi; \tfrac{1}{4}\lambda^{2/3}) + \delta^{2/5} \tilde{p}_1 (\xi) + \mathcal{O}(\delta^{4/5}) \\
          \tfrac{1}{4}\lambda^{2/3} + \delta^{2/5}  \tfrac{1}{3}(2\lambda^{2/3} \log(2) - \alpha_3) + \mathcal{O}(\delta^{4/5}) \\
              - \tfrac{1}{2} \lambda \xi + \mathcal{O}(\delta^{4/5})
    \end{bmatrix}.
    \end{equation}
    This is derived, as follows.
    The solution satisfies the boundary conditions $\tilde{\gamma}(-L,\delta) \in W^u(S_s^\delta)$ and $\tilde{\gamma}(L;\delta) \in W^s(S_s^\delta)$.
    We expand it in an asymptotic expansion,
    $\tilde{\gamma}(\xi,\delta) = \tilde{\gamma}_0(\xi) + \delta^{2/5} \tilde{\gamma}_1(\xi) + \mathcal{O}(\delta^{4/5})$.
    The boundary conditions become
    \begin{equation*}
        \tilde{\gamma}_0 (\pm L) 
        = \begin{bmatrix}
              u_{\rm H} (\pm L; W_{\rm TOTD}(0)) \\
              p_{\rm H} (\pm L; W_{\rm TOTD}(0)) \\
              W_{\rm TOTD}(0) \\
              0
        \end{bmatrix}
        \quad
        \tilde{\gamma}_1(\xi) =
        \begin{bmatrix}
            \tilde{u}_1 (\pm L) \\
            \tilde{p}_1 (\pm L) \\
            \left( \tfrac{d}{d \delta^{2/5}} W_{\rm TOTD}(\delta) \right) \left. \right\vert_{\delta=0} \\
            \mp q_{\rm TOTD}(\delta)
        \end{bmatrix}.
    \end{equation*}
    Here, $W_{\rm TOTD}(\delta)$ is given by \eqref{eq:WTOTD},
    $q_{\rm TOTD}(\delta)=\sqrt{\lambda W_{\rm TOTD}(\delta)}$,
    and $(\tilde{u}_1,\tilde{p}_1)$ denotes the solution of
    \eqref{eq:delta25sys} with boundary conditions
    $(\tilde{u}_1(0), \tilde{p}_1(0), \tilde{v}_1(0), \tilde{q}(0))=(0,0,w_0,q_0)$.
\end{remark}

\begin{remark}
    The geometric construction performed in this section for all $\mathcal{O}(1)$ values of $\lambda>0$ can be extended to the regime in which $0<\lambda\ll 1$, where the folded singularity is an RFSN-II.
    We refer to Section 7 of \cite{JDKV2026} for an example in which the true and faux canards of the RFSN-II point and the takeoff and touchdown curves are presented for the Brusselator model.
\end{remark}

\section{An example: The classical Gierer-Meinhardt PDE} \label{sec:GM-analysis}

In this section, we study the classical, non-dimensional Gierer-Meinhardt PDE as a prototypical example of the PDEs \eqref{eq:genAI-UV} with a non-degenerate fold point.
The conditions for the RFSN-II and Turing bifurcations are identified,
and the results of numerical simulations of attracting spatially-periodic canard solutions of the full PDE model are shown to illustrate Theorem~\ref{thm:2}.

\subsection{The RFSN-II point and the spatially-periodic canards created in Turing bifurcations in the classical Gierer-Meinhardt PDE}

In the limit of small activator diffusivity, the classical Gierer-Meinhardt equations are
\begin{equation}\label{eq:GM-nondim}
        \partial_t U = \delta^2 \partial_x^2 U + \frac{U^2}{V} - \gamma_1 U + \gamma_2, \hskip0.3truein
        \partial_t V = \partial_x^2 V + U^2 - V.
\end{equation}
The non-dimensional parameters satisfy $\gamma_1, \gamma_2 > 0$ and $0< \delta \ll 1$.
See especially Section 14.2 of \cite{M1989}, as well as \cite{GM-orig,Me1982} (where we take $\rho_h=0$) and \cite{EK2005,W1997}, where more generally $0<\delta < 1$.

The PDE \eqref{eq:GM-nondim} is of the form \eqref{eq:genAI-UV} with 
$F(U,V)= \frac{U^2}{V} -\gamma_1 U + \gamma_2$
and $G(U,V)=U^2 - V$.
The null set $\{ F=0 \}$ is given by the graph of $V_0(U) = U^2 (\gamma_1 U - \gamma_2)^{-1}$.
We are interested in that part of the null set with $U > \frac{\gamma_2}{\gamma_1}$.
The non-degenerate fold point on this curve is
$(U_f,V_f)=\left(\frac{2\gamma_2}{\gamma_1},\frac{4\gamma_2}{\gamma_1^2}\right)$.
Also, the point 
$(U_e,V_e)
=\left(\frac{\gamma_2 + 1}{\gamma_1},
\left( \frac{\gamma_2 + 1}{\gamma_1} \right)^2 \right)$ 
is an equilibrium.

Putting the model \eqref{eq:GM-nondim} into the general form \eqref{eq:genAI}, we find
$\frac{1}{2}\lambda = \frac{\gamma_1^2}{4}(1-\gamma_2)$,
$\alpha_1 = -\frac{4}{\gamma_1}$,
$\alpha_2=0$,
$\alpha_3 = -\gamma_1 \gamma_2$,
$\alpha_4=1$, 
and $\alpha_5=\gamma_2$.
We see that $\alpha_3<0$, because $\gamma_1, \gamma_2>0$.
Hence, in the local form \eqref{eq:genAI}, the classical Gierer-Meinhardt model is expressed as
\begin{equation}
    \begin{split}
       \partial_t u &= \delta^2 \partial_x^2 u - v - u^2 + \frac{4}{\gamma_1}uv + \mathcal{O}(v^2,u^2 v), \\
       \partial_t v &= \partial_x^2 v + \frac{\gamma_1^2}{4}(1-\gamma_2) + \gamma_1\gamma_2 u - v - \gamma_2 u^2.
    \end{split}
\end{equation}
In the $u$ equation, the higher order terms are $\mathcal{O}(v^2,u^2v)$, from expanding the nonlinear term.
In the $v$ equation, there are no higher-order terms, because $G$ only has terms that are quadratic in $U$ and linear in $V$.
Furthermore, the fold point is now at $(u_f,v_f)=(0,0)$ for all parameter values,
and the equilibrium is exactly at $(u_e,v_e)=\left(-\frac{\gamma_1}{4\gamma_2}(1 - \gamma_2),-\frac{\gamma_1^2}{16\gamma_2}(1-\gamma_2)^2\right)$.

The Turing bifurcation occurs at 
\begin{equation} \label{eq:gamma_2T-GM}
\begin{split} 
{\gamma_2}_T = \frac{\gamma_1 - 2\delta \sqrt{\gamma_1} - \delta^2}{(\sqrt{\gamma_1}+\delta)^2}, \qquad 
k_T^2 = \frac{\sqrt{\gamma_1}}{\delta}.
\end{split}
\end{equation}
(See Appendix~\ref{app:GM-Turing} for the derivation.)
Hence, $\gamma_{2T} \sim 1 - \frac{4\delta}{\sqrt{\gamma_1}}$ as $\delta \to 0$.
Similarly, in $\lambda$,
the Turing bifurcation occurs at 
$\lambda_T= \frac{\delta \gamma_1^2 (2\sqrt{\gamma_1}+\delta)}{(\sqrt{\gamma_1} + \delta)^2}$.
Hence, $\lambda_T \sim 2 \delta \gamma_1^{3/2}$ as $\delta \to 0$.

Then, in the desingularized reduced system of the spatial ODEs, the linear stability types of the folded singularity $M: (u,q)=(0,0)$ and the equilibrium $E: (u,q)=(u_e,0)$ are
\begin{equation}\label{eq:GM-classifyM}
    \begin{split}
      0 \le \gamma_2<1 &: M = {\rm RFS}, \quad E = {\rm center} \\
      \gamma_2 = 1 (\lambda=0) &: M = \text{RFSN-II}, \quad E = M \\
      \gamma_2>1 &: M = {\rm RFC}, \quad E = {\rm saddle},
  \end{split}
\end{equation}
recall Section~\ref{sec:MRFS+MRFSNII}.

A critical observation is that ${\gamma_2}_T$ is asymptotically close in $\delta$ to $\gamma_2=1$, where the RFSN-II singularity occurs, {\it i.e.,} $\lambda_T$ is asymptotically close to $\lambda=0$.
Therefore, also in the classical Gierer-Meinhardt PDE, the spatial ODEs have an RFSN-II singularity asymptotically close to the Turing bifurcation, just as in the Brusselator and van der Pol PDEs, recall \cite{JDKV2026,VDK2025}.  
(For reference, we note also that the eigenvalues of the fourth-order spatial ODE system at the 1:1 resonant Hopf point are $\omega=\sqrt{\delta}\gamma^{1/4}$, recall \eqref{eq:omega-NF}.)

In the $(\gamma_2,k)$ plane (see Figure~\ref{fig:gmexistence}), the homogeneous state $(U_e,V_e)$ is linearly unstable in the red-shaded region, which is bounded by the red curve. 
The Turing point is located at the tip.
Spatially-periodic solutions exist in the region (existence balloon) bounded by the black curve.

\begin{figure}[ht!]
  \centering
  \includegraphics[width=3.5in]{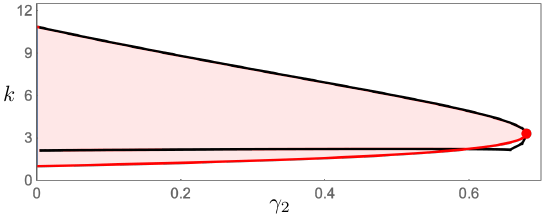}
  \caption{
  Existence balloon of the steady state ODEs of the Gierer-Meinhardt system \eqref{eq:GM-nondim} in the $(\gamma_2,k)$ parameter plane for $\gamma_1=0.3$ and $\delta=0.05$. 
  Here, $k = \tfrac{2\pi}{\delta X}$ where $X$ is the spatial period of the steady state ODEs.
  The black curve encloses the region where spatially-periodic solutions exist. 
  The red curve is the linear stability boundary for the homogeneous equilibrium state, and its tip (red dot) at $(\gamma_{2T},k_T) \approx (0.6794,3.336)$ is the Turing bifurcation. 
  Thus, the homogeneous equilibrium is unstable in the red-shaded region inside the red curve and stable in the region outside the red curve. Note that the upper red and upper black curves coincide.
  The coefficient $b$ that determines the criticality of the Turing bifurcation (see \eqref{eq:HI-nf-b}) is $b \approx 9.26475 \times 10^{-9}$, so that the Turing bifurcation is supercritical.
 The solutions here were obtained using numerical continuation, with the same method as in \cite{JDKV2026}.
  }
  \label{fig:gmexistence}
\end{figure}

The Turing bifurcation \eqref{eq:gamma_2T-GM} can be either subcritical or supercritical, depending on the parameters.
The exact formula for the coefficient that determines the criticality is 
\[
    b = -\tfrac{16}{9}\left( \tfrac{1}{2} \left( \sqrt{\gamma_1}+\delta \right) \right)^{12} \left( 2\gamma_1^2 - 23 \gamma_1 \sqrt{\gamma_1} \delta - 16 \gamma_1 \delta^2+153 \sqrt{\gamma_1} \delta^3+36\delta^4 \right). 
\]
This can be factored as
\begin{equation}\label{eq:HI-nf-b}
b = -\tfrac{32}{9} \left( \tfrac{1}{2} \left( \sqrt{\gamma_1}+\delta \right) \right)^{12} \left( \sqrt{\gamma_1}-\beta_1 \delta \right) \left( \sqrt{\gamma_1}-\beta_2 \delta \right)\left( \sqrt{\gamma_1}-\beta_3 \delta \right)\left( \sqrt{\gamma_1}-\beta_4 \delta \right), 
\end{equation} 
where the coefficients $\beta_1, \beta_2, \beta_3$, and $\beta_4$ are given by
\begin{equation*}
  \begin{split}
      \beta_1 &= \frac{1}{8} \left( 23 - H(\theta) - \sqrt{\frac{2}{3} \left( 1,843-8\sqrt{11,677} \cos \theta - 30,645\, H(\theta)^{-1} \right)} \right), \\ 
      \beta_2 &= \frac{1}{8} \left( 23 - H(\theta) + \sqrt{\frac{2}{3} \left( 1,843-8\sqrt{11,677} \cos \theta - 30,645\, H(\theta)^{-1} \right)} \right), \\
      \beta_3 &= \frac{1}{8} \left( 23 + H(\theta) - \sqrt{\frac{2}{3} \left( 1,843-8\sqrt{11,677} \cos \theta + 30,645\, H(\theta)^{-1} \right)} \right), \\
      \beta_4 &= \frac{1}{8} \left( 23 + H(\theta) + \sqrt{\frac{2}{3} \left( 1,843-8\sqrt{11,677} \cos \theta + 30,645\, H(\theta)^{-1} \right)} \right).
  \end{split}
\end{equation*}
Here, we have defined $\theta := \tfrac{1}{3} \tan^{-1} \left( \frac{162\sqrt{43,402,757}}{673,145} \right)$ and $H(\theta) = \sqrt{\frac{1}{3}\left( 1,843+16\sqrt{11,677}\cos \theta \right)}$.
Numerically, we find that 
\[ \beta_1 \approx -2.52731, \quad \beta_2 \approx -0.23159, \quad \beta_3 \approx 2.64888, \quad \text{ and } \quad \beta_4 \approx 11.61002. \]
Thus, the Turing bifurcation is subcritical for $\gamma_1 \in \left[0, \beta_3^2 \delta^2 \right) \cup \left(\beta_4^2\delta^2,\infty \right)$ and supercritical for $\gamma_1 \in \left( \beta_3^2 \delta^2 , \beta_4^2 \delta^2 \right)$.

\subsection{Stable spatially-periodic canard solutions observed in PDE simulations of the classical Gierer-Meinhardt model \texorpdfstring{\eqref{eq:GM-nondim}}{Lg}}\label{sec:GM-num} 

We present the results of numerical simulations of spatially-periodic canard solutions in the Gierer-Meinhardt model \eqref{eq:GM-nondim} to illustrate Theorem~\ref{thm:2}. 
We used the Method of Lines code in Matlab.

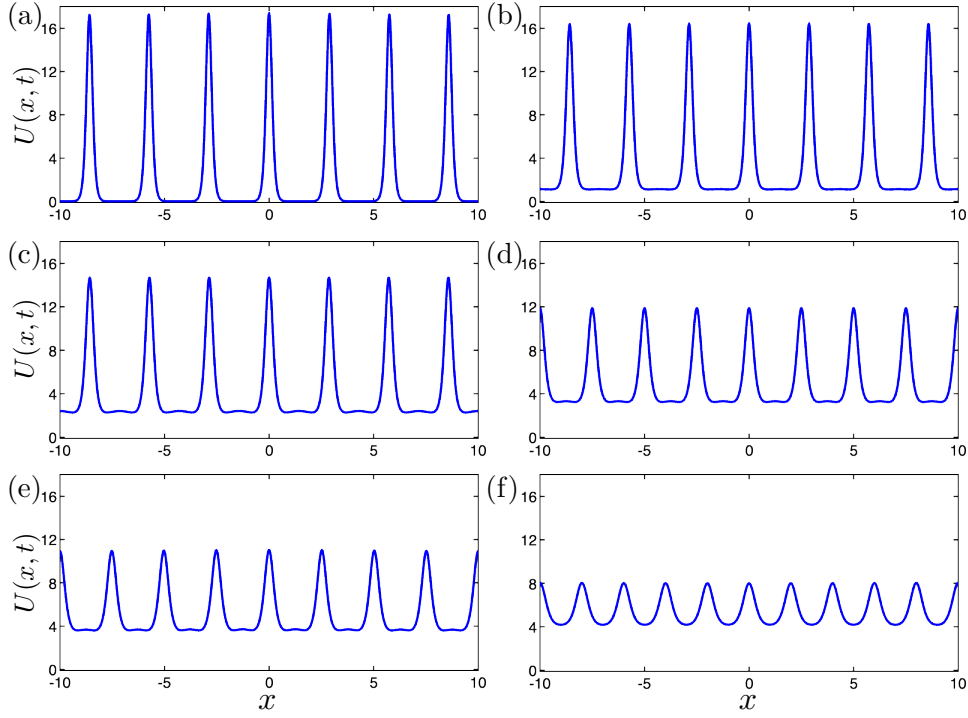
\begin{figure}[ht!]
    \centering
    \includegraphics[width=5in]{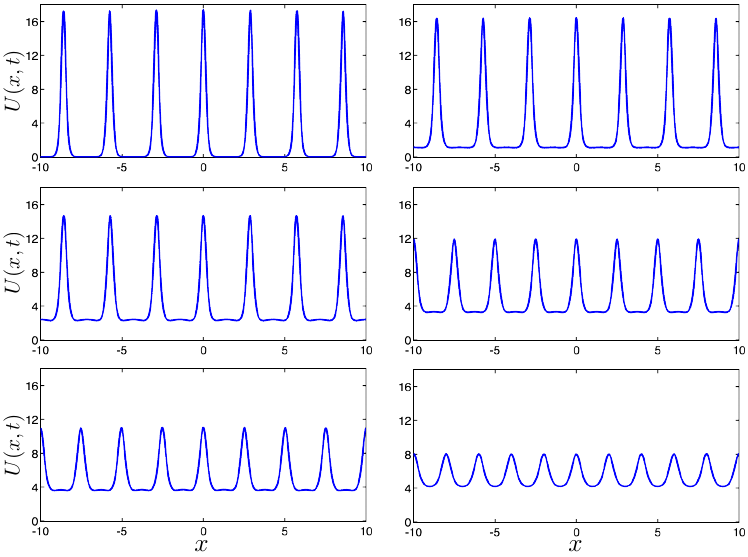}
    \put(-362,260){(a)}
    \put(-182,260){(b)}
    \put(-362,170){(c)}
    \put(-182,170){(d)}
    \put(-362,82){(e)}
    \put(-182,82){(f)}
    \caption{The $U$ component of a spatially-periodic attractor of the Gierer-Meinhardt model \eqref{eq:GM-nondim} as a function of $x$ for $\gamma_1=0.3$ and $\delta=0.05$, with (a) $(\gamma_2,k_0) = (0,8)$, (b) $(\gamma_2,k_0) = (0.28,8)$, (c) $(\gamma_2,k_0) = (0.49,8)$, (d) $(\gamma_2,k_0)=(0.61,5)$, (e) $(\gamma_2,k_0) = (0.64,5)$, and (f) $(\gamma_2,k_0) = (0.67,5)$. 
    The profiles are shown at time $t=15,000$, which is well after the solutions have reached steady state.
    In each simulation, the initial data is a small spatially-periodic perturbation of the homogeneous equilibrium state, {\it i.e.}, $(U(x,0),V(x,0)) = \left( U_e+0.1 \cos^8\left( \tfrac{k_0}{8} x \right), V_e + 0.1 \sin^4 \left( \tfrac{k_0}{4}x \right) \right)$.
    In all simulations, the spatial wave numbers $k$ of the attractors lie in the existence balloon for the corresponding value of $\gamma_2$.
    For reference, we add that the Turing bifurcation at $\gamma_{2T} \approx 0.679392$ is supercritical, recall \eqref{eq:gamma_2T-GM} and \eqref{eq:HI-nf-b}.
    The solutions presented here and in Figures~\ref{fig:gmvarydelta} and \ref{fig:gmheat} were obtained using direct numerical simulations of the PDE \eqref{eq:GM-nondim} in Matlab, with a second-order finite difference for the spatial derivatives and an implicit fourth-order Runge-Kutta scheme for the time-stepping. 
    }
    \label{fig:gmvarygamma2}
\end{figure}

In Figure~\ref{fig:gmvarygamma2}, we show a sequence of six stable spatially-periodic canard solutions --whose existence is established by Theorem~\ref{thm:2}-- that are observed for different values of the bifurcation parameter $\gamma_2$.
Here, the other system parameters, $\gamma_1$ and $\delta$, are fixed, and the Turing bifurcation is supercritical.
Each spatially-periodic canard solution has an interval along which $U(x)$ and $V(x)$ vary gradually in space, in alternation with a narrower interval on which the $U(x)$ component exhibits a narrow, steep pulse.
These spatial canards are created by the folded singularity $M_{\rm RFSN-II}$ for parameters near the Turing point and by $M_{\rm RFS}$ away from it.
The interval on which $U(x)$ and $V(x)$ exhibit gradual variation, which  includes a local maximum midway between adjacent pulses, corresponds to the interval along which the solution is near the true canard and then the faux canard of a folded singularity, either $M_{\rm RFS}$ or $M_{\rm RFSN-II}$ depending on parameters, recall Figure~\ref{fig:pulseconstruction}.

As $\vert \gamma_2 - \gamma_{2T} \vert$ grows, the amplitudes of the spatially-periodic canard attractors grow much more rapidly than the
standard square root of the distance in parameter space, which is the case for regular Turing bifurcations where the oscillations remain sinusoidal to leading order \cite{BDHL2023,E1965,EK2005,EP1998,GS1997,HI2011,IMD1989,M1989,S2003,T1952,W1997}. 
For example, even with the parameter $\gamma_2$ differing by only one part in a hundred from $\gamma_{2T}$, the amplitude is already two (see  Figure~\ref{fig:gmvarygamma2}(f)).
Furthermore, the amplitudes of the pulses are $\mathcal{O}(1/\delta)$ when $\gamma_2$ is taken to be $\mathcal{O}(1)$ away from $\gamma_{2T}$ (see Figure~\ref{fig:gmvarygamma2}(a)).

\begin{figure}[ht!]
    \centering
    \includegraphics[width=5in]{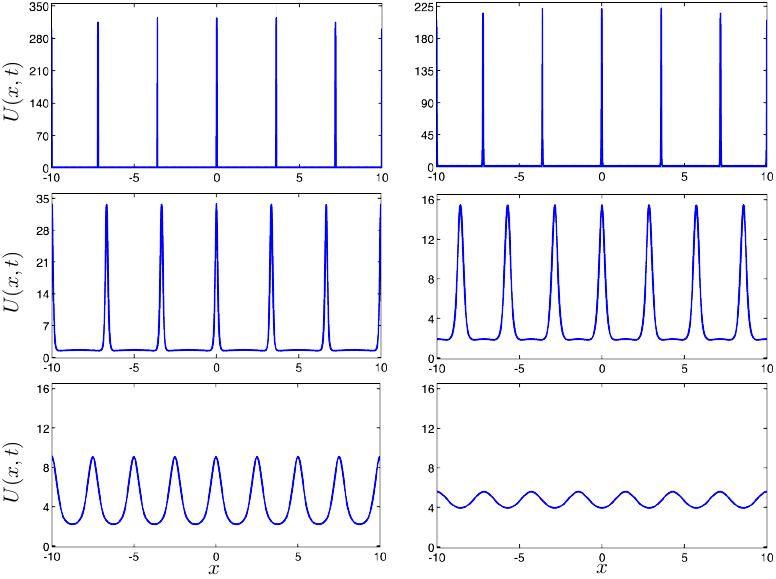}
    \put(-362,260){(a)}
    \put(-182,260){(b)}
    \put(-362,172){(c)}
    \put(-182,172){(d)}
    \put(-362,82){(e)}
    \put(-182,82){(f)}
    \caption{The effect of increasing $\delta$.
    Snapshots of the spatial attractor taken at time $t=150,000$ for $\gamma_1 = 0.3$ and $\gamma_2 = 0.42$ with (a) $\delta = 0.0001$,
    (b) $\delta = 0.005$,
    (c) $\delta = 0.025$,
    (d) $\delta = 0.05$,
    (e) $\delta = 0.075$, and
    (f) $\delta = 0.1$.
    The vertical scales in panels (a)--(c) shrink dramatically as $\delta$ is increased. 
    The vertical scales in panels (d)--(f) are identical to better highlight the transition from slow/fast spatially-periodic variation to the classical sinusoidal oscillations as $\delta$ is increased.
   Based on \eqref{eq:HI-nf-b}, the values of $b$ are
(a) $b \approx -6 \times 10^{-8}$,
(b) $b \approx -5.7 \times 10^{-8}$,
(c) $b \approx -4.5 \times 10^{-8}$,
(d) $b \approx 9.3 \times 10^{-9}$,
(e) $b \approx 1.4 \times 10^{-7}$, and
(f) $b \approx 3.8 \times 10^{-7}$.
Also, we observe that $b$ changes sign at $\delta \approx 0.047177$ and at $\delta \approx 0.206775$.
Thus, the Turing bifurcation is sub-critical in frames (a)-(c),
and it is super-critical in frames (d)-(f).
}
    \label{fig:gmvarydelta}
\end{figure}
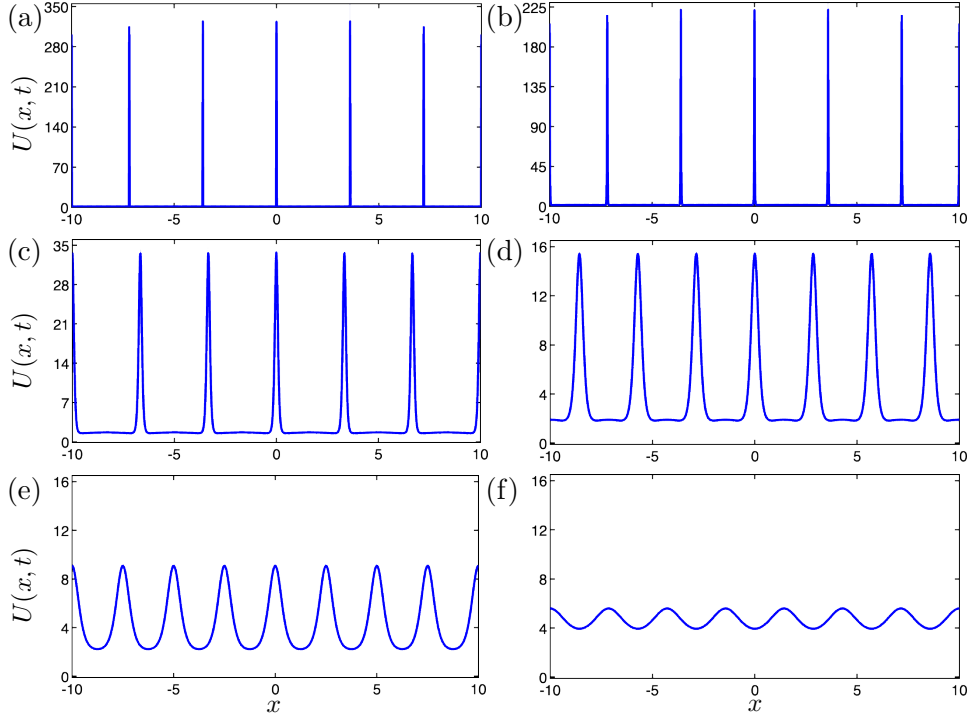

In Figure~\ref{fig:gmvarydelta}, we show the spatial profiles of $U(x,t)$ for the stable, steady-state, spatially-periodic canards for six values of $\delta$.
Here, $\gamma_1$ and $\gamma_2$ are held fixed, and the initial data is of the same type as in Figure~\ref{fig:gmvarygamma2}.
The attractors have a similar fast/slow spatial structure as in Figure~\ref{fig:gmvarygamma2}.

In Figures~\ref{fig:gmvarydelta} (a)--(c), the values of $\delta$ are such that the Turing bifurcations are subcritical.
In the attractors, the amplitudes of the pulses are of size 
$\mathcal{O}(1/\sqrt{\delta})$, and they are observed numerically to lie on a branch of stable periodic canard solutions that meets at a fold point with the branch of unstable spatially-periodic canards that emerges from the sub-critical Turing bifurcation. 
The parameter dependence here is qualitatively similar to that observed for the spatially-periodic canards that emerge from the sub-critical Turing bifurcations in the Brusselator and van der Pol PDEs \cite{JDKV2026,VDK2025}. 

By contrast, in Figures~\ref{fig:gmvarydelta} (d)--(f), the values of $\delta$ are such that the Turing bifurcations are super-critical.
These attractors lie on the branch of stable spatially-periodic canard solutions that emerges from the Turing bifurcation.
Also here, the amplitudes of the spatial oscillations grow much more rapidly than the square root of the distance between the parameter value and the Turing point, similar to what was observed in Figure~\ref{fig:gmvarygamma2}(f).

\begin{figure}[ht!]
    \centering
    \includegraphics[width=5in]{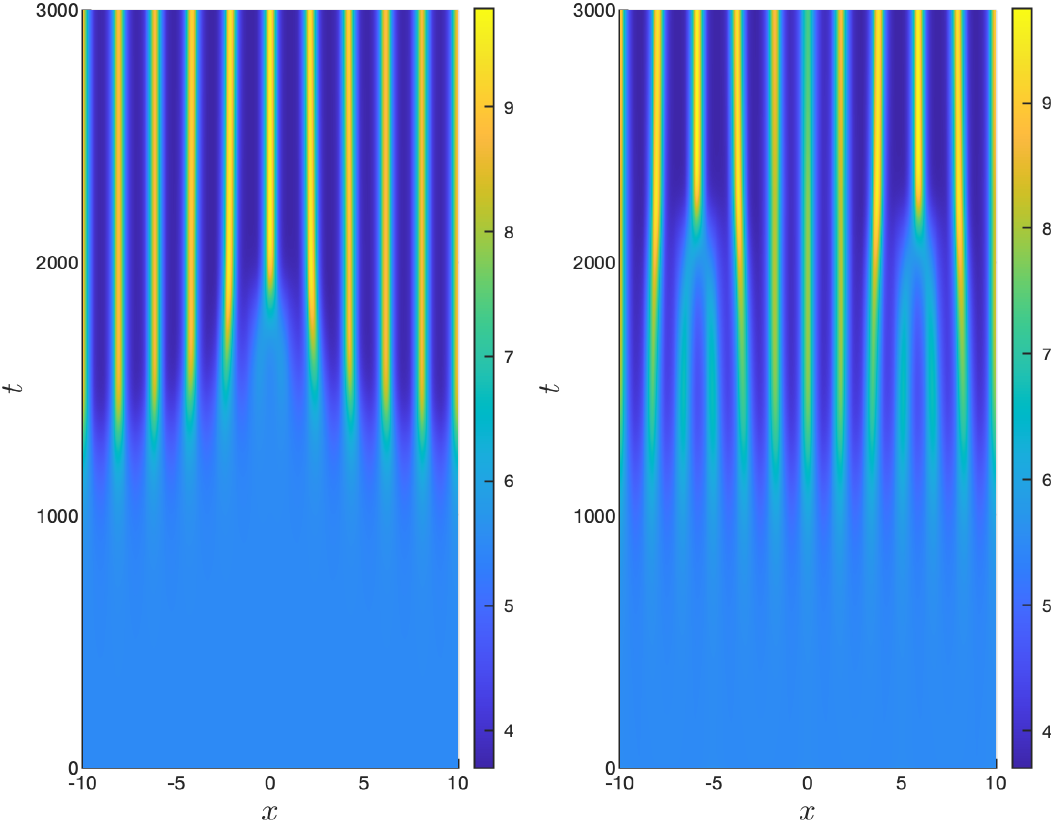}
    \put(-364,274){(a)}
    \put(-182,274){(b)}
    \caption{Space-time evolution of $U(x,t)$ in the Gierer-Meinhardt model \eqref{eq:GM-nondim} with $\gamma_1=0.3$, 
    $\gamma_2=0.65$, and $\delta = 0.05$. 
    The initial profile, $(U(x,0),V(x,0)) = \left( U_e+0.1 \cos^8\left( \tfrac{k_0}{8} x \right), V_e + 0.1 \sin^4 \left( \tfrac{k_0}{4}x \right) \right)$, has a spatial wavenumber of (a) $k_0 = 2$ and (b) $k_0 = 5$. 
    At steady state, the attractor has spatial period 2 and spatial wave number $\pi$ for both simulations.
    This attractor lies inside the existence balloon, recall Figure~\ref{fig:gmexistence}.
    }
    \label{fig:gmheat}
\end{figure}

Finally, in Figure~\ref{fig:gmheat}, we show the temporal evolution of two different initial data, which consist of small-amplitude, spatially-periodic perturbations of the homogeneous state $(U_e,V_e)$, just as in Figures~\ref{fig:gmvarygamma2} and \ref{fig:gmvarydelta}.
The simulations were run to time $t=15,000$.
In frame (a), the steady state was already reached before $t=3,000$, and in frame (b) it is reached at a later time, but still well before the end of the simulation (not shown).
While the initial data have different wave numbers, which lie outside the existence balloon,
both solutions are in the basin of attraction of --and evolve to-- the same spatially-periodic canard solution with period 2, which lies inside the existence balloon.

During the evolution toward steady state, the amplitudes of the local extrema in these solutions grow in time.
Also, some pulses disappear, typically with every other pulse disappearing and with the disappearances occurring first among the central pulses.
Similar amplitude growth and pulse disappearance are observed for other initial spatially-periodic initial data and for more general, small-amplitude initial data, with the same parameter values.
In addition, pulse creation is observed when the initial data has fewer local extrema than the attractor.

\begin{remark}
System \eqref{eq:GM-nondim}
is the classical Gierer-Meinhardt model in non-dimensionalized form. 
It may be derived from the dimensional form,
$
        \partial_{\hat{t}} \hat{U} = \delta^2 \partial_{\hat{x}}^2 \hat{U} + \rho \frac{\hat{U}^2}{\hat{V}} - \mu_a \hat{U} + \rho_a,  \hskip0.25truein 
        \partial_{\hat{t}} \hat{V} = \partial_{\hat{x}}^2 \hat{V} + \rho \hat{U}^2 - \mu_h \hat{V},
$
by setting 
$\hat{U}=U$,
$\hat{V}=\frac{\rho}{\mu_h} V$,
$\hat{x}=\frac{1}{\sqrt{\mu_h}}x$, and 
$\hat{t}=\frac{1}{\mu_h} t$.
\end{remark}

\begin{remark}
For the generalized Gierer-Meinhardt system, the first terms in $F$ and $G$ are $\frac{U^r}{V^s}$ and $\frac{U^t}{V^\omega}$, respectively, with the four powers being non-negative.
The generalized Gierer-Meinhardt system can also exhibit Turing bifurcations, and the methods of this article can be applied when the diffusivity of $U$ is small.
(Note that the classical Gierer-Meinhardt model corresponds to the case of $r=2, s=1, t=2$, and $\omega=0$.)
\end{remark} 

\noindent
{\bf Acknowledgments.}
R.J. was supported in part by a research account at Boston University.
T.K. thanks the organizers Montie Avery, Gregory Faye, Ryan Goh, Matthew Holzer, Gabriela Jaramillo, Merlin Pelz, and Qiliang Wu of the conference ``Order and function in complex systems," held at the University of Minnesota in August 2026, for their invitation to speak on these results.

\appendix

\section{The proof of Proposition \texorpdfstring{\ref{prop1}}{Lg}}\label{app:prop1} 

In this appendix, we present the proof of Proposition~\ref{prop1}.
Define 
\begin{equation*}
a_{ij} = \frac{1}{i!j!} \frac{\partial^{i+j} F}{\partial U^i \partial V^j}(U_f,V_f), \ \ {\rm where} \, \,  i,j=0,1,2,\dots,
\end{equation*} 
\begin{equation*} 
b_{ij} = \frac{1}{i!j!} \frac{\partial^{i+j} G}{\partial U^i \partial V^j}(U_f,V_f), \ \ {\rm where} \, \,  i,j=0,1,2,\ldots.
\end{equation*} 
Observe that $a_{00}, a_{10} = 0$  
and $a_{01}, a_{20} \ne 0$,
since $(U_f,V_f)$ is a non-degenerate fold point.

First, translate the fold point to the origin by setting $U = U_f + \tilde{U}$ and $V = V_f + \tilde{V}$ and Taylor expanding $F$ and $G$ about $(U_f,V_f)$.
This transforms the system \eqref{eq:genAI-UV} into 
\begin{equation*}
\begin{split}
    \partial_{t}\tilde{U} &= \delta^2 \partial^2_x \tilde{U} + a_{01} \tilde{V} + a_{20} {\tilde{U}}^2 + a_{11} \tilde{U}\tilde{V} + a_{02} {\tilde{V}}^2 + a_{30} {\tilde{U}}^3 + a_{21} {\tilde{U}}^2{\tilde{V}} + a_{12} {\tilde{U}}{\tilde{V}}^2 + a_{03} {\tilde{V}}^3 + \mathcal{O}(4), \\
    \partial_{t}\tilde{V} &= \partial^2_x \tilde{V} + b_{00} + b_{10}\tilde{U} + b_{01} \tilde{V} + b_{20} {\tilde{U}}^2 + b_{11} \tilde{U}\tilde{V} + b_{02} {\tilde{V}}^2 
    + b_{30} {\tilde{U}}^3 + b_{21} {\tilde{U}}^2{\tilde{V}} + b_{12} {\tilde{U}}{\tilde{V}}^2 + b_{03} {\tilde{V}}^3 + \mathcal{O}(4).
\end{split}
\end{equation*}
Here, the terms in $\mathcal{O}(k)$ vanish at least as fast as $\tilde{U}^k, \tilde{U}^{k-1}\tilde{V}, \ldots, \tilde{U} \tilde{V}^{k-1}, \tilde{V}^k$, as $(U,V)\to (0,0)$.

Then, scale the dependent variables by setting $u=-a_{20}\tilde{U}$ and $v=a_{01}a_{20} \tilde{V}$.
Locally, in a neighborhood of the fold point, the system of PDEs transforms into \eqref{eq:genAI} 
with parameters
\begin{equation} \label{eq:AppA-alpha_i}
\alpha_1 = \frac{-a_{11}}{a_{01}a_{20}}, 
\ \ \ 
\alpha_2 = \frac{-a_{30}}{a^2_{20}},
\ \ \ 
\frac{1}{2}\lambda = a_{01} a_{20} b_{00}, 
\ \ \ 
\alpha_3 = a_{01} b_{10},
\ \ \ 
\alpha_4 = -b_{01},
\ \ \ 
\alpha_5 = \frac{-a_{01}b_{20}}{a_{20}}.
\end{equation}
In a neighborhood of $(U_f,V_f)$, the terms $v^2, u^2v,$ and $u^4$ in $f$ are higher order, and the terms $uv, v^2$, and $u^3$ in $g$ are higher order, because the null set is locally a non-degenerate quadratic. 
\hfill
$\Box$

\section{The proof of Proposition \texorpdfstring{\ref{prop:haragusiooss}}{Lg}}\label{app:Turing}

In order to apply the normal form theory from Chapter 4.3.3 of \cite{HI2011} for systems with 1:1 resonant Hopf bifurcations (where these bifurcation points are labeled as the $(i \omega)^2$ resonant case), we transform variables in the vector field \eqref{eq:y-ODE-mu} to put the equilibrium \eqref{eq:equilib} at the origin.
Let
\begin{equation}
u = u_e(\mu_T) + u_1, \quad 
p = u_2, \quad 
v = v _e( \mu _T) + u_3, \quad 
q= u_4, \quad 
\mu = \mu _T + \nu.
\end{equation}
Next, we Taylor expand the vector field about the equilibrium,
\begin{equation}\label{eq:Taylorexpanded}
\begin{split}
{u_1}_x &= u_2, \\
{u_2}_x &= - \Sigma_{i + j + k } f_{ijk} u_1^i u_3^j \nu^k, \\
{u_3}_x &= \delta u_4, \\
{u_4}_x &= -\delta \Sigma_{i+j+k} g_{ijk} u_1^i u_3^j \nu^k.
\end{split}
\end{equation}
Here,
\begin{equation*} 
  \begin{split}
      f_{ijk} = \left. \frac{1}{(i+j+k)!} \frac{\partial^{i+j+k} f}{\partial u_1^i \partial u_3^j \partial \nu^k} \right|_{(u_1,u_2,u_3,u_4,\nu)=(0,0,0,0,0)} \\
      g_{ijk} = \left. \frac{1}{(i+j+k)!} \frac{\partial^{i+j+k} g}{\partial u_1^i \partial u_3^j \partial \nu^k} \right|_{(u_1,u_2,u_3,u_4,\nu)=(0,0,0,0,0)}
  \end{split}
\end{equation*}

It will be useful to re-express this vector field in terms of operators.
Let ${\bf u}=(u_1, u_2, u_3, u_4)^t$.
Out to third order, the vector field and linear operator $\mathbb{L}$ are 
\begin{equation}
\frac{d {\bf u}}{dy} = \mathbb{L} {\bf u }+ R_{20} ({\bf u}, {\bf u })
+ R_{30} ({\bf u}, {\bf u}, {\bf u })
+ \nu (R_{01} + R_{11} ({\bf u}) + R_{21} ({\bf u}, {\bf u})) 
+ \nu^2 ( R_{02} + R_{12}({\bf u}))
+ \nu ^3 R_{03},
\end{equation}
\begin{equation}
\mathbb{L} =
\left[
\begin{array}{cccc}
0 & 1 & 0 & 0 \\
-2\omega^2 + \delta^2 g_{010} & 0 & -f_{010} & 0 \\
0 & 0 & 0 & \delta \\
-\delta g_{100} & 0 & -\delta g_{010} & 0
\end{array}
\right],
\end{equation}
where we used
$f_{100} = 2\delta \sqrt{-f_{010} g_{100}} + \delta^2 g_{010} =2 \omega^2 -\delta^2 g_{010},$
which follows by \eqref{eq:Turing-condition} in Hypothesis~\ref{hypo2}.
The multi-linear functions $R_{ij}$ denote nonlinear terms of order ${\bf u}^i \mu^j$, where ${\bf u}^i$ denotes a monomial of the form $u_1^{i_1} u_2^{i_2} u_3^{i_3} u_4^{i_4}$ such that $i_1+i_2+i_3+i_4 = i$.
We find
\begin{equation}
R_{20} ({\bf u}, {\bf v}) =
\left[
\begin{array}{c}
0 \\
-(f_{200} u_1 v_1 + f_{110} u_1 v_3 + f_{020} u_3 v_3) \\
0 \\
-\delta (g_{200} u_1 v_1 + g_{110} u_1 v_3 + g_{020} u_3 v_3)
\end{array}
\right],
\end{equation}

\begin{equation}
R_{30} ({\bf u}, {\bf v}, {\bf w}) =
\left[
\begin{array}{c}
0 \\
-(f_{300} u_1 v_1 w_1 + f_{210} u_1 v_1 w_3 + f_{120} u_1 v_3 w_3 + f_{030} u_3 v_3 w_3) \\
0 \\
-\delta (g_{300} u_1 v_1 w_1 + g_{210} u_1 v_1 w_3 + g_{120} u_1 v_3 w_3
+ g_{030} u_3 v_3 w_3)
\end{array}
\right],
\end{equation}

\begin{equation}
R_{01} = \left[
\begin{array}{c}
0 \\
-f_{001} \\
0 \\
-\delta g_{001}
\end{array}
\right],
\qquad  
R_{11} ({\bf u}) = \left[
\begin{array}{c}
0 \\
-( f_{101} u_1 + f_{011} u_3)\\
0 \\
-\delta ( g_{101} u_1 + g_{011} u_3)
\end{array}
\right],
\end{equation}

\begin{equation}
R_{21} ({\bf u}, {\bf v}) = \left[
\begin{array}{c}
0 \\
-( f_{201} u_1 v_1 + f_{111} u_1 v_3 + f_{021} u_3 v_3) \\
0 \\
-\delta ( g_{201} u_1 v_1 + g_{111} u_1 v_3 + g_{021} u_3 v_3)
\end{array}
\right],
\end{equation}

\begin{equation}
R_{02} = \left[
\begin{array}{c}
0 \\
-f_{002} \\
0 \\
-\delta g_{002}
\end{array}
\right],
\qquad 
R_{12} ({\bf u}) = \left[
\begin{array}{c}
0 \\
-( f_{102} u_1 + f_{012} u_3)\\
0 \\
-\delta ( g_{102} u_1 + g_{012} u_3)
\end{array}
\right],
\qquad 
R_{03} = \left[
\begin{array}{c}
0 \\
-f_{003} \\
0 \\
-\delta g_{003}
\end{array}
\right].
\end{equation}

The vector field has a reversibility symmetry
\begin{equation}
{\bf F} (\mathcal{S} {\bf u}, \nu) = - \mathcal{S} {\bf F}({\bf u},\nu),
\hskip0.4truein
\mathcal{S} = \begin{bmatrix} 1 & 0 & 0 & 0 \\ 0 & -1 & 0 & 0 \\ 0 & 0 & 1 & 0 \\ 0 & 0 & 0 & -1 \end{bmatrix}.
\end{equation}
This reversibility symmetry exists, because the PDE has only even-order spatial derivatives. 

Next, we establish that the system \eqref{eq:Taylorexpanded} satisfies the criteria for the normal form theory of \cite{HI2011}. 
We use the eigenvectors and generalized eigenvectors of $\mathbb L$ corresponding to the eigenvalue $i \omega$ to build a basis for $\mathbb{R}^4$, where we recall that
\begin{equation}\label{eq:omega}
\omega^2 = \delta \sqrt{-f_{010} g_{100}} + \delta^2 g_{010}.
\end{equation}
These vectors are found by solving 
\begin{equation}
(\mathbb{L} - i \omega) {\bf F} = \xi,
\end{equation}
which has solutions provided the vector
$\xi = [\xi_1,\xi_2,\xi_3,\xi_4]^t$ lies in the column space of the matrix $\mathbb{L} - i \omega$, {\it i.e.,} provided 
\begin{equation}
\delta^2 g_{100} ( \omega \xi_1 - i \xi_2 )+(\delta^2 g_{010} - \omega^2)(\omega \xi_3 - i \delta \xi_4) = 0.
\end{equation}
Without loss of generality, we fix the eigenvector  and generalized eigenvector
\begin{equation}
\zeta_0 = \left[ \begin{array}{c}
-\omega f_{010} \\
-i \omega^2 f_{010}\\
\omega ( \omega^2- \delta^2 g_{010})\\
i \delta ( -f_{010} g_{100 } + g_{010} \omega^2-\delta^2 g_{010}^2)
\end{array}
\right];
\end{equation} 
\begin{equation}
\zeta_1 =
\left[
\begin{array}{c}
2 i \omega^2 f_{010} \\
\omega f_{010} (\delta^2 g_{010} - 3 \omega^2) \\
0 \\
-\frac{\delta}{\omega} ( g_{010}(\omega^2-\delta^2 g_{010})^2 + f_{010} g_{100} ( \omega^2+\delta^2 g_{010}))
\end{array}
\right].
\end{equation}
By direct computation, we verify that 
\[ \mathbb L \zeta_0 = i \omega \zeta_0, \quad (\mathbb L - i \omega) \zeta_1 = \zeta_0, \quad \mathcal{S} \zeta_0 = \overline{\zeta_0}, \quad \text{ and } \quad \mathcal{S} \zeta_1 = - \overline{\zeta_1}. \]
Hence, system \eqref{eq:Taylorexpanded} satisfies the criteria for the case of $(i\omega)^2$ resonance in Chapter 4.3.3 of \cite{HI2011}. 

As a consequence, there exists a change of variables of the form 
\begin{equation} \label{eq:app_trans}
{\bf u} = A \zeta_0 + B \zeta_1 + \overline{A} \, \overline{\zeta_0} +\overline{B} \, \overline{\zeta_1} + \tilde{F} (A,B,\overline{A},\overline{B},\nu), 
\end{equation}
where $A,B \in \mathbb{C}$,
such that the system \eqref{eq:Taylorexpanded} is transformed into the normal form
\begin{equation} \label{eq:HI_NormalForm}
\begin{split}
    \tfrac{dA}{dx} &= i \omega A + B + i A P(A,B,\nu) + \rho_A, \\ 
    \tfrac{dB}{dx} &= i \omega B + i B P(A,B,\nu) + A Q(A,B,\nu) + \rho_B.
\end{split}
\end{equation}
Here, $P$ and $Q$ are real-valued functions given by 
\begin{equation}\label{eq:PQ}
\begin{split}
    P(A,B,\nu) &= \alpha \nu + \beta A \overline{A} + \tfrac{1}{2} i \gamma \left( A \overline{B} - \overline{A} B \right) \\ 
    Q(U,V,\nu) &= a \nu + b A \overline{A} + \tfrac{1}{2} i c \left( A \overline{B} - \overline{A} B \right),
\end{split}
\end{equation}
and the functions $\rho_A=\rho_A(A,B,\overline{A},\overline{B},\nu)$ and $\rho_B=\rho_B(A,B,\overline{A},\overline{B},\nu)$ encode the higher order terms.
The conclusions of the theorem hold in the case that $a,b \ne 0$.
This is the normal form that is valid for all $\nu$ in a neighborhood of zero.
The conditions in Proposition~\ref{prop:haragusiooss} follow from this normal form, where we recall that ${\rm sgn}(a\nu) = - {\rm sgn}(\lambda-\lambda_T)$.
Also, we recall from Section~4.3.3 of \cite{HI2011} that the truncated normal form ({\it i.e.,} \eqref{eq:HI_NormalForm}-\eqref{eq:PQ} with $\rho_A, \rho_B = 0$)
is integrable, with two constants:
$K= \frac{i}{2}(A \overline{B} - \overline{A} B)$
and $H = \vert B \vert^2 - \int_0^{\vert A\vert^2} Q(s,K,\nu)ds$.

To find the coefficients $\alpha, \beta, \gamma, a, b, $ and $c$
and the terms in $\tilde{F}$, we follow \cite{HI2011}.
By differentiating \eqref{eq:app_trans} on $x$ and using \eqref{eq:Taylorexpanded} and \eqref{eq:HI_NormalForm}, we obtain the following invariance equation:
\begin{equation}
\begin{split} 
\mathbb{L} \tilde {F} + R &= (i \omega A + B) \partial_A \tilde{F} + i \omega B \partial_B \tilde{F} +( -i \omega \overline{A} +\overline{B})\partial_{A} \tilde{F} - i \omega \overline{B} \partial_{B} \tilde{F} \\
&{\hskip0.1truein} +( i A ( \zeta_0 + \partial_A \tilde{F} ) -i \overline{A} (\overline{\zeta_0} + \partial_{\overline A}\tilde{F})) P
+( i BP + AQ)( \zeta_1 + \partial_B \tilde{F})\\
&{\hskip0.1truein} + (-i \overline{B} P+ \overline{A} Q)( \overline{\zeta_1} +\partial_{\overline B} \tilde{F}) + h.o.t.
\end{split}
\end{equation}
We will solve the invariance equation order by order.
To do so, it is useful to Taylor expand
$\tilde{F} = \Sigma_{r+s+q+\ell+m \le p} A^r B^s \overline{A}^q \overline{B}^\ell \nu^m F_{rsq\ell m}$ and to project onto the vector
\begin{equation}
\zeta_1^* =
\frac {1}{4 \omega^2}
\left[
\begin{array}{c}
i \left( \frac{\delta^2 g_{010} - \omega^2}{f_{010}} \right) \\
-\frac{\omega}{f_{010}}\frac{f_{010}g_{100} +g_{010}(\delta^2 g_{010} - \omega^2)}
{f_{010}g_{100} + g_{010}( \delta^2 g_{010} - 2 \omega^2)} \\
-i \\
\frac{\omega^3}{\delta(f_{010}g_{100} +g_{010}(\delta^2g_{010} -2\omega^2))}
\end{array}
\right].
\end{equation}
This vector is orthogonal to the range of $(\mathbb{L} - i \omega)$, and it also satisfies
\[ 
\langle \zeta_0, \zeta_1^* \rangle = 0, \quad 
\langle \overline{\zeta_0}, \zeta_1^* \rangle = 0, \quad 
\langle \zeta_1, \zeta_1^* \rangle = 1, \quad \text{ and } \quad 
\langle \overline{\zeta_1}, \zeta_1^* \rangle = 0.
\]
Below, we will use the orthogonality of $\zeta_1^*$ to calculate some of the coefficients in the normal form.

We start with the following leading order terms:
\begin{equation}
\begin{split}
\mathcal{O}(\nu) &: \mathbb{L} F_{00001} + R _{01} = 0 \\
\mathcal{O}(\nu A ) &: (\mathbb{L} - i \omega) F_{10001} + R_{11}(\zeta_0) + R_{20}( \zeta_0,F_{00001}) + R_{20}(F_{00001},\zeta_0) = i \alpha \zeta_0 + a \zeta_1 \\
\mathcal{O}(\nu B) &: (\mathbb{L} - i \omega) F_{01001} + R_{11}(\zeta_1) + R_{20}( \zeta_1,F_{00001}) + R_{20}(F_{00001},\zeta_1) = i \alpha \zeta_1 + F_{10001}.
\end{split}
\end{equation}
These equations yield
\begin{equation}
\begin{split}
F_{00001} &= - \mathbb{L}^{-1} R_{01}, \\
a &= \langle (\mathbb{L} -i \omega) F_{10001} +R_{11} (\zeta_0) +R_{20}(\zeta_0,F_{00001}) + R_{20}(F_{00001},\zeta_0), \zeta_1^* \rangle,\\
i \alpha +\langle F_{10001}, \zeta_1^*\rangle &= \langle (\mathbb{L} -i \omega) F_{01001} +R_{11} (\zeta_1) + R_{20}(\zeta_1,F_{00001}) + R_{20}(F_{00001},\zeta_1),\zeta_1^* \rangle.
\end{split}
\end{equation}
We observe that ${\rm sgn}(a \nu) = - {\rm sgn} (\lambda-\lambda_T)$.

At second order, the invariance equation has the following quadratic terms: 
\begin{equation}
\begin{split}
\mathcal{O}(A^2) &: \quad 
   (\mathbb{L} - 2i \omega) F_{20000} + R_{20}(\zeta_0,\zeta_0) = 0, \\
\mathcal{O}(AB) &: \quad 
   (\mathbb{L} - 2i \omega) F_{11000} + R_{20}(\zeta_0,\zeta_1) + R_{20}(\zeta_1,\zeta_0) = 2 F_{20000}, \\
\mathcal{O}(A \overline{A}) &: \quad 
   \mathbb{L} F_{10100} + R_{20}(\zeta_0,\overline{\zeta_0}) + R_{20}(\overline{\zeta_0},\zeta_0) = 0, \\
\mathcal{O}(A \overline{B}) &: \quad 
   \mathbb{L}F_{10010} + R_{20} (\zeta_0,\overline{\zeta_1}) +R_{20}(\overline{\zeta_1},\zeta_0)= F_{10100},\\
\mathcal{O}(B^2) &: \quad 
   (\mathbb{L}-2i\omega) F_{02000} + R_{20}(\zeta_1,\zeta_1) = F_{11000}, \\
\mathcal{O}(B\overline{B}) &: \quad
   \mathbb{L} F_{01010} + R_{20}(\zeta_1,\overline{\zeta_1}) + R_{20} (\overline{\zeta_1},\zeta_1) = F_{01100} + F_{10010}.
\end{split}
\end{equation}
Also, there are terms of 
$\mathcal{O}(\overline{A}^2)$,
$\mathcal{O}(\overline{A}\, \overline{B})$, 
$\mathcal{O}(B \overline{A})$, and
$\mathcal{O}(\overline{B}^2)$.
These are the complex conjugates of the terms of
$\mathcal{O}(A^2)$,
$\mathcal{O}(AB)$, 
$\mathcal{O}(A \overline{B})$, and
$\mathcal{O}(B^2)$, respectively.
Moreover, at this order, all equations are solvable, because they are linear equations and the coefficient matrices (which are of the form $\mathbb{L}$ and $(\mathbb{L} \pm 2 i \omega)$) are invertible.

At third order, the crucial equation is
\begin{equation}\label{eq:b}
\begin{split} 
\mathcal{O}(A^2 \overline{A})&:
(\mathbb{L}- i \omega) F_{20100} + R_{20}(\zeta_0,F_{10100} ) + R_{20}(F_{10100},\zeta_0) + R_{20}(\overline{\zeta_0},F_{20000}) + R_{20}(F_{20000},\overline{\zeta_0}) \\
&{\hskip0.1truein} +R_{30}(\zeta_0,\zeta_0,\overline{\zeta_0}) + R_{30}(\zeta_0,\overline{\zeta_0},\zeta_0) + R_{30}(\overline{\zeta_0}, \zeta_0, \zeta_0 ) = i \beta \zeta_0 + b \zeta_1.
\end{split}
\end{equation}
Finally, by projecting equation \eqref{eq:b} onto the vector $\zeta_1^*$, we find that the coefficient $b$ in $Q$ (recall \eqref{eq:PQ}) is 
\begin{equation}\label{eq:def-b}
\begin{split} 
b &= \langle
R_{20}(\zeta_0,F_{10100} ) + R_{20}(F_{10100},\zeta_0) + R_{20}(\overline{\zeta_0},F_{20000}) + R_{20}(F_{20000},\overline{\zeta_0}) \\
&{\hskip0.1truein} +R_{30}(\zeta_0,\zeta_0,\overline{\zeta_0}) + R_{30}(\zeta_0,\overline{\zeta_0},\zeta_0) + R_{30}(\overline{\zeta_0}, \zeta_0, \zeta_0 ),
\zeta_1^*
\rangle.
\end{split}
\end{equation}
We refer to App. D.2 of \cite{HI2011} for full details of the general method.
The equations (D.29)--(D.34) there are the equations at second order,
and (D.35)--(D.44) are the equations at third order.
We add that $b$ has the opposite sign of the Landau coefficient (on the cubic term $\vert A\vert^2 A$) in the Ginzburg-Landau PDE.

For completeness, we list the other equations at third order here. 
All are solvable, because the operators 
$(\mathbb{L}\pm 3i\omega)$ and
$(\mathbb{L}\pm i\omega)$ are invertible.
\begin{equation}
\begin{split}
\mathcal{O}(A^3) &: (\mathbb{L} - 3i \omega) F_{30000} + R_{20}(\zeta_0,F_{20000}) + R_{20}(F_{20000},\zeta_0) + R_{30}(\zeta_0,\zeta_0,\zeta_0) = 0, \\
\mathcal{O}(B^3) &: (\mathbb{L} - 3i \omega) F_{03000} + R_{20}(\zeta_1,F_{02000}) + R_{20}(F_{02000},\zeta_1) + R_{30}(\zeta_1,\zeta_1,\zeta_1) = 0, \\
\mathcal{O}(A^2B) &: (\mathbb{L} - 3i \omega) F_{21000} 
+ R_{20}(\zeta_0,F_{11000}) + R_{20}(F_{11000},\zeta_0) 
+ R_{20}(\zeta_1,F_{20000}) + R_{20}(F_{20000},\zeta_1) \\
&{\qquad  } + R_{30}(\zeta_0,\zeta_0,\zeta_1)
+ R_{30}(\zeta_0,\zeta_1,\zeta_0)
+ R_{30}(\zeta_1,\zeta_0,\zeta_0)
= 3 F_{30000}, \\
\mathcal{O}(A^2\overline{A}) &: (\mathbb{L} - i \omega) F_{20100} 
+ R_{20}(\zeta_0,F_{10100}) + R_{20}(F_{10100},\zeta_0) 
+ R_{20}(\overline{\zeta_0},F_{20000}) + R_{20}(F_{20000},\overline{\zeta_0}) \\
&{\qquad} + R_{30}(\zeta_0,\zeta_0,\overline{\zeta_0})
+ R_{30}(\zeta_0,\overline{\zeta_0},\zeta_0)
+ R_{30}(\overline{\zeta_0},\zeta_0,\zeta_0)
= i \beta \zeta_0 + b \zeta_1, \\
\mathcal{O}(A^2\overline{B}) &: (\mathbb{L} - i \omega) F_{20010} 
+ R_{20}(\zeta_0,F_{10010}) + R_{20}(F_{10010},\zeta_0) 
+ R_{20}(\overline{\zeta_1},F_{20000}) + R_{20}(F_{20000},\overline{\zeta_1})  \\
&{\qquad} + R_{30}(\zeta_0,\zeta_0,\overline{\zeta_1})
+ R_{30}(\zeta_0,\overline{\zeta_1},\zeta_0)
+ R_{30}(\overline{\zeta_1},\zeta_0,\zeta_0)
= F_{20100} - \frac{1}{2}\gamma \zeta_0 + \frac{1}{2} i c \zeta_1, \\
\mathcal{O}(AB^2) &: (\mathbb{L} - 3i \omega) F_{12000} 
+ R_{20}(\zeta_0,F_{02000}) + R_{20}(F_{02000},\zeta_0) 
+ R_{20}(\zeta_1,F_{11000}) + R_{20}(F_{11000},\zeta_1) \\
&{\qquad  } + R_{30}(\zeta_0,\zeta_1,\zeta_1)
+ R_{30}(\zeta_1,\zeta_0,\zeta_1)
+ R_{30}(\zeta_1,\zeta_1,\zeta_0)
= 2 F_{21000}, \\
\mathcal{O}(AB\overline{A}) &: (\mathbb{L} - i \omega) F_{11100} 
+ R_{20}(\zeta_0,F_{01100}) + R_{20}(F_{01100},\zeta_0) 
+ R_{20}(\overline{\zeta_0},F_{11000}) + R_{20}(F_{11000},\overline{\zeta_0}) \\
&{\qquad  } 
+ R_{20}(\zeta_1,F_{10100}) + R_{20}(F_{10100},\zeta_1) 
+ R_{30}(\zeta_0,\overline{\zeta_0},\zeta_1)
+ R_{30}(\zeta_0,\zeta_1,\overline{\zeta_0})
+ R_{30}(\overline{\zeta_0},\zeta_0,\zeta_1) \\
&{\qquad }
+ R_{30}(\overline{\zeta_0},\zeta_1,\zeta_0)
+ R_{30}(\zeta_1,\zeta_0,\overline{\zeta_0})
+ R_{30}(\zeta_1,\overline{\zeta_0},\zeta_0)
= 2 F_{20100} + \frac{1}{2}\gamma \zeta_0 + i \left(\beta - \frac{1}{2}c\right) \zeta_1, \\
\mathcal{O}(AB\overline{B}) &: (\mathbb{L} - i \omega) F_{11010} 
+ R_{20}(\zeta_0,F_{01010}) + R_{20}(F_{01010},\zeta_0) 
+ R_{20}(\zeta_1,F_{10010}) + R_{20}(F_{10010},\zeta_1) \\
&{\qquad  } 
+ R_{20}(\overline{\zeta_1},F_{11000}) + R_{20}(F_{11000},\overline{\zeta_1}) 
+ R_{30}(\zeta_0,\zeta_1,\overline{\zeta_1})
+ R_{30}(\zeta_0,\overline{\zeta_1},\zeta_1)
+ R_{30}(\zeta_1,\zeta_0,\overline{\zeta_1}) \\
&{\qquad }
+ R_{30}(\zeta_1,\overline{\zeta_1},\zeta_0)
+ R_{30}(\overline{\zeta_1},\zeta_0,\zeta_1)
+ R_{30}(\overline{\zeta_1},\zeta_1,\zeta_0)
= 2 F_{20010} + F_{11100} - \frac{1}{2}\gamma \zeta_1, \\
\mathcal{O}(B^2\overline{A}) &: (\mathbb{L} - i \omega) F_{02100} 
+ R_{20}(\zeta_1,F_{01100}) + R_{20}(F_{01100},\zeta_1) 
+ R_{20}(\overline{\zeta_0},F_{02000}) + R_{20}(F_{02000},\overline{\zeta_0}) \\
&{\qquad} + R_{30}(\overline{\zeta_0},\zeta_1,\zeta_1)
+ R_{30}(\zeta_1,\overline{\zeta_0},\zeta_1)
+ R_{30}(\zeta_1,\zeta_1,\overline{\zeta_0})
= \frac{1}{2} \gamma \zeta_1 + F_{11100}, \\
\mathcal{O}(B^2\overline{B}) &: (\mathbb{L} - i \omega) F_{02010} 
+ R_{20}(\zeta_1,F_{01010}) + R_{20}(F_{01010},\zeta_1) 
+ R_{20}(\overline{\zeta_1},F_{02000}) + R_{20}(F_{02000},\overline{\zeta_1}) \\
&{\qquad} + R_{30}(\overline{\zeta_1},\zeta_1,\zeta_1)
+ R_{30}(\zeta_1,\overline{\zeta_1},\zeta_1)
+ R_{30}(\zeta_1,\zeta_1,\overline{\zeta_1})
= F_{02100} + F_{11010}. \\
\end{split}
\end{equation}
The equations for the terms at
$\mathcal{O}(\overline{A}^3)$,
$\mathcal{O}(\overline{B}^3)$,
$\mathcal{O}(A \overline{A}^2)$,
$\mathcal{O}(A \overline{A}\, \overline{B})$,
$\mathcal{O}(A \overline{B}^2)$,
$\mathcal{O}(B \overline{A}^2)$,
$\mathcal{O}(B \overline{A}\, \overline{B})$,
$\mathcal{O}(B \overline{B}^2)$,
$\mathcal{O}(\overline{A}^2 \overline{B})$, and
$\mathcal{O}(\overline{A}\, \overline{B}^2)$
are the complex conjugates of those at
$\mathcal{O}(A^3)$,
$\mathcal{O}(B^3)$,
$\mathcal{O}(A^2 \overline{A})$,
$\mathcal{O}(AB \overline{A})$,
$\mathcal{O}(\overline{A} B^2)$,
$\mathcal{O}(A^2 \overline{B})$,
$\mathcal{O}(AB\overline{B})$,
$\mathcal{O}(B^2 \overline{B})$, 
$\mathcal{O}(A^2 B)$, and
$\mathcal{O}(A B^2)$,
respectively.

\section{Proof of Proposition~\ref{prop:twist}}\label{app:proptwist}

In this appendix, we prove Proposition~\ref{prop:twist} that establishes the properties of solutions near the true and faux canards $\Gamma_t^0$ and $\Gamma_f^0$ of $M_{\rm RFS}$.
We begin by changing variables in \eqref{eq:K2} to rectify the algebraic solution $\Gamma_t^0$ to the $q_2$-axis. 
Let 
$(u_2,p_2,v_2,q_2) 
= (\tilde{u}_2 - \frac{1}{2}\sqrt{\lambda} \tilde{q}_2, \,
\tilde{p}_2 - \frac{1}{2} \sqrt{\lambda}, \,
\tilde{v}_2 - \frac{1}{4} \lambda \tilde{q}_2^2, \, 
-\frac{1}{2} \lambda \tilde{q}_2)$,
where we recall that $\lambda>0$ for $M_{\rm RFS}$.
Hence, in the $(\tilde{u}_2,\tilde{p}_2,\tilde{v}_2,\tilde{q}_2)$ coordinate system, the algebraic solution is given by $\tilde{\Gamma}_t (y_2) = (0,0,0,y_2)$ for all $y_2 \in \mathbb{R}$.

In these new variables, system~\eqref{eq:K2} becomes
\begin{equation}
    \begin{split}
\frac{d\tilde{u}_2}{dy_2} &= \tilde{p}_2 + \frac{1}{2} \sqrt{\lambda} r_2^2 \tilde{k}_2(\tilde{u}_2,\tilde{v}_2,\tilde{q}_2,r_2) \\
\frac{d\tilde{p}_2}{dy_2} &= \tilde{v}_2 + \tilde{u}_2^2 - \lambda \tilde{u}_2 \tilde{q}_2 + r_2^2 \tilde{h}_2(\tilde{u}_2,\tilde{v}_2,\tilde{q}_2,r_2) \\
\frac{d\tilde{v}_2}{dy_2} &= \frac{1}{2} \lambda r_2^2 \tilde{q}_2 \tilde{k}_2(\tilde{u}_2,\tilde{v}_2,\tilde{q}_2,r_2) \\
\frac{d\tilde{q}_2}{dy_2} &= 1 + r_2^2 \tilde{k}_2(\tilde{u}_2,\tilde{v}_2,\tilde{q}_2,r_2).
    \end{split}
\end{equation}
Here,
\begin{equation*}
    \begin{split}
    \tilde{h}_2 (\tilde{u}_2,\tilde{v}_2,\tilde{q}_2,r_2) 
    &= \alpha_1 \left( \tilde{u}_2 - \frac{1}{2}\sqrt{\lambda} \tilde{q}_2\right)\left(\tilde{v}_2 - \frac{1}{4} \lambda \tilde{q}_2^2\right) 
    +\alpha_2 \left(\tilde{u}_2 - \frac{1}{2}\sqrt{\lambda} \tilde{q}_2\right)^3 + \mathcal{O}(r_2^2),\\
     \tilde{k}_2 (\tilde{u}_2,\tilde{v}_2,\tilde{q}_2,r_2) 
    &= -\frac{2\alpha_3}{\lambda}  \left( \tilde{u}_2 - \frac{1}{2}\lambda \tilde{q}_2\right)
    + \mathcal{O}(r_2^2).\\
    \end{split}
\end{equation*}
Equivalently, in terms of the vector ${\bf u}_2= (\tilde{u}_2,\tilde{p}_2,\tilde{v}_2)^t$, the system may be written in non-autonomous form with $\tilde{q}_2$ as the independent variable,
\begin{equation}\label{eq:nonauto}
    \begin{split}
\frac{d\tilde{\bf{u}}_2}{d\tilde{q}_2} &= 
A(\tilde{q}_2,\sqrt{\lambda}) {\bf u}_2 + {\bf g} (\tilde{q}_2,\tilde{\bf u}_2,r_2, \sqrt{\lambda}),
    \end{split}
\end{equation}
where
\begin{equation*}
\begin{split} 
    &A(\tilde{q}_2,\sqrt{\lambda}) 
    = \left[
    \begin{array}{ccc}
          0 & 1 & 0 \\
          -\sqrt{\lambda} \tilde{q}_2 & 0 &  1 \\
          0 & 0 & 0 
    \end{array}
    \right], \quad {\rm and}\\
    &{\bf g}(\tilde{q}_2,\tilde{\bf u}_2,r_2,\sqrt{\lambda}) 
    = 
    \left[
    \begin{array}{c}
        r_2^2 \left( \frac{1}{2}\sqrt{\lambda} - \tilde{p}_2 \right)\tilde{k}_2 + \mathcal{O}(r_2^4) \\
        \tilde{u}_2^2 + r_2^2 \left( \tilde{h}_2 - (\tilde{v}_2 + \tilde{u}_2^2 -\sqrt{\lambda}\tilde{q}_2 \tilde{u}_2) \tilde{k}_2 \right) + \mathcal{O}(r_2^4) \\
        \frac{1}{2} \lambda r_2^2 \tilde{q}_2 \tilde{k}_2 + \mathcal{O}(r_2^4 \tilde{q}_2)
    \end{array}
    \right].
    \end{split}
\end{equation*}

Next, we see from \eqref{eq:nonauto} that the variational equation along $\tilde{\Gamma}_t$ is
\begin{equation*}
    \left[
    \begin{array}{c} 
    \frac{d\eta_1}{d\tilde{q}_2} \\
    \frac{d\eta_2}{d\tilde{q}_2}
    \end{array}
    \right]
 = \left[
 \begin{array}{cc}
0 & 1 \\
-\sqrt{\lambda} \tilde{q}_2 & 0 
 \end{array}
 \right]
\left[ 
\begin{array}{c}
\eta_1 \\
\eta_2
\end{array}
\right]
+ \left[
\begin{array}{c}
0 \\
\eta_3
\end{array}
\right], 
\end{equation*}
where $\eta_3$ is a constant.
This system may be written as an inhomogeneous (scalar) Airy equation,
$\frac{d^2\eta_1}{d\tilde{q}_2^2} + \lambda \tilde{q}_2 \eta_1 = \eta_3$.
The general solutions are
\begin{equation}\label{eq:solutions-eov}
\begin{split}
    \eta_1(\tilde{q}_2) 
    &= c_1 Ai(-\lambda^{1/6}\tilde{q}_2)
    + c_2 Bi(-\lambda^{1/6}\tilde{q}_2)
    - \pi \lambda^{-1/3} \zeta_3 Gi(-\lambda^{1/6}\tilde{q}_2) \, \, \, {\rm for} \, \, \tilde{q}_2<0, \\
    \eta_1(\tilde{q}_2) 
    &= c_1 Ai(-\lambda^{1/6}\tilde{q}_2)
    + c_2 Bi(-\lambda^{1/6}\tilde{q}_2)
    + \pi \lambda^{-1/3} \zeta_3 Hi(-\lambda^{1/6}\tilde{q}_2) \, \, \, {\rm for} \, \, \tilde{q}_2>0. 
\end{split}
\end{equation}
Here, we recall that $Ai(z)$ and $Bi(z)$ denote the Airy functions of the first and second kinds, $Gi(z)$ and $Hi(z)$ denote the inhomogeneous Airy functions, which are also known as Scorer functions, and $c_1$ and $c_2$ are arbitrary constants. Also, we recall that $Gi(z)+Hi(z) = Bi(z)$.

The Airy functions $Ai(z)$ and $Bi(z)$ have infinitely many zeros on the real $z$-axis, all of which are negative, \cite{DLMF}.
$Gi(z)$ also has infinitely many negative real zeros and no non-negative zeros,
while $Hi(z)$ has no real zeroes.
Hence, recalling that the argument of the Airy functions is $z=-\lambda^{1/6} \tilde{q}_2$, we see that the statements about solutions near $\Gamma_t^0$ in the proposition follow. 
On the orbit segments along $S_s^0$, $\tilde{q}_2<0$ and hence $z>0$,
whereas on segments near $S_c^0$, $\tilde{q}_2>0$ and hence $z<0$.

The analysis for the local dynamics near the other algebraic solution $\Gamma_f^0$ may be derived in a similar fashion.
One may use a similar coordinate change to straighten out $\Gamma_f^0$ along the $q_2$-axis.
The variational equation is the same as that about $\Gamma_t$ except that the lower left entry of the matrix is $+\sqrt{\lambda} \tilde{q}_2$.
The general solutions are
\begin{equation*}
\begin{split}
    \eta_1(\tilde{q}_2) 
    &= c_1 Ai(\lambda^{1/6}\tilde{q}_2)
    + c_2 Bi(\lambda^{1/6}\tilde{q}_2)
    - \pi \lambda^{-1/3} \zeta_3 Gi(\lambda^{1/6}\tilde{q}_2) \, \, \, {\rm for} \, \, \tilde{q}_2>0, \\
    \eta_1(\tilde{q}_2) 
    &= c_1 Ai(\lambda^{1/6}\tilde{q}_2)
    + c_2 Bi(\lambda^{1/6}\tilde{q}_2)
    + \pi \lambda^{-1/3} \zeta_3 Hi(\lambda^{1/6}\tilde{q}_2) \, \, \, {\rm for} \, \, \tilde{q}_2<0. 
\end{split}
\end{equation*}
Hence, solutions near $\Gamma_f^0$ can exhibit infinitely many twists near $S^0_s$ where $\tilde{q}_2<0$ (and hence $z<0$), whereas they exhibit no twists near $S^0_c$ where $\tilde{q}_2>0$.
\hfill$\Box$

\section{Proof of Proposition \texorpdfstring{\ref{prop:roleofGamma20}}{Lg}}\label{app:secondgeodesing}

In this appendix, we prove Proposition~\ref{prop:roleofGamma20}.
We drop the hats and tildes on the variables in \eqref{eq:systemonH2zero} for convenience, so that the system is
\begin{equation}\label{eq:appD}
    \begin{split}
        \frac{du_2}{d\xi_2} &= u_2 p_2 \\
        \frac{dp_2}{d\xi_2} &= \tfrac12 \left( p_2^2 + \alpha_3 q_2^2\right) + \tfrac23 u_2^3 \\
        \frac{dq_2}{d\xi_2} &= -u_2^2.
    \end{split}
\end{equation}
The lines $\mathcal L_{\pm}$ of equilibria intersect in a nilpotent fixed point at the origin. 
To study the dynamics around this nilpotent equilibrium, we use the blow-up transformation
\begin{equation}\label{eq:appDdesing}
u_2 = \rho \bar{u}_2, \quad
p_2 = \rho \bar{p}_2, \quad 
q_2 = \rho \bar{q}_2,
\end{equation}
which is a map $\Psi: \mathbb S^2 \times [-\rho_0,\rho_0] \to \mathbb R^3$, where $\overline u_2^2 + \overline p_2^2 + \overline q_2^2 = 1$ and $\rho_0$ is sufficiently small. 
In this setting, Proposition~\ref{prop:roleofGamma20} may be stated as 
\begin{proposition}\label{prop:roleofGamma20appendix}
Restricted to $\{ \bar{u}_2  \ge 0 \}$, system \eqref{eq:systemonH2zero} possesses three classes of heteroclinics:
\begin{itemize}
    \item Class 1: solutions that emanate from $\mathcal{L}_+^u \cap \mathbb{S}^2$, travel over the hemisphere $\{ \bar{u}_2 > 0 \}$ in the region enclosed by $\Gamma_0^\pm$, and terminate at the point $\mathcal{L}_+^s \cap \mathbb{S}^2$ on the equator $\{ \bar{u}_2=0 \}$;
     \item Class 2: solutions that emanate from $\mathcal{L}_+^u \cap \mathbb{S}^2$, travel over the hemisphere $\{ \bar{u}_2 > 0 \}$ in the region enclosed by $\Gamma_0^+$ and the equator, and terminate at the point $\mathcal{L}_-^s \cap \mathbb{S}^2$ on the equator;
     \item Class 3: solutions that emanate from $\mathcal{L}_-^u \cap \mathbb{S}^2$, travel over the hemisphere $\{ \bar{u}_2 > 0 \}$ in the region enclosed by $\Gamma_0^-$ and the equator, and terminate at the point $\mathcal{L}_+^s \cap \mathbb{S}^2$ on the equator.
     \end{itemize}
     \noindent 
In addition, the stable manifold, $\Gamma_{20}^+$, of the equilibrium located at the intersection of the negative $\hat{p}_2$-axis and $\mathbb{S}^2$ is the separatrix that divides between class 1 and 2 heteroclinics.
The unstable manifold, $\Gamma_{20}^-$, of the equilibrium located at the intersection of the positive $\hat{p}_2$-axis and $\mathbb{S}^2$ is the separatrix that divides between class 1 and 3 heteroclinics.
\end{proposition}

To prove Proposition~\ref{prop:roleofGamma20appendix}, we analyse the dynamics of \eqref{eq:appD} in the blown-up coordinates \eqref{eq:appDdesing}. We work in three coordinate charts
\begin{equation}\label{eq:appD3charts}
    \begin{split}
       {\rm central}: \quad & K_{21} = \{\bar{u}_2 = 1 \} : u_2=\rho_{21}, \, \, \, p_{21} = \rho_{21} p_{21}, \, \, \, q_2=\rho_{21} q_{21}  \\
       {\rm back}: \quad & K_{22} = \{ \bar{p}_2 = 1 \}: u_2 = \rho_{22} u_{22}, \, \, \, p_2=\rho_{22}, \, \, \, q_2=\rho_{22} q_{22}, \\
       {\rm bottom}: \quad & K_{23} = \{ \bar{q}_2 = -1 \} : u_2 = \rho_{23} u_{23}, \, \, \, 
       p_2=\rho_{23} p_{23}, \, \, \, q_2=-\rho_{23}.
       \end{split}
\end{equation}

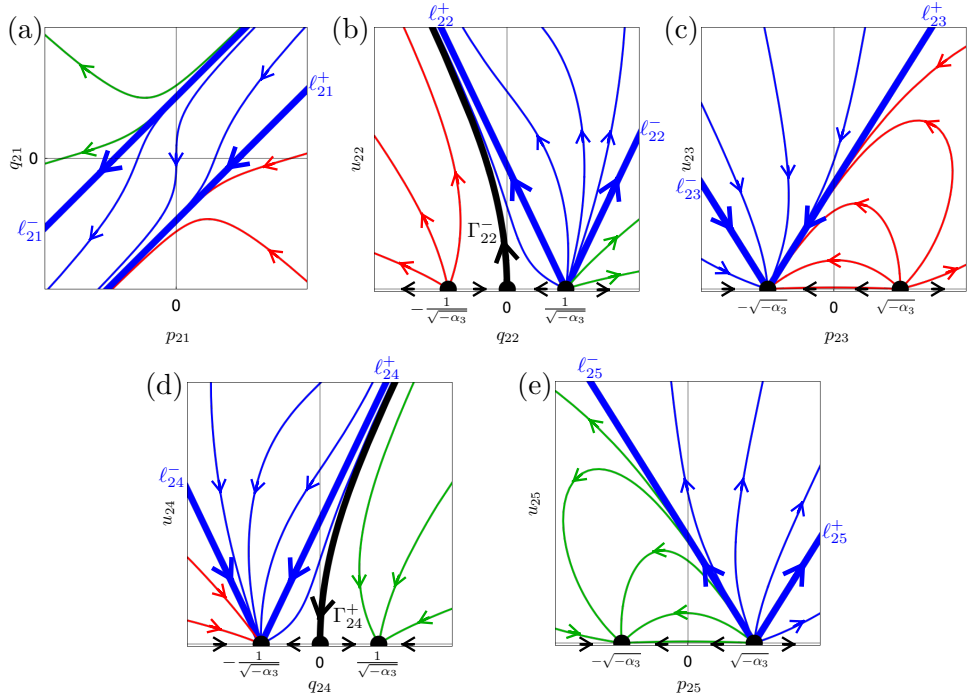
\begin{figure}[ht!]
    \centering
    \includegraphics[width=5in]{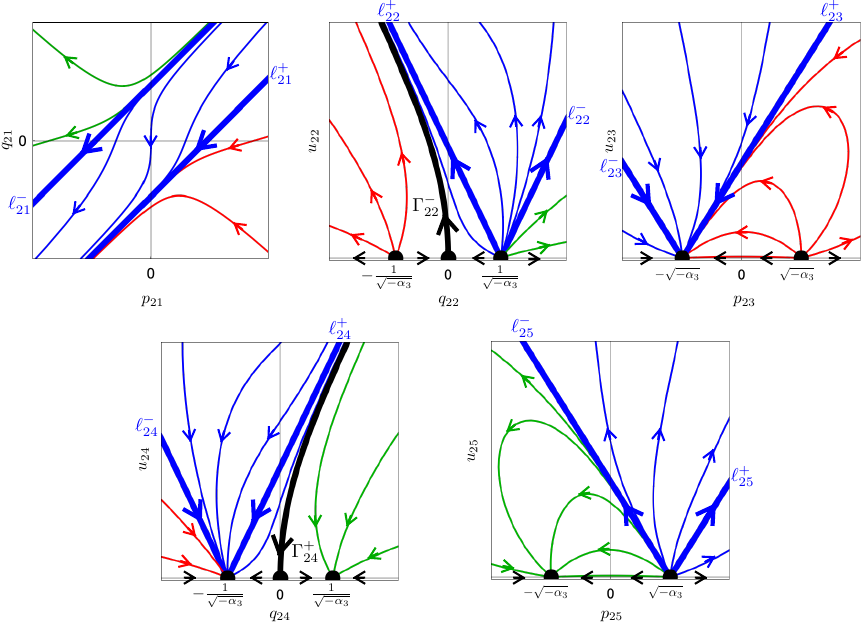}
    \put(-362,248){(a)}
    \put(-240,248){(b)}
    \put(-115,248){(c)}
    \put(-310,114){(d)}
    \put(-170,114){(e)}
    \caption{Dynamics of the unperturbed problem ($\rho_{2k} = 0$ for $k=1,2,\ldots,5$) in the different coordinate charts. 
    (a) In chart $K_{21}: \{ \overline u_2 = 1\}$, the line $\ell_{21}^+$ is attracting and the line $\ell_{21}^-$ is repelling. 
    (b) In chart $K_{22} : \{ \overline p_2 = 1 \}$, there is a pair of unstable nodes at $(u_{22},q_{22})=(0,\pm \tfrac{1}{\sqrt{-\alpha_3}})$ and a saddle at the origin. 
    The invariant line $\ell_{22}^-$ separates the class 1 (blue) and class 2 (green) heteroclinics. 
    The unstable manifold $\Gamma_{22}^-$ of the saddle divides between solutions corresponding to class 1 (blue) and class 3 (red) heteroclinics.
    (c) In chart $K_{23}: \{ \overline q_2=-1 \}$, there is a stable node at $(u_{23},p_{23})=(0,-\sqrt{-\alpha_3})$ and an unstable node at $(u_{23},p_{23})=(0,\sqrt{-\alpha_3})$. The invariant line $\ell_{23}^+$ separates the class 1 (blue) and class 3 (red) heteroclinics. 
    (d) In chart $K_{24}: \{ \overline p_2 = -1 \}$, there is a pair of stable nodes at $(u_{24},q_{24}) = (0,\pm \tfrac{1}{\sqrt{-\alpha_3}})$ and a saddle at the origin. The stable manifold $\Gamma_{24}^+$ of the saddle divides between solutions corresponding to class 1 (blue) and class 2 (green) heteroclinics. 
    The invariant line $\ell_{24}^-$ separates the class 1 (blue) and class 3 (red) heteroclinics. 
    (e) In chart $K_{25}: \{ \overline q_2 = 1 \}$, there is a stable node at $(u_{25},p_{25})=(0,-\sqrt{-\alpha_3})$ and an unstable node at $(u_{25},p_{25})=(0,\sqrt{-\alpha_3})$. The invariant line $\ell_{25}^-$ separates the class 1 (blue) and class 2 (green) heteroclinics.
    }
    \label{fig:secondblowupcharts}
\end{figure}

The dynamics in the central chart $K_{21}$ are governed by
\begin{equation}\label{eq:chartK21}
\begin{split}
    \dot{\rho}_{21} &= \rho_{21} p_{21}\\
    \dot{p}_{21} &= \tfrac12 (\alpha_3 q_{21}^2 - p_{21}^2) +\tfrac{2}{3}\rho_{21} \\
    \dot{q}_{21} &= -( 1 + p_{21}q_{21}),
\end{split}
    \end{equation}
where we have rescaled the independent variable by a factor of $\rho_{21}$, and the overdot represents the derivative with respect to the rescaled variable ($\eta_2$).
The subspace $\{ \rho_{21} =0 \}$ is invariant.
On it, there are attracting and repelling lines,
\begin{equation}
\begin{split} 
    \ell_{21}^+ &= \{ p_{21}- \sqrt{-\alpha_3} q_{21} = \sqrt{2}(-\alpha_3)^{1/4} \} \\
 \ell_{21}^- &= \{ p_{21}- \sqrt{-\alpha_3} q_{21} = - \sqrt{2}(-\alpha_3)^{1/4} \}.
    \end{split}
\end{equation}
In fact, all solutions on this plane may be found explicitly.
Let $z_{21}=\tfrac12 (p_{21} - \sqrt{-\alpha_3} q_{21})$ and 
$w_{21}=\tfrac12 (p_{21} + \sqrt{-\alpha_3} q_{21})$.
Then, the system decouples:
\begin{equation*}
   \begin{split}
       \dot{z}_{21} &=-z_{21}^2 + \tfrac{1}{2}\sqrt{-\alpha_3} \\
       \dot{w}_{21} &= -w_{21}^2 - \tfrac{1}{2}\sqrt{-\alpha_3}.
   \end{split}
\end{equation*}
Hence, the solutions are
\begin{equation*}
\begin{split}
z_{21}(\eta_{2})&= \frac{(-\alpha_3)^{1/4}}{\sqrt{2}} \tanh 
\left( \frac{(-\alpha_3)^{1/4}}{\sqrt{2}} \eta_2 + {\rm tanh}^{-1} \left( 
\frac{\sqrt{2}}{(-\alpha_3)^{1/4}} z_{21}(0) \right) \right), \\
w_{21}(\eta_{2})&= - \frac{(-\alpha_3)^{1/4}}{\sqrt{2}} \tan
\left( \frac{(-\alpha_3)^{1/4}}{\sqrt{2}} \eta_2 - {\rm tan}^{-1} \left( 
\frac{\sqrt{2}}{(-\alpha_3)^{1/4}} w_{21}(0) \right) \right).
\end{split}
\end{equation*}
The dynamics are illustrated in Figure~\ref{fig:secondblowupcharts}(a).

The dynamics in chart $K_{22}$ are governed by
\begin{equation}\label{eq:K22}
    \begin{split}
        \dot{u}_{22} &= \tfrac12 u_{22} \left(
        1 - \alpha_3 q_{22}^2 - \tfrac{4}{3}\rho_{22} u_{22}^3 \right) \\
        \dot{\rho}_{22} &= \tfrac12 \rho_{22} \left( 1 + \alpha_3 q_{22}^2 + \tfrac{4}{3} \rho_{22} u_{22}^3 \right)\\
        \dot{q}_{22} &= -u_{22}^2 - \tfrac12 q_{22} \left( 1 + \alpha_3 q_{22}^2 + \tfrac43  \rho_{22} u_{22}^3
        \right),
    \end{split}
\end{equation}
where we rescaled the independent variable by a factor of $\rho_{22}$ and recycled the overdot.
The plane $\{ \rho_{22}=0 \}$ is invariant.
On it, the system has three equilbria,
\begin{equation*} 
(u_{22},q_{22})= (0,0), \, \, \, (0, \pm \tfrac{1}{\sqrt{-\alpha_3}}).
\end{equation*} 
The origin is a saddle with spectra $\left\{ -\tfrac12, +\tfrac12 \right\}$
and eigenvectors corresponding to the $q_2$- and $u_2$-axes, respectively.
The other equilibria are unstable improper nodes, with repeated eigenvalues of $1$.

In addition, the system on this invariant plane has two invariant lines
\begin{equation}
\begin{split} 
    \ell_{22}^+ &= \{ \sqrt{2}(-\alpha_3)^{1/4} u_{22} + \sqrt{-\alpha_3} q_{22} = 1 \} \\
    \ell_{22}^- &= \{ \sqrt{2}(-\alpha_3)^{1/4} u_{22} - \sqrt{-\alpha_3} q_{22} = -1 \}.
\end{split}
\end{equation}
These lines are the images of the invariant lines in chart $K_{21}$:
\begin{equation*}
    \ell_{22}^\pm = \Pi_{12}(\ell_{21}^\pm),
\end{equation*}
as may be seen by using the transition map 
\begin{equation}\label{eq:Pi12}
\Pi_{12}(\rho_{21},p_{21},q_{21}) = (u_{22},\rho_{22},q_{22}) 
= \left(
\frac{1}{p_{21}}, \rho_{21}p_{21}, \frac{q_{21}}{p_{21}}
\right) \quad {\rm for} \, \, \, p_{21}>0.
\end{equation}
The dynamics on $\{ \rho_{22}=0 \}$ are illustrated in 
Figure~\ref{fig:secondblowupcharts}(b).

The dynamics in the bottom chart $K_{23}$ are governed by
\begin{equation}\label{eq:chartK23}
\begin{split}
    \dot{u}_{23} &= u_{23} (p_{23} - u_{23}^2)\\
    \dot{p}_{23} &= \tfrac12 \alpha_3 - u_{23}^2 p_{23} + \tfrac{1}{2} p_{23}^2 + \tfrac23 \rho_{23} u_{23}^3 \\
    \dot{\rho}_{23} &= \rho_{23} u_{23}^2,
\end{split}
\end{equation}
where we have rescaled the independent variable by a factor of $\rho_{23}$, and the overdot has been recycled again.
The subspace $\{ \rho_{23} = 0 \}$ is invariant.
On it, the system has two fixed points,
$(u_{23},p_{23})=(0, \pm \sqrt{-\alpha_3})$, which are unstable/stable improper nodes, respectively, with repeated eigenvalues $\pm \sqrt{-\alpha_3}$, respectively.

In addition, the system on this invariant plane has two invariant lines
\begin{equation}
\begin{split} 
    \ell_{23}^+ &= \{ -\sqrt{2}(-\alpha_3)^{1/4} u_{23} + p_{23} = -\sqrt{-\alpha_3} \} \\
    \ell_{23}^- &= \{ \sqrt{2}(-\alpha_3)^{1/4} u_{22} + p_{23} = - \sqrt{-\alpha_3} \}.
\end{split}
\end{equation}
These lines are the images of the invariant lines in chart $K_{21}$,
\begin{equation*}
    \ell_{23}^\pm = \Pi_{13}(\ell_{21}^\pm),
\end{equation*}
as may be seen by using the transition map 
\begin{equation}\label{eq:Pi13}
\Pi_{13}(\rho_{21},p_{21},q_{21}) = (u_{23},p_{23},\rho_{23}) 
= \left(
-\frac{1}{q_{21}}, -\frac{p_{21}}{q_{21}}, -\rho_{21}q_{21}
\right) \quad {\rm for} \, \,\, q_{21}<0.
\end{equation}
The dynamics on $\{ \rho_{23}=0 \}$ are illustrated in Figure~\ref{fig:secondblowupcharts}(c).

The dynamics in the coordinate charts $K_{24}=\{ \bar{p}_2=-1 \}$ and $K_{25}=\{ \bar{q}_2=1 \}$ may be obtained by using the reversibility symmetry $(u_2,p_2,q_2,\xi_2) \mapsto (u_2,-p_2,-q_2,-\xi_2)$. We show the unperturbed versions of these (i.e., in the invariant subspaces $\{\rho_{24}=0\}$ and $\{ \rho_{25} = 0 \}$) in Figures~\ref{fig:secondblowupcharts}(d) and (e), respectively.

Next, we find the representations of the algebraic solution $\Gamma_{20}$ in the charts $K_{21}$, $K_{22}$, and $K_{23}$.
In chart $K_{21}$, we use \eqref{eq:Gamma_20} and the first equation in \eqref{eq:appDdesing} to find that it is given in parametric form by
$\Gamma_{21} \hspace{-0.04truein}: (\rho_{21},p_{21}) = \left( - \frac{3\alpha_3}{4} q_{21}^2, -\frac{2}{3 q_{21}}\right)$.
Then, in $K_{22}$, it is given by
$\Gamma_{22}\hspace{-0.04truein}: (\rho_{22},q_{22}) = \left( - \tfrac{1}{3}\alpha_3 u_{22}, -\tfrac{2}{3} u_{22}^2 \right)$,
as may be verified by using \eqref{eq:Pi12} to map the representation from $K_{21}$ to $K_{22}$. 
Hence, $\Gamma_{22}$ emanates from the origin $(u_{22},q_{22})=(0,0)$, as shown in Figure~\ref{fig:secondblowupcharts}(b). 
Lastly, in $K_{23}$, it is given by 
$\Gamma_{23}\hspace{-0.04truein}: (p_{23},\rho_{23}) = \left( \tfrac{2}{3} u_{23}^2, -\tfrac{3\alpha_3}{4u_{23}^3} \right)$,
as may be verified by using \eqref{eq:Pi13} to map the representation from $K_{21}$ to $K_{23}$. 
(Also, one may obtain explicit representations in the charts using the dynamic independent variable for each chart.)

The first class of heteroclinic solutions (shown in blue in Figures~\ref{fig:secondblowup} and \ref{fig:secondblowupcharts}) consists of those that emanate from $(u_{22},q_{22})=(0,\frac{1}{\sqrt{-\alpha_3}})$ in $K_{22}$, which is the point on the equator at which the line $\overline{\mathcal{L}}^u_+$ of fixed points intersects $\mathbb{S}^2$, and that are forward asymptotic to (terminate in) the equilibrium at $(u_{23},p_{23})=(0,-\sqrt{-\alpha_3})$ on the equator, where the invariant line $\overline{\mathcal{L}}_+^s$ intersects $\mathbb{S}^2$.
These heteroclinics are bounded by $\overline{\Gamma}_0^+$ and $\bar{\Gamma}_0^-$.
(Equivalently, they terminate in the point $(u_{24},q_{24})=(0,-\frac{1}{\sqrt{-\alpha_3}})$ in chart $K_{24}$.)

The second class of heteroclinics (shown in green in Figures~\ref{fig:secondblowup} and \ref{fig:secondblowupcharts}) consists of those that emanate from $(u_{22},q_{22})=(0,\frac{1}{\sqrt{-\alpha_3}})$ in $K_{22}$, where the line $\overline{\mathcal{L}}^u_+$ of fixed points intersects $\mathbb{S}^2$, and that are forward asymptotic to (or terminate in) the equilibrium at $(u_{24},q_{24})=(0,\frac{1}{\sqrt{-\alpha_3}})$ on the equator, where the invariant line $\overline{\mathcal{L}}_-^s$ intersects $\mathbb{S}^2$.
These heteroclinics are bounded by the equator and $\overline{\Gamma}_0^+$, which is represented by $\Gamma_{24}^+$ in $K_{24}$.
See also Figure~\ref{fig:secondblowupcharts}(d).

The third class of heteroclinics (shown in red) is comprised of those that emanate from $(u_{22},q_{22})=(0,-\frac{1}{\sqrt{-\alpha_3}})$ in $K_{22}$, where the line $\overline{\mathcal{L}}^u_-$ of fixed points intersects $\mathbb{S}^2$, and that are forward asymptotic to (or terminate in) the equilibrium at $(u_{24},q_{24})=(0,-\frac{1}{\sqrt{-\alpha_3}})$ on the equator, where the invariant line $\overline{\mathcal{L}}_+^s$ intersects $\mathbb{S}^2$.
These heteroclinics are bounded by the equator and $\overline{\Gamma}_0^-$, which is represented by $\Gamma_{22}^-$.

The stable manifold, $\Gamma_{20}^+$, of the equilibrium point where the negative $\bar{p}_2$-axis intersects $\mathbb{S}^2$ separates the class 1 and 2 heteroclinics.
Similarly, the unstable manifold, $\Gamma_{20}^-$, of the equilibrium point where the positive $\bar{p}_2$-axis intersects $\mathbb{S}^2$ separates the class 1 and 3 heteroclinics.
This completes the proof of the proposition.

\begin{remark}
Within the coordinate charts $K_{2j}$ for $j=2,3,4,5$, there is an additional invariant plane given by $\{u_{2j}=0\}$.
The dynamics in these invariant planes possess two invariant lines which correspond to the lines $\mathcal{L}_\pm$. 
\end{remark}

\begin{remark}
    In the geometric desingularization of the cusp point of the spatial canard solutions in the rescaling chart of the van der Pol PDE analyzed in Section 6.1.2 of \cite{VDK2025}, the numbering of the three classes of heteroclinics is established in Proposition 6.2 and Figure 19 of \cite{VDK2025}.
    The numbering used here is similar.
    Note that the numbering of classes 2 and 3 on page 2679 (Appendix B) of \cite{VDK2025} is inadvertently reversed.
\end{remark}

\section{Turing bifurcation in the classical Gierer-Meinhardt PDE \texorpdfstring{\eqref{eq:GM-nondim}}{Lg}}
\label{app:GM-Turing}

In the classical Gierer-Meinhardt PDE \eqref{eq:GM-nondim}, the conditions for the homogeneous state $(U_e,V_e)$ to undergo a Turing bifurcation are 
\begin{equation*}
    \det \widetilde{M} =0 \hskip0.2truein {\rm and} \hskip0.18truein \frac{d}{dk^2} \det \widetilde{M} = 0,
\end{equation*}
where 
$$
\widetilde{M} = \left[ 
\begin{array}{cc}
\frac{\gamma_1(1-\gamma_2)}{1+\gamma_2} - \delta^2 k^2 & -\frac{\gamma_1^2}{(1+\gamma_2)^2} \\
\frac{2}{\gamma_1}(1+\gamma_2) & -1 -k^2 
\end{array}
\right],
$$
along with $\gamma_1(1-\gamma_2) - 1 - \gamma_2 < 0$, 
so that the equilibrium is stable for the reaction kinetics. 
Hence,
\begin{equation*} 
\delta^2 k^4 + \left( \delta^2 - \left(\frac{1-\gamma_2}{1+\gamma_2}\right)\gamma_1\right) k^2 + \gamma_1 = 0, \hskip0.4truein 
2\delta^2 k^2 + \left(\delta^2 - \left( \frac{1-\gamma_2}{1+\gamma_2}\right) \gamma_1  \right) = 0,
\end{equation*}
with $\frac{\gamma_1}{2\delta^2}\left( \frac{1-\gamma_2}{1+\gamma_2} \right) > \frac{1}{2}$
so that $k_T^2 > 0$.
These conditions yield \eqref{eq:gamma_2T-GM}.
(Note that, with $k=0$, the determinant of $M$ is just $\gamma_1$, and hence positive by assumption.)

\section{Two additional prototypical systems: The Brusselator model and the van der Pol PDE}\label{app:examples}

In this appendix, we present two additional prototypical pattern-forming PDE systems of the form \eqref{eq:genAI-UV} with non-degenerate folds. 
These have been chosen from chemistry and electrical engineering, to complement the Gierer-Meinhardt example from biology.

\subsection{The Brusselator model} \label{ex:Brusselator}
Introduced in \cite{PL1968}, the Brusselator PDE model,  
    \begin{equation}\label{eq:Brusselator}
        \begin{split}
        \partial_t U &= \delta^2 \partial_x^2 U + A - (B+1) U + U^2V, \\
        \partial_t V &= \partial_x^2 V + BU - U^2V,
        \end{split}
    \end{equation}
is of the form \eqref{eq:genAI-UV} in the limit of small activator diffusivity. 
See also \cite{CH1993,EP1998,GS1997,Pena2001, S2003, SZJV2018, W1997}. 
Here, $F(U,V)= A -(B+1)U + U^2V$
    and $G(U,V)=BU - U^2V $, where $A,B>0$ are parameters.
    The null set of $F$ is given by $V_0(U)=\frac{(B+1)U - A}{U^2}$, and
    $(U_f,V_f)=\left(\frac{2A}{B+1},\frac{ (B+1)^2}{4A} \right)$ is a non-degenerate fold point on it.
    Also, the point $(U_e,V_e)=\left(A,\frac{B}{A}\right)$ is an equilibrium for all $A,B>0$. 
    
Locally near the fold point, the Brusselator model can be expressed in the form \eqref{eq:genAI}.
We find: $\frac{1}{2}\lambda=\frac{A^2(B-1)}{B+1}$,
    $\alpha_1= -\frac{4}{B+1}$, 
    $\alpha_2= 0$,
    $\alpha_3= -\frac{4A^2}{(B+1)^2}$,
    $\alpha_4= \frac{4A^2}{(B+1)^2}$, and
    $\alpha_5= \frac{4A^2}{(B+1)^2}$.
Hence, locally, the PDE has the form 
    \begin{equation}\label{eq:Brusselator-nf}
        \begin{split}
            \partial_t u &= \delta^2 \partial_x^2 u -v -u^2 +\frac{4}{B+1} uv - \frac{4}{(B+1)^2}u^2v, \\
        \partial_t v &= \partial_x^2 v +\frac{A^2}{(B+1)}
        \left\{ (B-1) + \frac{4}{(B+1)}(u-v-u^2) + \frac{16}{(B+1)^2} uv - \frac{16}{(B+1)^3} u^2v \right\}.
        \end{split}
    \end{equation}
Note that this system is exact ({\it i.e.,} there are no higher order terms), because the nonlinearities in \eqref{eq:Brusselator} are polynomials of degree three.
Now, the non-degenerate fold point is located at $(u_f,v_f)=(0,0)$ for each $B>0$, and the equilibrium is at $(u_e,v_e)=\left(-\frac{1}{4}(B^2-1),-\frac{1}{4}(B-1)^2\right)$.

The stability types of the folded singularity and the equilibrium were determined as a function of $B$ in \cite{JDKV2026} (and they also follow from the calculations here in Section~\ref{sec:layerproblem+criticalmanifold+desingularizedsystem}).
For each $B>1$, the fold point is a reversible folded saddle (RFS) of the spatial ODE system, while the equilibrium is a center.
Then, at $B=1$ ($\lambda=0$), the RFS coincides with the equilibrium in a reversible folded saddle-node of type II (RFSN-II).
In particular, the equilibrium passes through $(0,0)$ at $B=1$. 
Then, for each $B<1$, the fold point is a folded center, while the equilibrium is a saddle.
That is, the folded and ordinary singularities exchange stability type at $B=1$. 

The homogeneous state $(U_e,V_e)$ of the PDE \eqref{eq:Brusselator} undergoes a Turing bifurcation at 
\begin{equation}\label{eq:Brusselator-BT}
B_T = ( 1 + \delta A )^2.
\end{equation}
In $\lambda$, the Turing bifurcation occurs
at $\lambda_T=\frac{2\delta A^3 (2 + \delta A)}{2 + 2 \delta A + \delta^2 A^2}$.
Hence, $\lambda_T = 2 \delta A^3$, to leading order.

In the normal form for the 1:1 resonant Hopf bifurcation,
the coefficient that determines the criticality of the Turing bifurcation in the Brusselator PDE is given exactly by
\begin{equation}
  b = -\frac{1}{36} A^2 \left( 8 - 30 \delta A - 43 
  \delta^2 A^2 + 3 \delta^3 A^3 + 8 \delta^4 A^4 \right).  
\end{equation}
(Recall Proposition 2.1 in \cite{JDKV2026}.)
This quartic function for $b$ has zeros at $A= -\frac{2}{\delta}, -\frac{1}{\delta}, 0, \frac{21 \pm \sqrt{313}}{16\delta}$.
Hence, 
\begin{equation}
    \begin{split}
  A &\in \left(0, \frac{21-\sqrt{313}}{16\delta}\right) \ \ \Rightarrow
  \ \ b < 0, \ \ {\rm subcritical} \ \ {\rm Turing}.  \\
  A &\in \left( \frac{21-\sqrt{313}}{16\delta}, \frac{21 + \sqrt{313}}{16\delta}  \right)
  \ \ \Rightarrow
  \ \ b > 0, \ \ {\rm supercritical} \ \ {\rm Turing}.
    \end{split}
\end{equation}

Theorems~\ref{thm:1} and \ref{thm:2} establish the existence of spatial canard solutions and spatially-periodic canard solutions for the Brusselator PDE \eqref{eq:Brusselator} for $B\ge 1$ ($\lambda\ge 0$) and $\delta$ sufficiently small. 
Furthermore, the latter emerge from sub-critical and super-critical Turing bifurcations, depending on the sign of $b$.
Hence, the analysis in this article establishes the rigorous existence of
spatially-periodic canards similar to those found numerically in the sub-critical case in \cite{JDKV2026}, as well as the stable spatially-periodic canards created in the super-critical Turing bifurcation.

\begin{remark}
An alternate derivation of the formula \eqref{eq:Brusselator-BT} for $B_T$ may be made from the condition \eqref{eq:Turing-condition}, as follows. 
The model \eqref{eq:Brusselator-nf} is of the form \eqref{eq:genAI-mu} with
\begin{equation*}
\begin{split} 
f(u,v;B)&=-v - u^2 +\frac{4}{B+1}uv - \frac{4}{(B+1)^2}u^2v, \\
g(u,v;B)&=\frac{A^2}{(B+1)}
\left\{ 
(B-1) + \frac{4}{(B+1)}(u-v-u^2) + \frac{16}{(B+1)^2} uv - \frac{16}{(B+1)^3}u^2v
\right\},
\end{split}
\end{equation*}
and $\mu = B$. 
Then, at the equilibrium point $(u_e,v_e)$,
$f_u=B-1$,
$f_v=-\frac{1}{4}(B+1)^2$,
$g_u=\frac{4A^2B}{(B+1)^2}$,
and $g_v=-A^2$,
so that the Turing bifurcation condition \eqref{eq:Turing-condition} yields $B_T$.
\end{remark}

\subsection{The van der Pol PDE}
\label{ex:vanderPol}

The van der Pol PDE 
\begin{equation}\label{eq:vdPPDE}
        \partial_t U = \delta^2 \partial_x^2 U + V -\frac{1}{3} U^3 + U,
        \hskip0.3truein 
        \partial_t V = \partial_x^2 V + \eps(a-U)
\end{equation}
is also an example of the general class of activator-inhibitor systems \eqref{eq:genAI-UV}. 
Here, $F(U,V)= V - \frac{1}{3}U^3 + U$ and $G(U,V)=\varepsilon(a-U)$,
where $a \in \mathbb{R}$ and $0< \varepsilon < \infty$.
The null set of $F$ is given by the graph of $V_0(U)=\frac{1}{3}U^3 - U$.
There is a non-degenerate fold point at $(U_f,V_f)=\left(1,-\frac{2}{3}\right)$, and another (by symmetry) at $\left(-1,\frac{2}{3}\right)$.
Also, the point $\left(U_e,V_e)=(a,\frac{1}{3}a^3-a\right)$ is an equilibrium point.

Focusing on $(1, -\frac{2}{3})$, we find that the PDE is of the form \eqref{eq:genAI},
with $\frac{1}{2}\lambda = \varepsilon(1-a)$,
$\alpha_1=0$,
$\alpha_2=\frac{1}{3}$,
$\alpha_3=-\varepsilon$,
$\alpha_4=0,$
and $\alpha_5=0$.
Hence, the system is
\begin{equation}\label{eq:vdP-nf}
        \partial_t u = \delta^2 \partial_x^2 u -v -u^2 -\frac{1}{3}u^3, 
        \hskip0.3truein 
        \partial_t v = \partial_x^2 v +\eps(1-a) + \eps u.
\end{equation}
There are no higher order terms in either equation, since $F$ is linear in $V$ and cubic in $U$, and since $G$ is linear in $U$ and independent of $V$.
In this formulation, the fold point is at $(u_f,v_f)=(0,0)$ for all $a$, 
and the equilibrium is at $(u_e,v_e)=\left(a-1,-(a-1)^2 - \frac{1}{3}(a-1)^3\right)$.

In the spatial ODE system, the stability types of the fold point at $(u_f,v_f)=(0,0)$ (which corresponds to $(U_f,V_f)=\left(1,-\frac{2}{3}\right)$) and the equilibrium were reported as functions of $a$, see Table 1 in \cite{VDK2025}.
For each $a<1$ ($\lambda>0$), the folded singularity is an RFS, and the equilibrium $(u_e,v_e)$ is a center.
Then, for $a=1$ ($\lambda=0$), the fold point is an RFSN-II, where the fold point and the equilibrium merge. 
Finally, for each $a>1$ ($\lambda<0$), the folded singularity is a reversible folded center (RFC), while the equilibrium is a saddle.

The homogeneous state undergoes a Turing bifurcation at 
\begin{equation}\label{eq:vdP-aT}
a_T= \sqrt{ 1 - 2\delta \sqrt{\varepsilon}},
\end{equation}
recall \cite{VDK2025}.
This parameter value is a 1:1 resonant Hopf bifurcation point of the spatial ODE system, with $\omega_0=\eps^{1/4}$ for all $\eps>0$ (where $\eps$ need not be small).
In terms of $\lambda$, the Turing bifurcation occurs at
$\lambda_T= 2\eps \left(1 - \sqrt{1-2\delta\sqrt{\eps}}\right)$, and hence
$\lambda_T = 2 \delta \eps^{3/2}$, to leading order.

The criticality of this Turing bifurcation is determined by the sign of the coefficient
\begin{equation}
    b=\frac{1}{36}\left(-8 + 25\delta \sqrt{\eps}\right),
\end{equation}
which is an exact formula.
Hence, 
\begin{equation}
    \begin{split}
        \delta &< \frac{8}{25\sqrt{\eps}} \ \ \Rightarrow \ \  b<0 \ \ \ 
        {\rm subcritical} \ \ {\rm Turing}, \\
        \delta &> \frac{8}{25\sqrt{\eps}} \ \ \Rightarrow \ \ b>0 \ \ \ {\rm supercritical} \ \ {\rm Turing}.
    \end{split}
\end{equation}
Recall Section 2 in \cite{VDK2025}.

Theorems~\ref{thm:1} and \ref{thm:2} establish the existence of spatial canard solutions and spatially-periodic canard solutions for the van der Pol PDE \eqref{eq:vdPPDE} for $a \le 1$ ($\lambda\ge 0$) and $\delta$ sufficiently small. 
Hence, the analysis in this article establishes the rigorous existence of spatially-periodic
canards similar to those found numerically in \cite{VDK2025}.
The van der Pol PDE has extra structure, since the system of four first-order spatial ODEs is Hamiltonian, and the conserved energy makes it possible to also construct large-amplitude spatially-periodic canard solutions, which have segments along all three branches of the cubic-shaped critical manifold. 
See Figures 1, 2, 7--11 there, as well as the spatially periodic canard attractors observed in direct PDE simulations, shown in Section 11 there.

\begin{remark} 
An alternate derivation of the formula \eqref{eq:vdP-aT} for $a_T$ may be made from the analysis in this section, as follows. 
Let $f(u,v;a)=-v - u^2 - \frac{1}{3}u^3$ 
and $g(u,v;a)=\eps(1-a) + \eps u$, so that \eqref{eq:vdP-nf} is of the form \eqref{eq:genAI-mu} with $\mu = a$. 
Recall that $u_e=a-1$. 
Then, $\sqrt{-f_v g_u} + \delta g_v = \sqrt{\eps}$, so that the condition \eqref{eq:Turing-condition} for the Turing bifurcation becomes $-2u_e - u_e^2 = 2\delta\sqrt{\eps}$, which also yields $a_T$ in closed form.
\end{remark}

\vspace{0.5in}

\vspace{0.1in}
\noindent
{\bf Declaration about competing interests.}
The authors declare that they have no competing interests conducting the research and preparing the work for submission.

\vspace{0.1in}
\noindent
{\bf Funding information.}
The work of R.J. was supported in part by a research account sponsored by Boston University.

\printbibliography

@article{T1952,
author={Turing, A.M.},
title = {The chemical basis of morphogenesis},
journal = {Philosophical Transactions of the Royal Society of London B},
volume = {237(641)},
year = {1952}, 
pages = {37--72}
}

@article{AK2002,
author = {I.S. Aranson and L. Kramer},
title={The world of the complex Ginzburg-Landau equation},
journal = {Reviews of Modern Physics},
year = {2002},
pages = {99--143},
volume = {74}
}

@article{AACS2025,
  title={Slow passage through the Busse balloon--predicting steps on the Eckhaus staircase},
  author={Asch, A. and Avery, M. and Cortez, A. and Scheel, A.},
  journal={European Journal of Applied Mathematics},
  volume={36},
  number={1},
  pages={1--26},
  year={2025},
  publisher={Cambridge University Press, Cambridge, UK}
}

@article{B1983,
author = {Benoit, E.},
year = {1983},
title = {Systemes lents-rapides dans $\mathbb{R}^3$ et leur canards},
journal={Asterisque},
volume={109-110},
pages = {159--191}
}

@article{BCDD1981,
author = {Benoit, E. and Callot, J.L. and Diener, F. and Diener, M.},
title = {Chasse au canard},
journal = {Collectanea Mathematica},
volume = {31-32},
year = {1981}, 
pages = {37--119}
}

@article{BKW2006,
author = {Br{\o}ns, M. and Krupa, M. and Wechselberger, M.},
year = {2006},
title={Mixed Mode Oscillations Due to the Generalized Canard Phenomenon},
journal={Fields Institute Communications},
volume={49}, 
pages = {39--63}
}

@article{BDHL2023,
author={Brown, C. and Derks, G. and van Heijster, P. and Lloyd, D.J.B.},
title={Analysing transitions from a {T}uring
instability to large periodic patterns
in a reaction-diffusion system},
journal={Nonlinearity},
volume={36},
year={2023}, 
pages={6839–-6878}
}

@article{B1978,
author={Busse, F.H.},
title={Nonlinear properties of thermal convection},
journal = {Reports on Progress in Physics}, 
volume={41},
year={1978}, 
pages = {1929--1967}
}

@book{C1981,
author = {Carr, J.},
title = {Applications of Centre Manifold Theory},
publisher = {Applied Mathematical Sciences series v.35, Springer Verlag, NY},
year= {1981}
}

@book{CE1990,
author = {Collet, P. and Eckmann, J.P.},
title = {Instabilities and Fronts in Extended Systems},
year = {1990},
publisher = {Princeton University Press, Princeton, NJ}
}

@article{CH1993,
author = {Cross, M.C. and Hohenberg, P.C.},
title = {Pattern formation outside of equilibrium},
journal = {Reviews Modern Physics},
volume = {65},
year = {1993},
pages = {851--1112}
}

@article{D1984,
author = {Diener, M.J.}, 
title = {The canard unchained or how fast/slow dynamical systems
bifurcate}, 
journal = {Mathematical Intelligencer},
volume = {6}, 
year = {1984}, 
pages = {38--49}
}

@book{DLMF,
author = {NIST},
publisher = {National Institute of Standards and Technology, Gaithersburg, Maryland},
title={Digital Library of Mathematical Functions, {\tt https://dlmf.nist.gov/}},
year={2010},
pages = {Chap. 9}
}

@article{DES1971,
author = {DiPrima, R. and Eckhaus, W. and Segel, L.A.},
title={Nonlinear wave number interaction in near-critical two-dimensional flows},
journal = {Journal of Fluid Mechanics},
year = {1971},
volume = {49},
pages = {705--744}
}

@inbook{D2019,
author = {Doelman, A.},
title  = {Pattern formation in reaction-diffusion systems -- an explicit approach},
editor =  {M. Peletier and R.A. van Santen and E. Steur},
booktitle = {Complexity Science, An Introduction},
year = {2019},
publisher = {World Scientific Press, Hackensack, NJ},
pages = {129--182}
}

@article{DGK2002,
author = {Doelman, A. and Gardner, R.A. and Kaper, T.J.},
title={A Stability Index Analysis of 1-D Patterns of the Gray-Scott Model},
journal={Memoirs of the AMS},
volume={155 (737)},
year = {2002}
}

@article{DR1996,
author = {Dumortier, F. and Roussarie, R.}, 
title = {Canard cycles and center manifolds}, 
journal = {Memoirs of the American Mathematical Society},
volume = {557},
year = {1996}, 
publisher = {Amer. Math. Soc, Providence, Rhode Island} 
}

@book{E1965,
author = {Eckhaus, W.},
title = {Studies in {N}on-{L}inear {S}tability {T}heory},
publisher = {Springer Tracts in Natural Philosophy,
Springer, Berlin},
volume = {6},
year = {1965}
}

@inbook{E1983,
author = {Eckhaus, W.}, 
title = {Relaxation oscillations including a standard chase on {F}rench ducks}, 
booktitle = {Asymptotic {A}nalysis {II}}, 
series = {Lecture Notes in Mathematics, v. 985},
editor = {Verhulst, F.},
publisher = {Springer, Berlin, Heidelberg},
year = {1983}, 
pages = {449--497}
}

@article{E1993,
author = {Eckhaus, W.},
title = {The {G}inzburg-{L}andau manifold is an attractor},
journal = {Journal of Nonlinear Science},
volume = {3},
year = {1993}, 
pages = {329--348}
}

@book{EK2005,
author = {Edelstein-Keshet, L.},
title = {Mathematical Models in Biology}, 
series = {Classics in Applied Mathematics Series, Number 46}, 
publisher = {Society for Industrial and Applied Mathematics, Philadelphia},
year = {2005}
}

@book{EP1998,
author = {Epstein, I.R. and Pojman, J.A.},
title = {Introduction to Nonlinear Chemical Dynamics},
publisher = {Oxford University Press, Oxford, U.K.},
year= {1998}
}

@article{F1979,
author = {Fenichel, N.},
title= {Geometric singular perturbation theory for ordinary differential equations},
journal = {Journal of Differential Equations},
year = {1979},
volume = {31},
pages = {53--98}
}

@article{GM-orig,
author = {Gierer, A. and Meinhardt, H.},
title = {A theory of biological pattern formation}, 
journal={Kybernetik},
year = {1972},
pages = {30-39},
volume = {12}
}

@book{GS1997,
author = {Gray, P. and Scott, S.K.},
title = {Chemical {O}scillations and {I}nstabilities: Nonlinear {C}hemical {K}inetics},
year = {1997},
publisher = {Oxford University Press, Oxford, UK}
}

@book{GH1983,
author={Guckenheimer, J.  and Holmes, P.},
title = {Nonlinear Oscillations, Dynamical Systems, and Bifurcations of Vector Fields},
year = {1983},
publisher = {Springer, New York}
}

@book{HI2011,
author = {Haragus, M. and Iooss, G.},
title = {Local {B}ifurcations, {C}enter {M}anifolds, and {N}ormal {F}orms in {I}nfinite-{D}imensional {D}ynamical {S}ystems},
year = {2011},
publisher = {Springer, London, Dordrecht, Heidelberg, New York}
}

@article{vH1991,
author = {A. van Harten},
title={On the validity of the Ginzburg-Landau equation},
journal = {Journal of Nonlinear Science},
year = {1991},
volume = {1},
pages = {397--422}
}

@book{HPS1977,
author = {M.W. Hirsch and C.C. Pugh and M. Shub},
title = {{Invariant Manifolds}},
publisher = {Springer Berlin, Heidelberg},
series = {Lecture Notes in Mathematics, v.583},
year = {1977}
}

@article{HH1952,
author = {Hodgkin, A.L. and Huxley, A.F.},
year = {1952}, 
title ={A quantitative description of membrane current and its application to conduction and excitation in nerve},
journal = {Journal of Physiology},
volume={117},
pages = {500-–544} 
}

@book{H2006,
author = {R. Hoyle},
title = {Pattern Formation: An Introduction to Methods},
publisher = {Cambridge University Press},
year = {2006}
}

@article{IMD1989,
author = {Iooss, G. and Mielke, A. and Demay, Y.},
title = {{Theory of steady {G}inzburg-{L}andau equation, in hydrodynamic stability problems}},
journal = {European Journal of Mechanics B Fluids},
volume = {8},
year = {1989}, 
pages = {229--268}
}

@article{IP1993,
author = {Iooss, G. and Peroueme, M.C.},
title = {Perturbed homoclinic solutions in reversible 1:1 resonant vector fields},
year = {1993},
volume = {103},
pages = {62--88},
journal = {Journal of Differential Equations}
}

@article{JDKV2026,
author = {Jencks, R. and Doelman, A. and Kaper, T.J. and Vo, T.},
title = { Stable and unstable spatially-periodic canards created in subcritical singular Turing bifurcations in the Brusselator system},
journal = {Journal of Nonlinear Science},
year = {2026},
volume = {36},
number = {55},
pages = {66 pages}
}

@article{JZ2010,
title = {Nonlinear stability of spatially-periodic traveling-wave
solutions of systems of reaction diffusion equations},
author ={Johnson, M.A. and Zumbrun, K.},
journal = {Annales de l Institut Henri Poincaré C: Analyse Non Linéaire},
volume = {28},
pages = {471-483},
year = {2010},
}

@inbook{J1995,
author = {Jones, C.K.R.T.},
title= {Geometric singular perturbation theory},
year = {1995},
booktitle = {Dynamical Systems}, 
publisher = {Springer Lecture Notes in Mathematics}, 
volume ={1609},
editor = {R. Johnson},
pages = {44--118}
}

@article{K1999,
title = {Regular and Irregular Patterns in Semiarid Vegetation},
author = {Klausmeier, C.A.},
journal = {Science},
year = {1999},
volume = {284},
pages = {1826--1828}
}

@article{K2015,
  title={Spatial localization in dissipative systems},
  author={Knobloch, E.},
  journal={Ann. Rev. Cond. Mat. Phys.},
  volume={6},
  number={1},
  pages={325--359},
  year={2015},
  publisher={Annual Reviews}
}

@article{KS2001,
author = {Krupa, M. and Szmolyan, P.},
title = {Extending geometric singular perturbation theory to non-hyperbolic points -- fold and canard points in two dimensions},
journal = {SIAM Journal on Mathematical Analysis},
volume = {33},
year = {2001}, 
pages = {286--314}
}

@article{KW2010,
author = {Krupa, M. and Wechselberger, M.},
title = {Local analysis near a folded saddle-node singularity},
journal = {Journal of Differential Equations},
volume = {248},
year = {2010}, 
pages = {2841--2888}
}

@book{K1984,
author = {Y. Kuramoto},
title = {Chemical Oscillations, Waves, and Turbulence},
publisher = {Springer Series in Synergetics},
volume = {19},
year = {1984}
}

@article{LE1991,
author = {Lengyel, I. and Epstein, I.R.},
year = {1991}, 
title = {Modeling of Turing structures in the chlorite-iodide-malonic acid-starch reaction systems}, 
journal= {Science}, 
volume= {251},
pages = {650-–652}
}

@article{LE1992,
author = {Lengyel, I. and Epstein, I.R.}, 
year = {1992}, 
title ={A chemical approach to designing Turing patterns in reaction-diffusion system},
journal={Proceedings of the National Academy of Sciences, USA},
volume = {89}, 
pages = {3977-–3979}
}

@book{Me1982,
author = {Meinhardt, H.},
title = {Models of Biological Pattern Formation},
year = {1982},
publisher = {Academic Press, New York}
}

@inbook{M2002,
author = {Mielke, A. },
title = {The Ginzburg-Landau equation in its role as modulation equation},
booktitle = {Handbook of Dynamical Systems},
volume = {2},
editor = {B. Fiedler},
publisher = {North Holland (Elsevier)},
pages = {759--834},
year = {2002}
}

@article{MS1995,
author = {Mielke, A. and Schneider, G.},
title = {Attractors for modulation equations on unbounded domains
-Existence and comparison},
journal = {Nonlinearity},
volume = {8},
pages = {743--768},
year = {1995}
}

@book{MKKR1984,
author = {Mishchenko, E.F and Kolesov, Yu.S. and Kolesov, A.Yu. and Rhozov, N.Kh.}, 
title = {Asymptotic Methods in Singularly Perturbed Systems}, 
series = {Monographs in Contemporary Mathematics},
publisher = {Consultants Bureau, Plenum Publishing, NY},
year = {1984}
}

@article{MW2017,
author = {Mitry, J. and Wechselberger, M.},
title = {Folded saddles and faux canards},
journal = {SIAM Journal on Applied Dynamical Systems},
volume={16},
year = {2017},
pages = {546--596}
}

@article{Moehlis,
author = {J. Moehlis}, 
title = {Canards in a surface oxidation reaction},
journal = {Journal of Nonlinear Science},
volume = {12},
year = {2002}, 
pages = {319--345}
}

@book{M1989,
author = {J.D. Murray},
title = {Mathematical Biology},
year = {1989},
publisher = {Springer Verlag, Berlin}
}

@article{Mu1982,
author={J.D. Murray},
title={Parameter space for Turing instability in reaction-diffusion mechanisms: A comparison of models},
year={1982},
journal={Journal of Theoretical Biology},
volume={98},
pages={143-163}
}

@article{NW1969,
author = {Newell, A. and Whitehead, J.},
title = {Finite bandwidth, finite amplitude convection},
journal = {Journal of Fluid Mechanics},
volume ={38},
pages = {279--303},
year = {1969}
}

@book{N2002,
author={Y. Nishiura},
title = {Far-from-Equilibrium Patterns},
publisher = {Translations of Mathematical Monographs, American Mathematical Society, Providence},
year= {2002},
volume = {209}
}

@article{Pearson1993,
author={Pearson, J.E.},
journal={Science},
volume={261},
pages = {189--192},
title={Complex Patterns in a Simple System},
year={1993}
}

@article{Pena2001,
    author = {Pe\~{n}a, B. and Perez-Garcia, C.},
    title = {Stability of {T}uring patterns in the {B}russelator model},
    journal = {Physical Review E},
    volume = {64},
    number = {056213},
    year = {2001}
}

@article{PL1968,
author = {Prigogine, I. and Lefever, R.},
title = {Symmetry breaking instabilities in dissipative systems II},
year = {1968},
journal = {Journal of Chemical Physics},
volume = {48},
pages = {1695--1700}
}

@article{dRDR2016,
  title={Spectra and stability of spatially periodic pulse patterns: Evans function factorization via Riccati transformation},
  author={de Rijk, B. and Doelman, A. and Rademacher, J.},
  journal={SIAM Journal on Mathematical Analysis},
  volume={48},
  number={1},
  pages={61--121},
  year={2016},
  publisher={SIAM}
}

@article{RRW2015,
author= {Roberts, K.-L. and Rubin, J.E. and Wechselberger, M.},
year = {2015},
journal = {SIAM Journal on Applied Dynamical Systems},
volume = {14},
pages = {1808--1844},
title = {Averaging, Folded Singularities, and Torus Canards: Explaining Transitions between Bursting and Spiking in a Coupled Neuron Model}
}

@article{RKW2008,
title = {Canard Induced Mixed-Mode Oscillations in a Medial Entorhinal Cortex Layer II Stellate Cell Model},
author = {Rotstein, H. and Kopell, N. and Wechselberger, M.},
year = {2008},
journal = {SIAM Journal on Applied Dynamical Systems},
volume = {7},
pages ={1582--1611}
}

@article{RKZE2003,
author = {Rotstein, H.G. and Kopell, N. and Zhabotinsky, A.M. and Epstein, I.R.},
title = {Canard phenomenon and localization of oscillations in the {B}elousov-{Z}habotinsky reaction with global feedback},
journal = {Journal of Chemical Physics},
volume = {119},
year = {2003}, 
pages = {8824--8832}
}

@article{S1979,
author = {Schnakenberg, J.},
title = {Simple chemical reaction systems with limit cycle behaviour},
journal = {Journal of Theoretical Biology},
volume = {81},
year = {1991},
pages = {389--400}
}

@article{S1996,
title = {Diffusive stability of spatial periodic solutions
of the Swift-Hohenberg equation},
author = {G. Schneider},
journal = {Communications in Mathematical Physics},
year = {1996},
pages = {679--702},
volume = {178}
}

@article{S1998,
author = {G. Schneider},
title = {Nonlinear diffusive stability of spatially periodic solutions –
Abstract theorem and higher space dimensions},
journal = {Tohoku University Mathematical Publications},
pages = {159–167},
volume = {8},
year = {1998}
}

@book{SU2017,
author = {Schneider, G. and Uecker, H.},
title = {Nonlinear {PDE}s: {A} {D}ynamical {S}ystems {A}pproach},
year = {2017},
series = {Graduate Texts in Mathematics},  
publisher = {American Mathematical Society, Providence, RI},
volume = {182}
}

@book{S1983,
author = {J. Smoller},
title = {Shock Waves and Reaction-Diffusion Equations},
series = {Grundlehren der Mathematische Wissenshaften}, 
publisher = {Springer, Berlin},
year = {1983},
volume = {258}
}

@inbook{ST2001,
    author = {C. Soto-Trevino},
    title = {A geometric method for periodic orbits in singularly-perturbed systems}, 
    booktitle = {Multiple-time-scale Dynamical Systems}, 
    volume = {122},
    series = {IMA Volumes on Mathematics and its Applications},
    pages = {141–-202},
    editor = {Jones, C.K.R.T. and Khibnik, A.},
    publisher = {Springer, New York}, 
    year = {2001}
}

@article{SDHR2013,
author = {S. van der Stelt and A. Doelman and G.M. Hek and J.D.M. Rademacher},
title = {Rise and fall of periodic patterns for a generalized Klausmeier-Gray-Scott model},
journal = {Journal of Nonlinear Science},
year = {2013},
pages = {39--95},
volume = {23}
}

@article{SS1971,
author = {K. Stewartson and J.T. Stuart},
title = {A nonlinear instability theory for a wave system in plane Poiseuille flow},
journal = {Journal of Fluid Mechanics},
year = {1971},
pages = {529--545},
volume = {48}
}

@article{SZJV2018,
author = {Sukhtayev, A. and Zumbrun, K. and Jung, S. and Venkatraman, R.},
title = {Diffusive stability of spatially periodic solutions of the {B}russelator model},
journal = {Communications in Mathematical Physics},
volume = {358},
year = {2018},
pages = {1--43}
}

@article{SW2001,
author={Szmolyan, P. and Wechselberger, M.}, 
title = {Canards in $\mathbb{R}^3$},
journal={Journal of Differential Equations},
volume={177},
year = {2001},
pages = {419--453}
}

@inbook{T1976,
author = {Takens, F.},
title = {Constrained equations: a study of implicit differential equations and their discontinuous solutions}, 
booktitle = {Structural stability, the theory of catastrophes, and applications in the sciences}, 
editor = {Hilton, P.},
publisher = {Lecture Notes in Mathematics, Springer-Verlag, Berlin},
volume={525},
pages = {143-–234},
year={1976}
}

@article{S2003,
author = {van Saarloos, W.},
title = {Front propagation into unstable states},
journal = {Physics Reports},
volume = {386},
year = {2003}, 
pages = {29--222}
}

@article{VDK2025,
author={Vo, T. and Doelman, A. and Kaper, T.J.},
title={Les {C}anards de {T}uring},
journal = {SIAM Journal on Applied Dynamical Systems},
volume = {24},
pages = {2618-2684},
year = {2025}
}

@article{VW2015,
author = {Vo, T. and Wechselberger, M.},
title = {Canards of Folded Saddle-Node Type I},
journal = {SIAM Journal on Mathematical Analysis},
volume ={47},
year = {2015},
pages = {3235--3283}
}

@book{W1997,
author = {Walgraef, D.},
title = {Spatio-temporal Pattern Formation},
year = {1997},
publisher = {Springer-Verlag, New York, Berlin, Heidelberg}
}
\end{document}